\documentclass[11pt,oneside,reqno]{amsart}
\usepackage[latin9]{inputenc}
\usepackage{geometry}
\usepackage{array}
\usepackage{mathrsfs}
\usepackage{bm}
\usepackage{multirow}
\usepackage{amsbsy}
\usepackage{amstext}
\usepackage{amsthm}
\usepackage{amssymb}
\usepackage{stmaryrd}
\usepackage{setspace}
\usepackage[unicode=true,
 bookmarks=false,
 breaklinks=false,pdfborder={0 0 1},backref=section,colorlinks=false]
 {hyperref}

\makeatletter

\providecommand{\tabularnewline}{\\}

\numberwithin{equation}{section}
\numberwithin{figure}{section}
\theoremstyle{plain}
\newtheorem{thm}{\protect\theoremname}[section]
\theoremstyle{remark}
\newtheorem{notation}[thm]{\protect\notationname}
\theoremstyle{remark}
\newtheorem{rem}[thm]{\protect\remarkname}
\theoremstyle{definition}
\newtheorem{defn}[thm]{\protect\definitionname}
\theoremstyle{plain}
\newtheorem{lem}[thm]{\protect\lemmaname}
\theoremstyle{remark}
\newtheorem*{acknowledgement*}{\protect\acknowledgementname}

\usepackage{amsthm}
\usepackage{mathrsfs}
\usepackage[noadjust]{cite}
\usepackage{enumitem}
	\setlist[itemize]{leftmargin=*}
	\setlist[enumerate]{leftmargin=*}
\usepackage{tikz}
\usetikzlibrary[patterns]

\makeatother

\providecommand{\acknowledgementname}{Acknowledgement}
\providecommand{\definitionname}{Definition}
\providecommand{\lemmaname}{Lemma}
\providecommand{\notationname}{Notation}
\providecommand{\remarkname}{Remark}
\providecommand{\theoremname}{Theorem}

\begin{document}
\title[Metastable Hierarchy in Lattice Models and Kawasaki Dynamics]{Metastable Hierarchy in Abstract Low-Temperature Lattice Models:
An Application to Kawasaki Dynamics for Ising Lattice Gas with Macroscopic
Number of Particles}
\author{Seonwoo Kim}
\address{June E Huh Center for Mathematical Challenges, Korea Institute for
Advanced Study, Republic of Korea.}
\email{seonwookim@kias.re.kr}
\begin{abstract}
This article is divided into two parts. In the first part, we study
the hierarchical phenomenon of metastability in low-temperature lattice
models in the most general setting. Given an abstract dynamical system
governed by a Hamiltonian function, we prove that there exists a hierarchical
decomposition of the collection of stable plateaux in the system into
multiple $\mathfrak{m}$ levels, such that at each level there exist
tunneling metastable transitions between the stable plateaux, which
can be characterized by convergence to an explicit simple Markov chain as the
inverse temperature $\beta$ tends to infinity. In the second part,
as an application, we characterize the $3$-level metastable hierarchy
in Kawasaki dynamics for Ising lattice gas with macroscopic number
of particles. We prove that the ground states in this model are those
in which the particles line up and form a one-dimensional strip, and
identify the full structure relevant to the tunneling transitions
between these ground states. In particular, the results differ from
the previous work \cite{BL Kawasaki} in that the particles in the
ground states are likely to form a strip rather than a square droplet.
The main tool is the resolvent approach to metastability, recently
developed in \cite{LMS resolvent}. Along with the analysis, we present
a theorem on the sharp asymptotics of the exit distribution from cycles,
which to the author's knowledge is not known in the community and
therefore may be of independent interest.
\end{abstract}

\maketitle
\tableofcontents{}

\section{\label{sec1}Introduction}

Metastability occurs in a wide class of low-temperature dynamical
systems that possess two or more locally stable states, including
countless real-world examples such as the supercooling of water or
the persistence of stock prices at high levels. In the context of
statistical physics, this phenomenon can be understood as a first-order
phase transition with respect to intrinsic physical parameters such
as magnetization or particle density. It occurs in a large number
of concrete models falling in the range of small random perturbations
of dynamical systems \cite{BEGK,FW,LanLeeSeo Meta-1,LanLeeSeo Meta-2,LanMarianiSeo},
condensing interacting particle systems \cite{BL ZRP,BDG,Kim RIP-2nd,Kim RIP-3rd,LMS critical,Seo NRZRP},
ferromagnetic spin systems at low temperature \cite{BAC,BdHN,KS IsingPotts-growing,KS IsingPotts-3D,MNOS,NS Ising1},
etc.

Inspired by the Freidlin--Wentzell theory \cite{FW} on large deviation
properties of perturbed dynamical systems, Cassandro, Galves, Olivieri
and Vares \cite{CGOV} successfully developed this theory in discrete
ferromagnetic systems with applications to the Curie--Weiss model
and the contact process on $\mathbb{Z}$. Since then, there has been
a vast literature on the application of this so-called \emph{pathwise
approach} to metastability, including some notable works \cite{BAC,MNOS,NZ,NS Ising1,OS rev}.
See the monograph \cite{OV} for an extensive history, overview and
literature on the subject.

In the early 2000s, Bovier, Eckhoff, Gayrard and Klein \cite{BEGK}
made a remarkable breakthrough by interpreting the well-known \emph{potential
theory} in the language of metastability. They discovered that a sharp
asymptotic of a potential-theoretic object called the \emph{capacity}
is the key to give a much more refined analysis of the metastable
behavior, in particular the exact prefactor of the mean metastable
transition time, and thus the full \emph{Eyring--Kramers formula}.
We refer the readers to a recent monograph \cite{BdH} for detailed
aspects and complete literature on the potential-theoretic approach
to metastability.

Recently, Landim, Marcondes and Seo \cite{LMS resolvent} introduced
a completely new methodology, namely the \emph{resolvent approach},
to study metastability occurring in systems with a complicated structure
of multiple stable states. Let $L_{N}$ be the infinitesimal generator
of a metastable dynamics, where $N$ is the control parameter of the
system. The idea is to compare the accelerated microscopic resolvent
equation $(\lambda-\theta_{N}L_{N})F_{N}=G_{N}$ and the macroscopic
resolvent equation $(\lambda-\mathfrak{L})f=g$, where $\lambda$
is a positive constant, $g$ is an arbitrary function, $G_{N}$ is
the lifted function of $g$, and $\theta_{N}$ is a suitably chosen
time scale. They then proved that the metastability of the $L_{N}$-dynamics
characterized by $(\theta_{N},\,\mathfrak{L})$ is equivalent to the
property that the microscopic solution $F_{N}$ is close to the macroscopic
solution $f$ in each metastable valley as $N\to\infty$; see Sections
\ref{sec4.3} and \ref{sec5.4} for detailed formulations in the context
of this article. In particular, the equivalence of this resolvent
condition (cf. \eqref{eq:res cond} and \eqref{eq:res cond-h} below)
to the occurrence of metastability makes the resolvent characterization
to be the most effective tool so far. Then, Landim, Lee and Seo \cite{LanLeeSeo Meta-1,LanLeeSeo Meta-2}
successfully applied the resolvent approach to characterize the metastability
occurring in the overdamped Langevin dynamics with general potentials,
in particular to characterize the full hierarchical structure depending
on the different depths of the valleys.

In this article, we apply the resolvent approach to a general class
of low-temperature ferromagnetic lattice models to identify a hierarchical
structure of tunneling transitions that occur in the complicated energy
landscape of metastable states. To fix ideas, we consider the Metropolis
dynamics (cf. \eqref{eq:rbeta def}) in the state space such that
the dynamics jumps with rate $1$ to equal or lower energy configurations
and with $\beta$-exponentially small rate to higher energy configurations,
where $\beta$ denotes the inverse temperature. Extending the results
\cite{LanLeeSeo Meta-1,LanLeeSeo Meta-2} from diffusion to discrete
systems, we characterize each \emph{stable plateau} (see Definition
\ref{def:stab plat}) as a metastable element in the system and quantify
the metastable transitions at each level with a limiting Markov chain
between these stable plateaux, by accelerating the dynamics and proving
convergence to a limit object. The exact formulation of this convergence
is given in Section \ref{sec2.2}.

More specifically, denote temporarily by $\mathscr{P}^{1}$ the collection
of stable plateaux in the system. For each stable plateau, its initial
depth is defined as the minimum energy barrier to make a transition
to another one. Then, the initial metastable tunneling transitions
occur between those with the minimum initial depth. In turn, stable
plateaux in each irreducible class subject to the first-level transitions
form a new element in the second level, so that next we can characterize
the second-level tunneling transitions in the new collection $\mathscr{P}^{2}$
in a similar way. We can iterate this procedure until we obtain a
unique irreducible class of transitions at the terminal level, say
level $\mathfrak{m}$. Then, the sequence $\mathscr{P}^{1},\,\mathscr{P}^{2},\,\dots,\,\mathscr{P}^{\mathfrak{m}}$
contains complete information about the metastable hierarchy in the
system. See Section \ref{sec2.1} for the exact formulation of this
scheme.

In particular, compared to the previous work \cite{BL rev-MP} where the authors also studied the hierarchical decomposition of Markov chains, the work presented in this article provides a much more explicit general framework in that here we can derive the exact limiting Markov chain at each level by thoroughly analyzing the saddle structure of the energy landscape. More precisely, the limiting Markov chain is obtained by ignoring all the negligible jump rates of the system; see Definition \ref{def:gen const} for the details.

We chose the Metropolis dynamics as our Markov chain in the abstract
system to avoid further technical difficulties irrelevant to the main
results of the article. Nevertheless, we believe that the main ideas
presented here are robust enough to be directly applicable to a broader
class of stochastic dynamics, such as non-reversible dynamics in the
category of Freidlin--Wentzell-type Markov chains (cf. \cite[Section 6]{OV}
or \cite{CNSoh}).

As an application, we study a concrete and fundamental model that
fits into this framework, the \emph{Kawasaki dynamics} for the Ising
gas model with macroscopic number of particles in a fixed lattice
in the zero external field setup. The Kawasaki dynamics for the Ising
lattice gas was previously studied in \cite{BL Kawasaki}, in the
case where the total number of particles $\mathscr{N}$ is bounded
by $\frac{L^{2}}{4}$, where $L$ is the side length of the lattice.
Their analysis started from the fact that in this case, the ground
states are those in which the particles cluster together to form a
square droplet. However, if $\mathscr{N}>\frac{L^{2}}{4}$ then this
characteristic breaks down; here, in the ground state, the particles
tend to line up in one direction, forming a one-dimensional strip
droplet. Because of this fact, the geometry of the saddle structure
changes dramatically, and we need a completely new analysis to study
the metastable transitions here. In short, we identify a \emph{phase
transition} in the shape of the ground states at a threshold $\mathscr{N}^{*}:=\frac{L^{2}}{4}$
(see also Remark \ref{rem:phase trans particles}).

We emphasize here that the Kawasaki dynamics for Ising particles with
zero external field is difficult to treat with the previous methods,
and that it is a suitable model to apply the general methodology developed
in this article for the following reasons. First, as shown in \cite{KS IsingPotts-growing,KS IsingPotts-2D,KS IsingPotts-3D},
in interacting spin systems where the spins play the same role (i.e.,
external field $h=0$), the energy landscape and in particular the
saddle structure tend to be much more complicated than in their asymmetric
counterparts (i.e., $h\ne0$ as in \cite{BAC,NS Ising1}). The same
phenomenon occurs here; between the stable configurations (which are
monochromatic configurations, see Theorem \ref{thm:Ham S}) there
is a huge saddle plateau structure that also contains a large number
of shallow valleys. Second, the saddle structure is highly complicated
even in the bulk part of the metastable transitions, which was not
the case in the Glauber dynamics \cite{BGN Potts-zero,KS IsingPotts-2D},
where the bulk transitions were approximately one-dimensional. This
feature makes it implausible to decompose the energy landscape and
treat each part independently, as was done in the previous works.

One can also study the deepest metastable transitions between the
stable configurations in a new setup where the lattice size $L$ also
tends to infinity along with the inverse temperature $\beta$. When
the number of particles $\mathscr{N}$ is smaller than $\frac{L^{2}}{4}$,
Gois and Landim \cite{GoisLan} proved that the tunneling dynamics
between the stable states can be approximated as a two-dimensional
Brownian motion, mainly due to the fact that each stable configuration
consists of a square droplet that can move in the plane. However,
in the regime of this article ($\mathscr{N}>\frac{L^{2}}{4}$), the
particles in the stable configurations now form a one-dimensional
strip rather than a square. This difference suggests that in our new
regime, the limiting tunneling dynamics would be a one-dimensional
Brownian motion, with the position of the strip moving along the line.
This serves as a potential project for the near future.

Among the key ideas to be presented in the following, a notable part
is a sharp characterization of the exit distributions from cycles
(cf. Definition \ref{def:cycle}). According to the classical Freidlin--Wentzell
theory \cite{FW}, it was well known since the 1970s that a typical
(with probability tending to $1$) exit from a cycle occurs at its
boundary with minimum energy. However, it was impossible to obtain
an exact asymptotics of the exit distribution on the energy minimizers
using the pathwise approach. In Section \ref{sec2.3}, we provide
a sharp estimate of the exit distributions from cycles, which seems
to be the first result in this direction to the author's knowledge
and may deserve an independent interest in the community. The key
idea is to apply the $H^{1}$-approximation approach, recently developed
in \cite{KS IsingPotts-2D,LanSeo RW-pot-field}. See Sections \ref{sec2.3}
and \ref{sec9} for the exact formulation and its proof.

The remainder of this article is organized as follows. In Section
\ref{sec2}, we study the general framework for establishing the hierarchical
structure of metastable transitions in low-temperature lattice models,
along with a key theorem on the exit distribution from cycles in Section
\ref{sec2.3}. Then, in Section \ref{sec3}, we apply this framework
to the specific model, the Kawasaki dynamics for the Ising lattice
gas system with macroscopic number of particles. In Sections \ref{sec4}
and \ref{sec5}, we provide a proof of the general theory developed
in Section \ref{sec2}. In Sections \ref{sec6}, \ref{sec7} and \ref{sec8},
we perform a detailed analysis of the complicated energy landscape
of the Kawasaki dynamics, and then apply the general methodology to
prove the main results in Section \ref{sec3}. In Section \ref{sec9},
we prove the theorem on the exit distribution from cycles. In Appendix
\ref{appenA}, we give a short proof of a technical lemma used in
Section \ref{sec7}.

\section{\label{sec2}Metastable Hierarchy in Low-Temperature Lattice Models}

Consider a finite state space $\Omega$ which is given an (undirected) edge
structure, such that we write $\eta\sim\xi$ (and $\xi\sim\eta$)
if and only if $\{\eta,\,\xi\}$ is an edge. We assume that the resulting
graph $\Omega$ is connected.

Suppose that a \emph{Hamiltonian} (or \emph{energy}) function $\mathbb{H}:\,\Omega\to\mathbb{R}$
is given to the system, such that the corresponding \emph{Gibbs measure}
$\mu_{\beta}$ on $\Omega$ is defined as
\begin{equation}
\mu_{\beta}(\eta):=Z_{\beta}^{-1}e^{-\beta\mathbb{H}(\eta)}\quad\text{for}\quad\eta\in\Omega,\label{eq:Gibbs def}
\end{equation}
where $\beta>0$ is the inverse temperature of the system and $Z_{\beta}:=\sum_{\eta\in\Omega}e^{-\beta\mathbb{H}(\eta)}$
is the normalizing constant such that $\mu_{\beta}$ is a probability
measure. We are interested in the regime of zero-temperature limit:
$\beta\to\infty$.
\begin{notation}
\label{nota:bot bdry def}For $\mathcal{A}\subseteq\Omega$, define
the \emph{bottom} of $\mathcal{A}$ as $\mathcal{F}(\mathcal{A}):=\{\xi\in\mathcal{A}:\,\mathbb{H}(\xi)=\min_{\mathcal{A}}\mathbb{H}\}$
and the (outer) \emph{boundary} of $\mathcal{A}$ as $\partial\mathcal{A}:=\{\eta\in\Omega\setminus\mathcal{A}:\,\eta\sim\xi\quad\text{for some}\quad\xi\in\mathcal{A}\}$.
Moreover, define $\partial^{\star}\mathcal{A}:=\mathcal{F}(\partial\mathcal{A})$.
\end{notation}

Consider the \emph{Metropolis dynamics} $\{\eta_{\beta}(t)\}_{t\ge0}$
in $\Omega$ whose transition rate function $r_{\beta}:\,\Omega\times\Omega\to[0,\,\infty)$
is represented as
\begin{equation}
r_{\beta}(\eta,\,\xi)=\begin{cases}
e^{-\beta\max\{\mathbb{H}(\xi)-\mathbb{H}(\eta),\,0\}} & \text{if}\quad\eta\sim\xi,\\
0 & \text{otherwise}.
\end{cases}\label{eq:rbeta def}
\end{equation}
Denote by $L_{\beta}$ the corresponding infinitesimal generator.
Since $\Omega$ is connected, $\{\eta_{\beta}(t)\}_{t\ge0}$ is irreducible
in $\Omega$. One can easily notice that the Metropolis dynamics is
reversible with respect to $\mu_{\beta}$:
\begin{equation}
\mu_{\beta}(\eta)r_{\beta}(\eta,\,\xi)=\mu_{\beta}(\xi)r_{\beta}(\xi,\,\eta)=\begin{cases}
Z_{\beta}^{-1}e^{-\beta\max\{\mathbb{H}(\eta),\,\mathbb{H}(\xi)\}} & \text{if}\quad\eta\sim\xi,\\
0 & \text{otherwise}.
\end{cases}\label{eq:det bal}
\end{equation}
We denote by $\mathcal{S}:=\mathcal{F}(\Omega)$ the minimizer of
the Hamiltonian (cf. Notation \ref{nota:bot bdry def}). The elements
of $\mathcal{S}$ are called \emph{ground states}. By \eqref{eq:Gibbs def},
it is straightforward that $\lim_{\beta\to\infty}\mu_{\beta}(\mathcal{S})=1$.

A sequence $\omega=(\omega_{n})_{n=0}^{N}$ of configurations is called
a \emph{path} from $\omega_{0}=\eta$ to $\omega_{N}=\xi$ if $\omega_{n}\sim\omega_{n+1}$
for all $n\in[0,\,N-1]$.\footnote{In this article, $[\alpha,\,\alpha']$ denotes the collection of integers
from $\alpha$ to $\alpha'$.} In this case, we write $\omega:\,\eta\to\xi$. Moreover, writing
$\omega:\,\mathcal{A}\to\mathcal{B}$ for disjoint subsets $\mathcal{A},\,\mathcal{B}$
implies that $\omega:\,\eta\to\xi$ for some $\eta\in\mathcal{A}$
and $\xi\in\mathcal{B}$. For each path $\omega$, define the \emph{height}
of $\omega$ as
\begin{equation}
\Phi_{\omega}:=\max_{n\in[0,\,N]}\mathbb{H}(\omega_{n}).\label{eq:Phi-omega def}
\end{equation}
Then, for $\eta,\,\xi\in\Omega$, define the \emph{communication height}
or \emph{energy barrier} between $\eta$ and $\xi$ as
\begin{equation}
\Phi(\eta,\,\xi):=\min_{\omega:\,\eta\to\xi}\Phi_{\omega}.\label{eq:comm height def}
\end{equation}
More generally, for two disjoint subsets $\mathcal{A},\,\mathcal{B}\subseteq\Omega$,
define $\Phi(\mathcal{A},\,\mathcal{B}):=\min_{\omega:\,\mathcal{A}\to\mathcal{B}}\Phi_{\omega}$.
By concatenating the paths, it holds that
\begin{equation}
\Phi(\eta,\,\xi)\le\max\{\Phi(\eta,\,\zeta),\,\Phi(\zeta,\,\xi)\}\quad\text{for all}\quad\eta,\,\xi,\,\zeta\in\Omega.\label{eq:Phi max ineq}
\end{equation}
Now, define
\begin{equation}
\overline{\Phi}:=\max_{\bm{s},\,\bm{s'}\in\mathcal{S}}\Phi(\bm{s},\,\bm{s'}).\label{eq:Phi-bar def}
\end{equation}
In other words, $\overline{\Phi}$ is the minimal energy level subject
to which we observe all transitions between the ground states in $\mathcal{S}$.
Thus, we naturally define
\begin{equation}
\overline{\Omega}:=\{\eta\in\Omega:\,\Phi(\mathcal{S},\,\eta)\le\overline{\Phi}\},\label{eq:Omega-bar def}
\end{equation}
which is the collection of all configurations reachable from the ground
states by paths of height at most $\overline{\Phi}$. Hereafter, our
analysis is focused on the collection $\overline{\Omega}$, which
is a connected subgraph of $\Omega$.
\begin{rem}
\label{rem:Omega Omega-bar}In the zero-temperature limit, the energy
landscape near the ground states captures all the essential features
of the metastability phenomenon. Thus, instead of trying to characterize
all the locally stable and saddle structures in the whole system,
we focus on the essential subset $\overline{\Omega}$. Nevertheless,
we note here that the following analyses below are fully valid even
if $\overline{\Omega}$ is replaced by the full space $\Omega$; such
attempts have been carried out in \cite{BL rev-MP}.
\end{rem}

Now, we define two crucial notions explored in this article; namely \emph{stable plateaux} and \emph{cycles}. We call that a set $\mathcal{A}$ is \emph{connected} if for any $\eta,\,\xi\in\mathcal{A}$, there exists a path $\eta=\eta_0,\,\eta_1,\,\dots,\,\eta_N=\xi$ inside $\mathcal{A}$, i.e., $\eta_n\in\mathcal{A}$ for all $n\in[0,\,N]$.
\begin{defn}[Stable plateau]
\label{def:stab plat}A nonempty connected set $\mathcal{P}$ is
called a \emph{stable plateau} if the following two statements hold.
\begin{itemize}
\item For all $\eta,\,\xi\in\mathcal{P}$, we have $\mathbb{H}(\eta)=\mathbb{H}(\xi)$;
we denote by $\mathbb{H}(\mathcal{P})$ the common energy value.\footnote{We adopt this notation for any set with common energy.}
\item It holds that $\mathbb{H}(\zeta)>\mathbb{H}(\mathcal{P})$ for all
$\zeta\in\partial\mathcal{P}$.
\end{itemize}
\end{defn}

We denote by $\nu_{0}$ the number of stable plateaux in $\overline{\Omega}$
and by $\mathscr{P}^{1}$ their collection:
\begin{equation}
\mathscr{P}^{1}:=\{\mathcal{P}_{i}^{1}:\,i\in[1,\,\nu_{0}]\}.\label{eq:P1 def}
\end{equation}
It is clear that each ground state in $\mathcal{S}$ constitutes a
stable plateau in $\overline{\Omega}$, i.e., for every $\bm{s}\in\mathcal{S}$
there exists $\mathcal{P}\in\mathscr{P}^{1}$ such that $\bm{s}\in\mathcal{P}$. Note that here it is possible for a stable plateau to contain two or more ground states.
\begin{defn}[Cycle]
\label{def:cycle}A nonempty connected set $\mathcal{C}$ is a (nontrivial)
\emph{cycle} (cf. \cite[Definition 6.5]{OV}) if
\begin{equation}
\max_{\eta\in\mathcal{C}}\mathbb{H}(\eta)<\min_{\xi\in\partial\mathcal{C}}\mathbb{H}(\xi).\label{eq:cyc def}
\end{equation}
Each stable plateau is clearly a cycle. Then, the \emph{depth} of
cycle $\mathcal{C}$ is defined as
\begin{equation}
\Gamma^{\mathcal{C}}:=\min_{\partial\mathcal{C}}\mathbb{H}-\min_{\mathcal{C}}\mathbb{H}>0.\label{eq:depth-def}
\end{equation}
\end{defn}

Each cycle, in particular each stable plateau, is \emph{metastable}
in the sense that it takes an exponentially long time to exit (cf.
\cite[Theorem 6.23]{OV}) and that the exit time is asymptotically
exponentially distributed (cf. \cite[Theorem 6.30]{OV}). On the other
hand, any configuration that is not contained in a stable plateau
is not metastable, since its holding rate is of order $1$. Thus,
from now on we assume that $\nu_{0}\ge2$, consider $\mathscr{P}^{1}$
as the collection of all metastable elements and study the tunneling
transitions between the elements of $\mathscr{P}^{1}$.

It is easy to verify that for any two cycles $\mathcal{C},\,\mathcal{C}'$
with $\mathcal{C}\cap\mathcal{C}'\ne\emptyset$, it holds that $\mathcal{C}\subseteq\mathcal{C}'$
or $\mathcal{C}'\subseteq\mathcal{C}$ (see \cite[Proposition 6.8]{OV}
for a proof). Moreover, the following lemma holds.
\begin{lem}
\label{lem:cyc bot P}For every cycle $\mathcal{C}$, its bottom $\mathcal{F}(\mathcal{C})$
is a union of stable plateaux.
\end{lem}

\begin{proof}
Decompose $\mathcal{F}(\mathcal{C})$ into connected components as
$\mathcal{F}(\mathcal{C})=\mathcal{A}_{1}\cup\cdots\cup\mathcal{A}_{N}$
and consider $\partial\mathcal{A}_{n}$ for each $n\in[1,\,N]$. We
claim that $\mathbb{H}(\xi)>\mathbb{H}(\mathcal{A}_{n})$ for all
$\xi\in\partial\mathcal{A}_{n}$ which concludes the proof of the
lemma. Indeed, if $\xi\notin\mathcal{C}$ such that $\xi\in\partial\mathcal{C}$,
then $\mathbb{H}(\xi)>\mathbb{H}(\mathcal{A}_{n})$ by \eqref{eq:cyc def}.
If $\xi\in\mathcal{C}$, then $\mathbb{H}(\xi)>\mathbb{H}(\mathcal{A}_{n})$
since $\mathcal{A}_{n}\subseteq\mathcal{F}(\mathcal{C})$ and $\xi\in\mathcal{C}\setminus\mathcal{F}(\mathcal{C})$.
\end{proof}
Now, we present a general framework of constructing Markov jumps between
cycles.
\begin{defn}[Construction of Markovian jumps between cycles]
\label{def:gen const}Suppose that we are given a pair $(\mathscr{C},\,\Gamma^{\star})$
where $\mathscr{C}$ is a collection of disjoint cycles in $\overline{\Omega}$
with $|\mathscr{C}|\ge2$ and $\Gamma^{\star}$ is a positive real
number. Define
\begin{equation}
\mathscr{P}^{\mathscr{C}}:=\{\mathcal{F}(\mathcal{C}):\,\mathcal{C}\in\mathscr{C}\}.\label{eq:PC def}
\end{equation}
By Lemma \ref{lem:cyc bot P}, $\mathscr{P}^{\mathscr{C}}$ is a collection
of disjoint unions of stable plateaux in $\overline{\Omega}$. Define
\begin{equation}
\mathscr{C}^{\star}:=\{\mathcal{C}\in\mathscr{C}:\,\Gamma^{\mathcal{C}}\ge\Gamma^{\star}\}\quad\text{and}\quad\mathscr{C}^{\sharp}:=\{\mathcal{C}\in\mathscr{C}:\,\Gamma^{\mathcal{C}}<\Gamma^{\star}\}\quad\text{such that}\quad\mathscr{C}=\mathscr{C}^{\star}\cup\mathscr{C}^{\sharp}.\label{eq:C-star C-sharp def}
\end{equation}
In words, $\mathscr{C}^{\star}$ consists of the cycles in $\mathscr{C}$
with depth at least $\Gamma^{\star}$. Accordingly, define
\begin{equation}
\mathscr{P}^{\mathscr{C}^{\star}}:=\{\mathcal{F}(\mathcal{C}):\,\mathcal{C}\in\mathscr{C}^{\star}\}\quad\text{and}\quad\mathscr{P}^{\mathscr{C}^{\sharp}}:=\{\mathcal{F}(\mathcal{C}):\,\mathcal{C}\in\mathscr{C}^{\sharp}\}\quad\text{such that}\quad\mathscr{P}^{\mathscr{C}}=\mathscr{P}^{\mathscr{C}^{\star}}\cup\mathscr{P}^{\mathscr{C}^{\sharp}}.\label{eq:P-star P-sharp def}
\end{equation}
In addition, define $\Delta^{\mathscr{C}}:=\overline{\Omega}\setminus\bigcup_{\mathcal{C}\in\mathscr{C}}\mathcal{C}$.

The \emph{contracted graph} $\Omega^{\mathscr{C}}$ is obtained from
$\overline{\Omega}$ by contracting the elements in each cycle $\mathcal{C}\in\mathscr{C}$
into single element $\mathcal{F}(\mathcal{C})\in\mathscr{P}^{\mathscr{C}}$,
such that $\Omega^{\mathscr{C}}=\Delta^{\mathscr{C}}\cup\mathscr{P}^{\mathscr{C}}$.
Then, the induced Markov chain $\{\mathfrak{X}^{\mathscr{C}}(t)\}_{t\ge0}$
in $\Omega^{\mathscr{C}}$ with transition rate $\mathfrak{R}^{\mathscr{C}}(\cdot,\,\cdot)$
is defined as follows:
\begin{equation}
\begin{cases}
\mathfrak{R}^{\mathscr{C}}(\eta,\,\xi):=1 & \text{if}\quad\eta\sim\xi\quad\text{and}\quad\mathbb{H}(\xi)\le\mathbb{H}(\eta),\\
\mathfrak{R}^{\mathscr{C}}(\eta,\,\mathcal{F}(\mathcal{C})):=|\{\zeta\in\mathcal{C}:\,\eta\sim\zeta\}| & \text{if}\quad\eta\in\partial\mathcal{C},\\
\mathfrak{R}^{\mathscr{C}}(\mathcal{F}(\mathcal{C}),\,\eta):=|\mathcal{F}(\mathcal{C})|^{-1}\cdot|\{\zeta\in\mathcal{C}:\,\eta\sim\zeta\}| & \text{if}\quad\Gamma^{\mathcal{C}}\le\Gamma^{\star}\quad\text{and}\quad\eta\in\partial^{\star}\mathcal{C},
\end{cases}\label{eq:RC def}
\end{equation}
and $\mathfrak{R}^{\mathscr{C}}(\cdot,\,\cdot):=0$ otherwise. Note
that $\{\mathfrak{X}^{\mathscr{C}}(t)\}_{t\ge0}$ is not necessarily
irreducible, since every $\mathcal{F}(\mathcal{C})\in\mathscr{P}^{\mathscr{C}}$
with $\Gamma^{\mathcal{C}}>\Gamma^{\star}$ is an absorbing state.
Thus, $\{\mathfrak{X}^{\mathscr{C}}(t)\}_{t\ge0}$ encodes the asymptotic
jumps between the configurations and cycles with depth at most $\Gamma^{\star}$.
By \eqref{eq:C-star C-sharp def} and \eqref{eq:RC def}, among the
cycles in $\mathscr{C}^{\star}$ only those with depth exactly $\Gamma^{\star}$
have positive rates with respect to $\mathfrak{R}^{\mathscr{C}}(\cdot,\,\cdot)$.

Then, consider the \emph{trace} Markov chain $\{\mathfrak{X}^{\mathscr{C}^{\star}}(t)\}_{t\ge0}$
in $\mathscr{P}^{\mathscr{C}^{\star}}$ whose transition rate function
$\mathfrak{R}^{\mathscr{C}^{\star}}(\cdot,\,\cdot)$ is defined as\footnote{In this article, $\mathcal{T}_{\mathcal{A}}$ denotes the (random)
hitting time of set $\mathcal{A}$.}
\begin{equation}
\mathfrak{R}^{\mathscr{C}^{\star}}(\mathcal{F}(\mathcal{C}),\,\mathcal{F}(\mathcal{C}')):=\sum_{\eta\in\Delta^{\mathscr{C}}}\mathfrak{R}^{\mathscr{C}}(\mathcal{F}(\mathcal{C}),\,\eta)\cdot{\bf P}_{\eta}^{\mathscr{C}}[\mathcal{T}_{\mathcal{F}(\mathcal{C}')}=\mathcal{T}_{\mathscr{P}^{\mathscr{C}^{\star}}}],\label{eq:R-Cstar def}
\end{equation}
where ${\bf P}_{\eta}^{\mathscr{C}}$ denotes the law of $\{\mathfrak{X}^{\mathscr{C}}(t)\}_{t\ge0}$
starting from $\eta\in\Omega^{\mathscr{C}}$. Refer to \cite[Section 6.1]{BL TM}
for the precise definition of the trace process.
\end{defn}

\subsection{\label{sec2.1}Hierarchical structure of stable plateaux}

\subsubsection*{Initial level}

Recall from \eqref{eq:P1 def} that $\mathscr{P}^{1}=\{\mathcal{P}_{1}^{1},\,\dots,\,\mathcal{P}_{\nu_{0}}^{1}\}$
is the collection of all $\nu_{0}\ge2$ stable plateaux in $\overline{\Omega}$.
For each $\mathcal{P}_{i}^{1}\in\mathscr{P}^{1}$, define the \emph{initial
depth} as
\begin{equation}
\Gamma_{i}^{1}:=\Phi(\mathcal{P}_{i}^{1},\,\breve{\mathcal{P}}_{i}^{1})-\mathbb{H}(\mathcal{P}_{i}^{1})>0,\quad\text{where}\quad\breve{\mathcal{P}}_{i}^{1}:=\bigcup_{j\in[1,\,\nu_{0}]:\,j\ne i}\mathcal{P}_{j}^{1}.\label{eq:Gammai1 def}
\end{equation}
Accordingly, define
\begin{equation}
\Gamma^{\star,1}:=\min_{i\in[1,\,\nu_{0}]}\Gamma_{i}^{1}>0.\label{eq:Gamma-star1 def}
\end{equation}
Then, define
\begin{equation}
\mathcal{V}_{i}^{1}:=\{\eta\in\Omega:\,\Phi(\mathcal{P}_{i}^{1},\,\eta)-\mathbb{H}(\mathcal{P}_{i}^{1})<\Gamma_{i}^{1}\}.\label{eq:Vi1 def}
\end{equation}

\begin{lem}
\label{lem:Vi1 cycle}Collections $\mathcal{V}_{i}^{1}$ for $i\in[1,\,\nu_{0}]$
are disjoint cycles in $\overline{\Omega}$ with bottom $\mathcal{P}_{i}^{1}$
and depth $\Gamma_{i}^{1}$.
\end{lem}

\begin{proof}
By \eqref{eq:Omega-bar def} and \eqref{eq:Gammai1 def} it is clear
that $\mathbb{H}(\mathcal{P}_{i}^{1})+\Gamma_{i}^{1}\le\overline{\Phi}$,
thus $\mathcal{V}_{i}^{1}$ is a subset of $\overline{\Omega}$.\footnote{The same reasoning applies to all other collections in the remainder
as well; thus, we regard them to be subsets of $\overline{\Omega}$
without further explanation.} To prove the disjointness, suppose the contrary that there exists
$\eta\in\mathcal{V}_{i}^{1}\cap\mathcal{V}_{j}^{1}$ for $i\ne j$.
By \eqref{eq:Phi max ineq},
\[
\Phi(\mathcal{P}_{i}^{1},\,\mathcal{P}_{j}^{1})\le\max\{\Phi(\mathcal{P}_{i}^{1},\,\eta),\,\Phi(\mathcal{P}_{j}^{1},\,\eta)\}<\max\{\mathbb{H}(\mathcal{P}_{i}^{1})+\Gamma_{i}^{1},\,\mathbb{H}(\mathcal{P}_{j}^{1})+\Gamma_{j}^{1}\}.
\]
On the other hand, by \eqref{eq:Gammai1 def}, $\mathbb{H}(\mathcal{P}_{i}^{1})+\Gamma_{i}^{1}=\Phi(\mathcal{P}_{i}^{1},\,\breve{\mathcal{P}}_{i}^{1})\le\Phi(\mathcal{P}_{i}^{1},\,\mathcal{P}_{j}^{1})$
and similarly $\mathbb{H}(\mathcal{P}_{j}^{1})+\Gamma_{j}^{1}\le\Phi(\mathcal{P}_{j}^{1},\,\mathcal{P}_{i}^{1})$,
which contradict the displayed inequality. Thus, the collections $\mathcal{V}_{i}^{1}$
for $i\in[1,\,\nu_{0}]$ are disjoint.

By \eqref{eq:Vi1 def} and the fact that $\mathcal{P}_{i}^{1}$ is
connected, it follows immediately that $\mathcal{V}_{i}^{1}$ is also
connected. Moreover, by the definition of $\mathcal{V}_{i}^{1}$,
it is clear that $\mathbb{H}(\eta)<\mathbb{H}(\mathcal{P}_{i}^{1})+\Gamma_{i}^{1}$
for all $\eta\in\mathcal{V}_{i}^{1}$ and $\mathbb{H}(\zeta)\ge\mathbb{H}(\mathcal{P}_{i}^{1})+\Gamma_{i}^{1}$
for all $\zeta\in\partial\mathcal{V}_{i}^{1}$. These facts prove
that $\mathcal{V}_{i}^{1}$ is a cycle.

Since $\mathcal{F}(\mathcal{V}_{i}^{1})$ is a union of stable plateaux
by Lemma \ref{lem:cyc bot P} and $\mathcal{V}_{i}^{1}\cap\mathcal{P}_{j}^{1}=\emptyset$
for all $j\ne i$ by the disjointness, we obtain that $\mathcal{F}(\mathcal{V}_{i}^{1})=\mathcal{P}_{i}^{1}$.
Finally, there exists $\xi\in\partial\mathcal{V}_{i}^{1}$ such that
$\mathbb{H}(\xi)=\mathbb{H}(\mathcal{P}_{i}^{1})+\Gamma_{i}^{1}$
due to the existence of a path $\mathcal{P}_{i}^{1}\to\breve{\mathcal{P}}_{i}^{1}$
of height $\mathbb{H}(\mathcal{P}_{i}^{1})+\Gamma_{i}^{1}$ guaranteed
by \eqref{eq:Gammai1 def}. Collecting these observations, we calculate
as
\[
\Gamma^{\mathcal{V}_{i}^{1}}=\min_{\partial\mathcal{V}_{i}^{1}}\mathbb{H}-\min_{\mathcal{V}_{i}^{1}}\mathbb{H}=(\mathbb{H}(\mathcal{P}_{i}^{1})+\Gamma_{i}^{1})-\mathbb{H}(\mathcal{P}_{i}^{1})=\Gamma_{i}^{1}.
\]
This completes the proof of Lemma \ref{lem:Vi1 cycle}.
\end{proof}
Collect
\begin{equation}
\mathscr{C}^{1}:=\{\mathcal{V}_{i}^{1}:\,i\in[1,\,\nu_{0}]\}.\label{eq:C1 def}
\end{equation}
We apply the general construction given in Definition \ref{def:gen const}
to $(\mathscr{C}^{1},\,\Gamma^{\star,1})$. By \eqref{eq:PC def}
and Lemma \ref{lem:Vi1 cycle}, it readily holds that $\mathscr{P}^{\mathscr{C}^{1}}=\mathscr{P}^{1}$.
\begin{notation}
\label{nota:abbr 1}We abbreviate $\mathscr{C}^{\star,1}:=(\mathscr{C}^{1})^{\star}$,
$\mathscr{C}^{\sharp,1}:=(\mathscr{C}^{1})^{\sharp}$, $\mathscr{P}^{\star,1}:=\mathscr{P}^{(\mathscr{C}^{1})^{\star}}$,
$\mathscr{P}^{\sharp,1}:=\mathscr{P}^{(\mathscr{C}^{1})^{\sharp}}$,
$\Delta^{1}:=\Delta^{\mathscr{C}^{1}}$, $\Omega^{1}:=\Omega^{\mathscr{C}^{1}}$,
$\mathfrak{X}^{1}(t):=\mathfrak{X}^{\mathscr{C}^{1}}(t)$, $\mathfrak{R}^{1}(\cdot,\,\cdot):=\mathfrak{R}^{\mathscr{C}^{1}}(\cdot,\,\cdot)$,
$\mathfrak{X}^{\star,1}(t):=\mathfrak{X}^{(\mathscr{C}^{1})^{\star}}(t)$
and $\mathfrak{R}^{\star,1}(\cdot,\,\cdot):=\mathfrak{R}^{(\mathscr{C}^{1})^{\star}}(\cdot,\,\cdot)$.
\end{notation}

Note that by \eqref{eq:Gamma-star1 def} and Lemma \ref{lem:Vi1 cycle},
$\Gamma^{\mathcal{V}_{i}^{1}}\ge\Gamma^{\star,1}$ for all $i\in[1,\,\nu_{0}]$,
thus it holds that
\begin{equation}
\mathscr{C}^{\star,1}=\mathscr{C}^{1}\quad\text{and}\quad\mathscr{C}^{\sharp,1}=\emptyset,\label{eq:C-star1 C-sharp1 prop}
\end{equation}
and accordingly,
\begin{equation}
\mathscr{P}^{\star,1}=\mathscr{P}^{1}\quad\text{and}\quad\mathscr{P}^{\sharp,1}=\emptyset.\label{eq:P-star1 P-sharp1 prop}
\end{equation}
Now, $\{\mathfrak{X}^{\star,1}(t)\}_{t\ge0}$ decomposes $\mathscr{P}^{\star,1}$
into
\begin{equation}
\mathscr{P}^{\star,1}=\mathscr{P}_{1}^{\star,1}\cup\cdots\cup\mathscr{P}_{\nu_{1}}^{\star,1}\cup\mathscr{P}_{{\rm tr}}^{\star,1},\label{eq:P-star1 dec}
\end{equation}
where $\mathscr{P}_{1}^{\star,1},\,\dots,\,\mathscr{P}_{\nu_{1}}^{\star,1}$
are irreducible components and $\mathscr{P}_{{\rm tr}}^{\star,1}$
is the collection of transient elements. Accordingly, decompose
\begin{equation}
\mathscr{C}^{\star,1}=\mathscr{C}_{1}^{\star,1}\cup\cdots\cup\mathscr{C}_{\nu_{1}}^{\star,1}\cup\mathscr{C}_{{\rm tr}}^{\star,1},\label{eq:C-star1 dec}
\end{equation}
where $\mathscr{C}_{m}^{\star,1}:=\{\mathcal{V}_{i}^{1}\in\mathscr{C}^{\star,1}:\,\mathcal{P}_{i}^{1}\in\mathscr{P}_{m}^{\star,1}\}$
for $m\in\{1,\,\dots,\,\nu_{1},\,{\rm tr}\}$.

\subsubsection*{From level $h-1$ to level $h$}

For an integer $h\ge2$, suppose that we are given a collection $\mathscr{C}^{h-1}$
of disjoint cycles in $\overline{\Omega}$ and a collection $\mathscr{P}^{h-1}=\{\mathcal{F}(\mathcal{C}):\,\mathcal{C}\in\mathscr{C}^{h-1}\}$,
along with decompositions
\begin{equation}
\mathscr{C}^{h-1}=\bigcup_{m\in[1,\,\nu_{h-1}]}\mathscr{C}_{m}^{\star,h-1}\cup\mathscr{C}_{{\rm tr}}^{\star,h-1}\cup\mathscr{C}^{\sharp,h-1}\label{eq:Ch-1 Ph-1 ind hyp}
\end{equation}
and
\[
\mathscr{P}^{h-1}=\bigcup_{m\in[1,\,\nu_{h-1}]}\mathscr{P}_{m}^{\star,h-1}\cup\mathscr{P}_{{\rm tr}}^{\star,h-1}\cup\mathscr{P}^{\sharp,h-1},
\]
where $\mathscr{P}_{m}^{\star,h-1}=\{\mathcal{F}(\mathcal{C}):\,\mathcal{C}\in\mathscr{C}_{m}^{\star,h-1}\}$
for $m\in\{1,\,\dots,\,\nu_{h-1},\,{\rm tr}\}$ and $\mathscr{P}^{\sharp,h-1}=\{\mathcal{F}(\mathcal{C}):\,\mathcal{C}\in\mathscr{C}^{\sharp,h-1}\}$.
Note that Lemma \ref{lem:Vi1 cycle}, \eqref{eq:C-star1 C-sharp1 prop},
\eqref{eq:P-star1 P-sharp1 prop}, \eqref{eq:P-star1 dec} and \eqref{eq:C-star1 dec}
guarantee this assumption for $h=2$. Further suppose that $\nu_{h-1}\ge2$.
Then, for each $i\in[1,\,\nu_{h-1}]$ define
\begin{equation}
\mathcal{P}_{i}^{h}:=\bigcup_{\mathcal{P}\in\mathscr{P}_{i}^{\star,h-1}}\mathcal{P}\quad\text{and}\quad\mathscr{P}^{\star,h}:=\{\mathcal{P}_{i}^{h}:\,i\in[1,\,\nu_{h-1}]\}.\label{eq:P-starh def}
\end{equation}

\begin{lem}
\label{lem:Pih energy}For each $i\in[1,\,\nu_{h-1}]$, it holds that
$\mathbb{H}(\eta)=\mathbb{H}(\xi)$ for all $\eta,\,\xi\in\mathcal{P}_{i}^{h}$.
\end{lem}

As in \eqref{eq:Gammai1 def}, define the \emph{$h$-th depth} as
\begin{equation}
\Gamma_{i}^{h}:=\Phi(\mathcal{P}_{i}^{h},\,\breve{\mathcal{P}}_{i}^{h})-\mathbb{H}(\mathcal{P}_{i}^{h})>0,\quad\text{where}\quad\breve{\mathcal{P}}_{i}^{h}:=\bigcup_{j\in[1,\,\nu_{h-1}]:\,j\ne i}\mathcal{P}_{j}^{h}.\label{eq:Gammaih def}
\end{equation}
Then, write
\begin{equation}
\Gamma^{\star,h}:=\min_{i\in[1,\,\nu_{h-1}]}\Gamma_{i}^{h}>0.\label{eq:Gamma-starh def}
\end{equation}

\begin{lem}
\label{lem:Gamma-star inc}It holds that $\Gamma^{\star,h}>\Gamma^{\star,h-1}$.
\end{lem}

Next, define
\begin{equation}
\mathcal{V}_{i}^{h}:=\{\eta\in\Omega:\,\Phi(\mathcal{P}_{i}^{h},\,\eta)-\mathbb{H}(\mathcal{P}_{i}^{h})<\Gamma_{i}^{h}\}.\label{eq:Vih def}
\end{equation}

\begin{lem}
\label{lem:Vih cycle}Collections $\mathcal{V}_{i}^{h}$ for $i\in[1,\,\nu_{h-1}]$
are disjoint cycles in $\overline{\Omega}$ with bottom $\mathcal{P}_{i}^{h}$
and depth $\Gamma_{i}^{h}$.
\end{lem}

Now, define
\begin{equation}
\mathscr{C}^{h}:=\{\mathcal{V}_{i}^{h}:\,i\in[1,\,\nu_{h-1}]\}\cup\{\mathcal{C}\in\mathscr{C}_{{\rm tr}}^{\star,h-1}\cup\mathscr{C}^{\sharp,h-1}:\,\mathcal{C}\cap\mathcal{V}_{i}^{h}=\emptyset\quad\text{for all}\quad i\in[1,\,\nu_{h-1}]\}.\label{eq:Ch def}
\end{equation}
Then, we apply the construction in Definition \ref{def:gen const}
to $(\mathscr{C}^{h},\,\Gamma^{\star,h})$.
\begin{lem}
\label{lem:Ch prop}It holds that $(\mathscr{C}^{h})^{\star}=\{\mathcal{V}_{i}^{h}:\,i\in[1,\,\nu_{h-1}]\}$
and
\[
(\mathscr{C}^{h})^{\sharp}=\{\mathcal{C}\in\mathscr{C}_{{\rm tr}}^{\star,h-1}\cup\mathscr{C}^{\sharp,h-1}:\,\mathcal{C}\cap\mathcal{V}_{i}^{h}=\emptyset\quad\text{for all}\quad i\in[1,\,\nu_{h-1}]\}.
\]
\end{lem}

By Lemmas \ref{lem:Vih cycle} and \ref{lem:Ch prop}, we have $\mathscr{P}^{(\mathscr{C}^{h})^{\star}}=\{\mathcal{P}_{i}^{h}:\,i\in[1,\,\nu_{h-1}]\}=\mathscr{P}^{\star,h}$.
\begin{notation}
\label{nota:abbr h}Abbreviate $\mathscr{P}^{h}:=\mathscr{P}^{\mathscr{C}^{h}}$,
$\mathscr{C}^{\star,h}:=(\mathscr{C}^{h})^{\star}$, $\mathscr{C}^{\sharp,h}:=(\mathscr{C}^{h})^{\sharp}$,
$\mathscr{P}^{\sharp,h}:=\mathscr{P}^{(\mathscr{C}^{h})^{\sharp}}$,
$\Delta^{h}:=\Delta^{\mathscr{C}^{h}}$, $\Omega^{h}:=\Omega^{\mathscr{C}^{h}}$,
$\mathfrak{X}^{h}(t):=\mathfrak{X}^{\mathscr{C}^{h}}(t)$, $\mathfrak{R}^{h}(\cdot,\,\cdot):=\mathfrak{R}^{\mathscr{C}^{h}}(\cdot,\,\cdot)$,
$\mathfrak{X}^{\star,h}(t):=\mathfrak{X}^{(\mathscr{C}^{h})^{\star}}(t)$
and $\mathfrak{R}^{\star,h}(\cdot,\,\cdot):=\mathfrak{R}^{(\mathscr{C}^{h})^{\star}}(\cdot,\,\cdot)$.
\end{notation}

The Markov chain $\{\mathfrak{X}^{\star,h}(t)\}_{t\ge0}$ decomposes
$\mathscr{P}^{\star,h}$ into
\begin{equation}
\mathscr{P}^{\star,h}=\mathscr{P}_{1}^{\star,h}\cup\cdots\cup\mathscr{P}_{\nu_{h}}^{\star,h}\cup\mathscr{P}_{{\rm tr}}^{\star,h},\label{eq:P-starh dec}
\end{equation}
with $\nu_{h}$ irreducible components and a transient collection.
Accordingly, we obtain
\begin{equation}
\mathscr{C}^{\star,h}=\mathscr{C}_{1}^{\star,h}\cup\cdots\cup\mathscr{C}_{\nu_{h}}^{\star,h}\cup\mathscr{C}_{{\rm tr}}^{\star,h},\label{eq:C-starh dec}
\end{equation}
where $\mathscr{C}_{m}^{\star,h}:=\{\mathcal{V}_{i}^{h}\in\mathscr{C}^{\star,h}:\,\mathcal{P}_{i}^{h}\in\mathscr{P}_{m}^{\star,h}\}$
for $m\in\{1,\,\dots,\,\nu_{h},\,{\rm tr}\}$.

Finally, we may repeat the same inductive procedure provided that
$\nu_{h}\ge2$. The proof of Lemmas \ref{lem:Pih energy}, \ref{lem:Gamma-star inc},
\ref{lem:Vih cycle} and \ref{lem:Ch prop} are given in Section \ref{sec5.1}.

According to the construction, the number of irreducible components
decreases strictly:
\begin{thm}
\label{thm:nu dec}For all $h\ge1$, it holds that $\nu_{h}<\nu_{h-1}$.
\end{thm}

\begin{figure}
{\small{}\begin{tikzpicture}[scale=0.85]
\draw (0,3) sin (0.5,0);
\draw (0.5,0) cos (0.75,0.5); \draw (0.75,0.5) sin (1,1);
\draw (1,1) cos (1.25,0.5); \draw (1.25,0.5) sin (1.5,0);
\draw (1.5,0) cos (1.75,1); \draw (1.75,1) sin (2,2);
\draw (2,2) cos (2.25,1); \draw (2.25,1) sin (2.5,0);
\draw (2.5,0) cos (2.75,1.5); \draw (2.75,1.5) sin (3,3);
\draw (3,3) cos (3.25,2.5); \draw (3.25,2.5) sin (3.5,2);
\draw (3.5,2) cos (3.75,2.5); \draw (3.75,2.5) sin (4,3);
\draw (4,3) cos (4.25,2); \draw (4.25,2) sin (4.5,1);
\draw (4.5,1) cos (4.75,1.5); \draw (4.75,1.5) sin (5,2);
\draw (5,2) cos (5.25,1.5); \draw (5.25,1.5) sin (5.5,1);
\draw (5.5,1) cos (5.75,2); \draw (5.75,2) sin (6,3);
\draw (6,3) cos (6.25,1.5); \draw (6.25,1.5) sin (6.5,0);
\draw (6.5,0) cos (7,3);

\foreach \i in {0.5,1.5,2.5,6.5} {
\draw[very thick] (\i-0.1,0)--(\i+0.1,0); }
\foreach \i in {4.5,5.5} {
\draw[very thick] (\i-0.1,1)--(\i+0.1,1); }
\foreach \i in {3.5} {
\draw[very thick] (\i-0.1,2)--(\i+0.1,2); }

\draw (0.5,0) node[below]{$\bm{\mathcal{P}_1^1}$};
\draw (1.5,0) node[below]{$\bm{\mathcal{P}_2^1}$};
\draw (2.5,0) node[below]{$\bm{\mathcal{P}_3^1}$};
\draw (3.5,2) node[below]{$\mathcal{P}_4^1$};
\draw (4.5,1) node[below]{$\bm{\mathcal{P}_5^1}$};
\draw (5.5,1) node[below]{$\bm{\mathcal{P}_6^1}$};
\draw (6.5,0) node[below]{$\bm{\mathcal{P}_7^1}$};

\draw[red,<-] (0.5,0.75) sin (1,1.25); \draw[red,<-] (1.5,0.75) sin (1,1.25);
\draw[red,<-] (2.5,2.75) sin (3,3.25); \draw[red] (3.5,2.75) sin (3,3.25);
\draw[red] (3.5,2.75) sin (4,3.25); \draw[red,<-] (4.5,2.75) sin (4,3.25);
\draw[red,<-] (4.5,1.75) sin (5,2.25); \draw[red,<-] (5.5,1.75) sin (5,2.25);
\end{tikzpicture}
\hspace{20mm}
\begin{tikzpicture}[scale=0.85]
\draw (0,3) sin (0.5,0);
\draw (0.5,0) cos (0.75,0.5); \draw (0.75,0.5) sin (1,1);
\draw (1,1) cos (1.25,0.5); \draw (1.25,0.5) sin (1.5,0);
\draw (1.5,0) cos (1.75,1); \draw (1.75,1) sin (2,2);
\draw (2,2) cos (2.25,1); \draw (2.25,1) sin (2.5,0);
\draw (2.5,0) cos (2.75,1.5); \draw (2.75,1.5) sin (3,3);
\draw (3,3) cos (3.25,2.5); \draw (3.25,2.5) sin (3.5,2);
\draw (3.5,2) cos (3.75,2.5); \draw (3.75,2.5) sin (4,3);
\draw (4,3) cos (4.25,2); \draw (4.25,2) sin (4.5,1);
\draw (4.5,1) cos (4.75,1.5); \draw (4.75,1.5) sin (5,2);
\draw (5,2) cos (5.25,1.5); \draw (5.25,1.5) sin (5.5,1);
\draw (5.5,1) cos (5.75,2); \draw (5.75,2) sin (6,3);
\draw (6,3) cos (6.25,1.5); \draw (6.25,1.5) sin (6.5,0);
\draw (6.5,0) cos (7,3);

\draw[very thick] (0.4,0)--(1.6,0);
\draw[very thick] (2.4,0)--(2.6,0);
\draw[very thick] (4.4,1)--(5.6,1);
\draw[very thick] (6.4,0)--(6.6,0);

\draw (1,0) node[below]{$\bm{\mathcal{P}_1^2}$};
\draw (2.5,0) node[below]{$\bm{\mathcal{P}_2^2}$};
\draw (5,1) node[below]{$\mathcal{P}_3^2$};
\draw (6.5,0) node[below]{$\bm{\mathcal{P}_4^2}$};

\draw[red,<-] (1,1.75) sin (1.75,2.25); \draw[red,<-] (2.5,1.75) sin (1.75,2.25);
\draw[red,<-] (2.5,2.75) sin (3.75,3.25); \draw[red] (5,2.75) sin (3.75,3.25);
\draw[red] (5,2.75) sin (5.75,3.25); \draw[red,<-] (6.5,2.75) sin (5.75,3.25);
\end{tikzpicture}\\
\begin{tikzpicture}[scale=0.85]
\draw (0,3) sin (0.5,0);
\draw (0.5,0) cos (0.75,0.5); \draw (0.75,0.5) sin (1,1);
\draw (1,1) cos (1.25,0.5); \draw (1.25,0.5) sin (1.5,0);
\draw (1.5,0) cos (1.75,1); \draw (1.75,1) sin (2,2);
\draw (2,2) cos (2.25,1); \draw (2.25,1) sin (2.5,0);
\draw (2.5,0) cos (2.75,1.5); \draw (2.75,1.5) sin (3,3);
\draw (3,3) cos (3.25,2.5); \draw (3.25,2.5) sin (3.5,2);
\draw (3.5,2) cos (3.75,2.5); \draw (3.75,2.5) sin (4,3);
\draw (4,3) cos (4.25,2); \draw (4.25,2) sin (4.5,1);
\draw (4.5,1) cos (4.75,1.5); \draw (4.75,1.5) sin (5,2);
\draw (5,2) cos (5.25,1.5); \draw (5.25,1.5) sin (5.5,1);
\draw (5.5,1) cos (5.75,2); \draw (5.75,2) sin (6,3);
\draw (6,3) cos (6.25,1.5); \draw (6.25,1.5) sin (6.5,0);
\draw (6.5,0) cos (7,3);

\draw[very thick] (0.4,0)--(2.6,0);
\draw[very thick] (6.4,0)--(6.6,0);

\draw (1.5,0) node[below]{$\bm{\mathcal{P}_1^3}$};
\draw (6.5,0) node[below]{$\bm{\mathcal{P}_2^3}$};

\draw[red,<-] (1.5,2.75) sin (4,3.75); \draw[red,<-] (6.5,2.75) sin (4,3.75);
\end{tikzpicture}}\caption{\label{Fig2.1}Example of a hierarchical decomposition of stable plateaux
in $\mathscr{P}^{1}$ with $\mathfrak{m}=3$. At each level $h\in[1,\,3]$,
bold-faced elements are recurrent and the rest are transient with
respect to $\{\mathfrak{X}^{\star,h}(t)\}_{t\ge0}$. At level $1$,
we have $\mathscr{P}_{1}^{\star,1}=\{\mathcal{P}_{1}^{1},\,\mathcal{P}_{2}^{1}\}$,
$\mathscr{P}_{2}^{\star,1}=\{\mathcal{P}_{3}^{1}\}$, $\mathscr{P}_{3}^{\star,1}=\{\mathcal{P}_{5}^{1},\,\mathcal{P}_{6}^{1}\}$,
$\mathscr{P}_{4}^{\star,1}=\{\mathcal{P}_{7}^{1}\}$ and $\mathscr{P}_{{\rm tr}}^{\star,1}=\{\mathcal{P}_{4}^{1}\}$.
At level $2$, we have $\mathscr{P}_{1}^{\star,2}=\{\mathcal{P}_{1}^{2},\,\mathcal{P}_{2}^{2}\}$,
$\mathscr{P}_{2}^{\star,2}=\{\mathcal{P}_{4}^{2}\}$ and $\mathscr{P}_{{\rm tr}}^{\star,2}=\{\mathcal{P}_{3}^{2}\}$.
Finally, at level $\mathfrak{m}=3$, we have $\mathscr{P}^{\star,3}=\mathscr{P}_{1}^{\star,3}=\{\mathcal{P}_{1}^{3},\,\mathcal{P}_{2}^{3}\}$
which is exactly composed of the ground states.}
\end{figure}
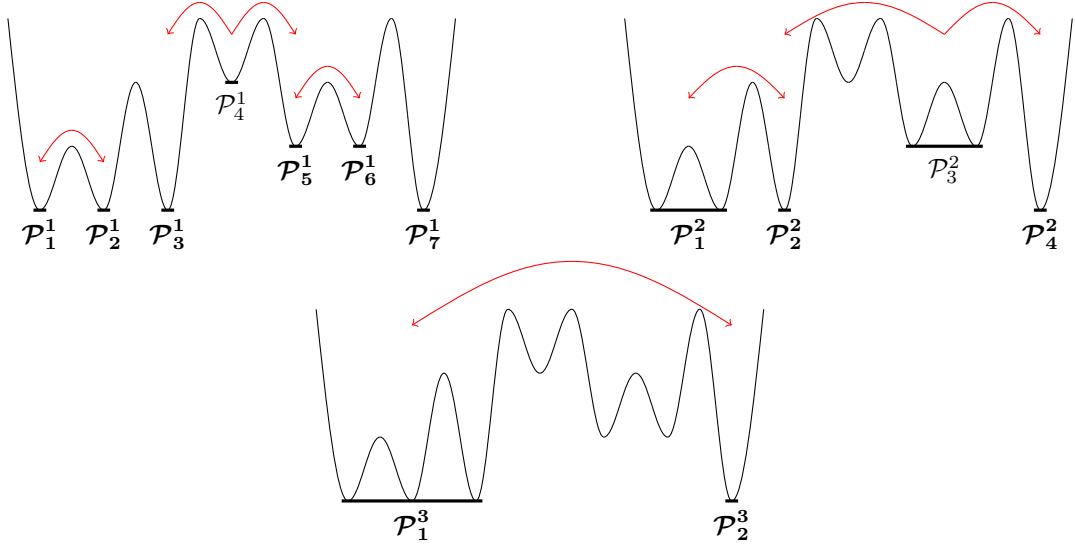

By Theorem \ref{thm:nu dec}, there exists a terminal integer $\mathfrak{m}$
such that $\nu_{\mathfrak{m}-1}>1$ and $\nu_{\mathfrak{m}}=1$. In
turn, the sequence $\mathscr{P}^{1},\,\mathscr{P}^{2},\,\dots,\,\mathscr{P}^{\mathfrak{m}}$
constitutes the full hierarchical decomposition of the stable plateaux
in $\overline{\Omega}$. See Figure \ref{Fig2.1} for an illustration.

In addition, the ground states in $\mathcal{S}$ are always contained
in the recurrent collection: for each $h\ge1$, define
\begin{equation}
\mathscr{P}_{{\rm rec}}^{\star,h}:=\mathscr{P}_{1}^{\star,h}\cup\cdots\cup\mathscr{P}_{\nu_{h}}^{\star,h}\quad\text{and}\quad\mathscr{C}_{{\rm rec}}^{\star,h}:=\mathscr{C}_{1}^{\star,h}\cup\cdots\cup\mathscr{C}_{\nu_{h}}^{\star,h}.\label{eq:P-starh C-starh rec}
\end{equation}

\begin{thm}[Ground states are always recurrent]
\label{thm:ground states rec}For all $\bm{s}\in\mathcal{S}$ and
$h\ge1$, there exists $\mathcal{P}_{i}^{h}\in\mathscr{P}_{{\rm rec}}^{\star,h}$
such that $\bm{s}\in\mathcal{P}_{i}^{h}$.
\end{thm}

In particular, by Theorem \ref{thm:ground states rec} with $h=\mathfrak{m}$
and Lemma \ref{lem:Pih energy}, the unique irreducible collection
$\mathscr{P}_{1}^{\star,\mathfrak{m}}$ of the terminal level $\mathfrak{m}$
consists of exactly all the ground states in $\mathcal{S}$. We provide
a proof of Theorems \ref{thm:nu dec} and \ref{thm:ground states rec}
in Sections \ref{sec4.2} (initial step) and \ref{sec5.3} (inductive
step).

\subsection{\label{sec2.2}Metastable hierarchy of tunneling transitions between
stable plateaux}

For each $h\in[1,\,\mathfrak{m}]$, define
\begin{equation}
\mathcal{V}^{\star,h}:=\bigcup_{i\in[1,\,\nu_{h-1}]}\mathcal{V}_{i}^{h}.\label{eq:V-starh def}
\end{equation}
Define a projection function $\Psi^{h}:\,\mathcal{V}^{\star,h}\to\mathscr{P}^{\star,h}$
as
\begin{equation}
\Psi^{h}(\eta):=\mathcal{P}_{i}^{h}\quad\text{for each}\quad\eta\in\mathcal{V}_{i}^{h}\quad\text{and}\quad i\in[1,\,\nu_{h-1}].\label{eq:Psih def}
\end{equation}
Consider the trace process $\{\eta_{\beta}^{h}(t)\}_{t\ge0}$ of the
original process in $\mathcal{V}^{\star,h}$. Then, define the \emph{$h$-th
order process} $\{X_{\beta}^{h}(t)\}_{t\ge0}$ in $\mathscr{P}^{\star,h}$
as 
\[
X_{\beta}^{h}(t):=\Psi^{h}(\eta_{\beta}^{h}(e^{\Gamma^{\star,h}\beta}t))\quad\text{for}\quad t\ge0.
\]
We are ready to state our main result.
\begin{thm}[Hierarchical tunneling metastable transitions]
\label{thm:main}For each $h\in[1,\,\mathfrak{m}]$, the following
statements are valid.
\begin{enumerate}
\item The $h$-th order process $\{X_{\beta}^{h}(t)\}_{t\ge0}$ in $\mathscr{P}^{\star,h}$
converges to $\{\mathfrak{X}^{\star,h}(t)\}_{t\ge0}$.
\item The original process spends negligible time outside $\mathcal{V}^{\star,h}$
in the time scale $e^{\Gamma^{\star,h}\beta}$:
\[
\lim_{\beta\to\infty}\mathbb{E}_{\eta}\Big[\int_{0}^{T}{\bf 1}\{\eta_{\beta}(e^{\Gamma^{\star,h}\beta}t)\notin\mathcal{V}^{\star,h}\}{\rm d}t\Big]=0\quad\text{for all}\quad T>0\quad\text{and}\quad\eta\in\mathcal{V}^{\star,h}.
\]
\end{enumerate}
\end{thm}

According to Theorem \ref{thm:main}, \eqref{eq:P-star P-sharp def}
and \eqref{eq:RC def}, at each level $h\in[1,\,\mathfrak{m}]$, there
exist tunneling transitions between the stable plateaux of depth at
least $\Gamma^{\star,h}$ in the time scale $e^{\beta\Gamma^{\star,h}}$,
where those with depth strictly greater than $\Gamma^{\star,h}$ are
absorbing. This is consistent with \eqref{eq:Gamma-starh def} in
that $\Gamma^{\star,h}$ is the minimum energy barrier between the
stable plateaux at level $h$.

We prove Theorem \ref{thm:main} in Sections \ref{sec4.3} (initial
step) and \ref{sec5.4} (inductive step).

\subsection{\label{sec2.3}Exit distribution from cycles}

In this subsection, we record a theorem regarding the precise exit
distribution of cycles in low-temperature lattice models. Fix a cycle
$\mathcal{C}$ which is not the full set $\Omega$ and recall the
definitions of $\partial\mathcal{C}$ and $\partial^{\star}\mathcal{C}$
from Notation \ref{nota:bot bdry def}. It is well known that (cf.
\cite[Corollary 6.25]{OV}) it is most likely to escape $\mathcal{C}$
via $\partial^{\star}\mathcal{C}$, the minimizing set of $\mathbb{H}$
in $\partial\mathcal{C}$:
\[
\lim_{\beta\to\infty}\inf_{\eta\in\mathcal{C}}\mathbb{P}_{\eta}[\eta_{\beta}(\mathcal{T}_{\partial\mathcal{C}})\in\partial^{\star}\mathcal{C}]=1.
\]
However, according to the previous methods to metastability, it was
unable to characterize the exact probability of escaping $\mathcal{C}$
via a specific configuration $\xi\in\partial^{\star}\mathcal{C}$
in the limit as $\beta\to\infty$. It turns out that this is now possible
by applying the new approach, the \emph{$H^{1}$-approximation method}
to metastability, recently developed in \cite{KS IsingPotts-2D,LanSeo RW-pot-field}.
\begin{thm}
\label{thm:exit H1}Given a cycle $\mathcal{C}\ne\Omega$, it holds
for each $\xi_{0}\in\partial^{\star}\mathcal{C}$ that
\[
\lim_{\beta\to\infty}\mathbb{P}_{\eta_{0}}[\eta_{\beta}(\mathcal{T}_{\partial\mathcal{C}})=\xi_{0}]=\frac{\sum_{\eta\in\mathcal{C}}{\bf 1}\{\eta\sim\xi_{0}\}}{\sum_{\eta\in\mathcal{C}}\sum_{\xi\in\partial^{\star}\mathcal{C}}{\bf 1}\{\eta\sim\xi\}}\quad\text{for all}\quad\eta_{0}\in\mathcal{C}.
\]
\end{thm}

In particular, Theorem \ref{thm:exit H1} states that the probability
of escaping $\mathcal{C}$ at $\xi_{0}\in\partial^{\star}\mathcal{C}$
is asymptotically proportional to the number of configurations in
$\mathcal{C}$ connected to $\xi_{0}$. This is consistent with the
definition of the transition rate of the induced Markov chain in \eqref{eq:RC def},
and plays a crucial role in the proof of Lemma \ref{lem:Fbeta Deltah}.
We prove Theorem \ref{thm:exit H1} in Section \ref{sec9}.

\section{\label{sec3}Kawasaki Dynamics with Macroscopic Number of Ising Particles}

In this section, we apply the general results presented in Section
\ref{sec2} to the Kawasaki dynamics with macroscopic number of Ising
particles.

\subsection*{Ising gas model}

For fixed positive integers $K$ and $L$, we consider a two-dimensional
periodic square lattice $\Lambda=(V,\,E)$ of side lengths $K$ and
$L$, i.e.,
\[
V:=\mathbb{T}_{K}\times\mathbb{T}_{L}=\{0,\,1,\,\dots,\,K-1\}\times\{0,\,1,\,\dots,\,L-1\},
\]
and $E$ is the set of unordered nearest-neighbor bonds in $V$. To
fix ideas, we assume that $K>L$.\footnote{We briefly discuss the case of $K=L$ in Remark \ref{rem:K=00003DL}.}

For a positive integer $\mathscr{N}$, the particle configuration
space $\Omega=\Omega_{\mathscr{N}}$ is defined as
\[
\Omega:=\Big\{\eta=(\eta(x))_{x\in V}\in\{0,\,1\}^{V}:\,\sum_{x\in V}\eta(x)=\mathscr{N}\Big\},
\]
where value $1$ (resp. $0$) indicates the occupied (resp. vacant)
state. In this sense, $\mathscr{N}$ indicates the total number of
particles in the system. We are interested in the case where there
exist macroscopic number of particles in the system. For this purpose,
we assume that $\mathscr{N}=L\mathscr{N}_{0}$ for an integer $\mathscr{N}_{0}$
such that
\begin{equation}
\frac{L}{4}<\mathscr{N}_{0}<\frac{K}{2},\quad\text{thus}\quad\frac{L^{2}}{4}<\mathscr{N}<\frac{KL}{2}.\label{eq:N0 assump}
\end{equation}
For each $\eta\in\Omega$, we define the Hamiltonian $\mathbb{H}(\eta)$
as
\begin{equation}
\mathbb{H}(\eta):=-\sum_{\{x,\,y\}\in E}\eta(x)\eta(y).\label{eq:H def}
\end{equation}
It is clear that $\mathbb{H}:\,\Omega\to\mathbb{R}$ is an integer-valued
function. According to this Hamiltonian, we assign a Gibbs measure
to $\Omega$ as in \eqref{eq:Gibbs def}.
\begin{rem}
\label{rem:spin system}We also regard this gas model as an Ising
spin system in $\Lambda$ with two spins $0$ and $1$.
\end{rem}

\subsection*{Kawasaki dynamics}

We define a dynamical system in $\Omega$. The \emph{Kawasaki dynamics}
in $\Omega$ is the continuous-time Markov chain $\{\eta_{\beta}(t)\}_{t\ge0}$
whose transition rate function $r_{\beta}:\,\Omega\times\Omega\to[0,\,\infty)$
is defined as
\begin{equation}
r_{\beta}(\eta,\,\xi):=\begin{cases}
e^{-\beta\max\{\mathbb{H}(\eta^{x\leftrightarrow y})-\mathbb{H}(\eta),\,0\}} & \text{if}\quad\xi=\eta^{x\leftrightarrow y}\ne\eta\quad\text{and}\quad\{x,\,y\}\in E,\\
0 & \text{otherwise}.
\end{cases}\label{eq:r-beta-def}
\end{equation}
Here, $\eta^{x\leftrightarrow y}$ is obtained from $\eta$ by exchanging
the states (vacant or occupied) at $x$ and $y$:
\begin{equation}
\eta^{x\leftrightarrow y}(x)=\eta(y),\quad\eta^{x\leftrightarrow y}(y)=\eta(x),\quad\text{and}\quad\eta^{x\leftrightarrow y}(z)=\eta(z)\quad\text{for}\quad z\ne x,\,y.\label{eq:eta-xy def}
\end{equation}
According to the Kawasaki dynamics, each particle in the lattice jumps
independently to its vacant neighbor with rate $1$ if the jump does
not increase the energy and with exponentially small rate $e^{-\beta\Delta}$
if the jump increases the energy by $\Delta>0$.

Denote by $L_{\beta}$ the corresponding infinitesimal stochastic
generator. It is clear that the Kawasaki dynamics $\{\eta_{\beta}(t)\}_{t\ge0}$
in $\Omega$ fits into the general setting considered in Section \ref{sec2}.

\subsection*{Ground states}

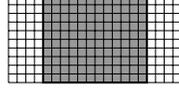
\begin{figure}
\begin{tikzpicture}[scale=0.115]
\fill[white] (-4,0) rectangle (16,10);

\fill[black!40!white] (0,0) rectangle (12,10);

\draw[very thin] (-4,0) grid (16,10);
\foreach \i in {0} {
\draw[thick] (0+20*\i,0) rectangle (12+20*\i,10);
}
\end{tikzpicture}\caption{\label{Fig3.1}Configuration $\bm{\sigma^{k}}$. We illustrate each
configuration in the dual lattice in the sense that the gray (resp.
white) faces indicate the $\mathscr{N}=L\mathscr{N}_{0}$ occupied
(resp. $KL-\mathscr{N}$ vacant) sites. Moreover, in each illustration,
the region surrounded by thick line indicates $\mathfrak{c}^{k}\cup\mathfrak{c}^{k+1}\cup\cdots\cup\mathfrak{c}^{k+\mathscr{N}_{0}-1}$
(cf. \eqref{eq:ck def}).}
\end{figure}

For each $k\in\mathbb{T}_{K}$, denote by $\mathfrak{c}^{k}$ the
$k$-th column in $\Lambda=\mathbb{T}_{K}\times\mathbb{T}_{L}$:
\begin{equation}
\mathfrak{c}^{k}:=\{k\}\times\mathbb{T}_{L}.\label{eq:ck def}
\end{equation}
Then, denote by $\bm{\sigma^{k}}\in\Omega$ the configuration such
that (see Figure \ref{Fig3.1})
\begin{equation}
\{x\in V:\,\bm{\sigma^{k}}(x)=1\}=\mathfrak{c}^{k}\cup\mathfrak{c}^{k+1}\cup\cdots\cup\mathfrak{c}^{k+\mathscr{N}_{0}-1}.\label{eq:sigmak def}
\end{equation}

\begin{thm}[Ground states]
\label{thm:Ham S}For all $\eta\in\Omega$, it holds that $\mathbb{H}(\eta)\ge-2L\mathscr{N}_{0}+L$.
Moreover, equality holds if and only if $\eta=\bm{\sigma^{k}}$ for
some $k\in\mathbb{T}_{K}$. In turn, $\mathcal{S}=\{\bm{\sigma^{k}}:\,k\in\mathbb{T}_{K}\}$.
\end{thm}

The proof of Theorem \ref{thm:Ham S} is given in Section \ref{sec6.1}.
Hereafter, we define
\begin{equation}
\mathbb{H}_{0}:=-2L\mathscr{N}_{0}+L,\label{eq:H0 def}
\end{equation}
such that $\mathbb{H}(\eta)=\mathbb{H}_{0}$ for all $\eta\in\mathcal{S}$
and $\mathbb{H}(\eta)>\mathbb{H}_{0}$ for all $\eta\notin\mathcal{S}$.
\begin{rem}
\label{rem:phase trans particles}We note that the assumption in \eqref{eq:N0 assump}
highlights one of the main novelties here compared to the previous
studies \cite{BL Kawasaki,GoisLan}, where the authors considered
the case of $\mathscr{N}<\frac{L^{2}}{4}$. As a result, we notice
in Theorem \ref{thm:Ham S} that in order to lower the energy of the
system, the particles are most likely to line up in the vertical direction
(which has a shorter side length) and form a one-dimensional strip.
This is in contrast to the fact that when $\mathscr{N}<\frac{L^{2}}{4}$,
the particles are most likely to gather up and form a square cluster,
which has a two-dimensional geometry. This reveals a \emph{phase transition}
in the shape of the ground states at the sharp threshold $\mathscr{N}^{*}=\frac{L^{2}}{4}$.
\end{rem}

\subsection*{Energy barrier between ground states}

The energy barrier between the ground states in $\mathcal{S}$ is
presented as follows.
\begin{thm}[Energy barrier between ground states]
\label{thm:E barrier}For all distinct $k,\,k'\in\mathbb{T}_{K}$,
it holds that 
\[
\Phi(\bm{\sigma^{k}},\,\bm{\sigma^{k'}})=\mathbb{H}_{0}+4.
\]
\end{thm}

We investigate detailed features of the energy landscape in Sections
\ref{sec6} and \ref{sec7}. As a byproduct of these precise results,
we prove Theorem \ref{thm:E barrier} at the end of Section \ref{sec7.2}.

Theorem \ref{thm:E barrier} indicates that in the present setting,
it holds that (cf. \eqref{eq:Phi-bar def} and \eqref{eq:Omega-bar def})
\begin{equation}
\overline{\Phi}=\mathbb{H}_{0}+4,\quad\text{thus}\quad\overline{\Omega}=\{\eta\in\Omega:\,\Phi(\mathcal{S},\,\eta)\le\mathbb{H}_{0}+4\}.\label{eq:Kawa Omega-bar}
\end{equation}
 Hereafter, our investigation is focused on $\overline{\Omega}$.

\subsection*{Hierarchical tunneling transitions between metastable valleys}

Now, we are ready to state the main results on the hierarchical tunneling
transitions of the Kawasaki dynamics. Recall from \eqref{eq:P1 def}
that $\mathscr{P}^{1}$ is the collection of all stable plateaux in
$\overline{\Omega}$. First, we characterize $\mathscr{P}^{1}$. Since
the definitions of the stable plateaux that appear below are quite
complicated, we defer the exact definitions to their first appearances
in the storyline, which are mentioned in the theorem below.
\begin{thm}[Characterization of $\mathscr{P}^{1}$]
\label{thm:P1 char}The elements of $\mathscr{P}^{1}$ are classified
as follows.
\begin{itemize}
\item \textbf{Energy $\mathbb{H}_{0}$}
\begin{itemize}
\item $\{\bm{\sigma^{k}}\}$ for each $k\in\mathbb{T}_{K}$ (cf. \eqref{eq:sigmak def})
\end{itemize}
\item \textbf{Energy $\mathbb{H}_{0}+2$}
\begin{itemize}
\item $\{\bm{\sigma_{m;\ell,\ell'}^{k}}\}$ for each $m\in[2,\,L-2]$ and
$\ell,\,\ell'\in\mathbb{T}_{L}$ (cf. Definition \ref{def:shallow bot})
\item $\mathcal{S}_{1}^{k}$, $\mathcal{S}_{L-1}^{k-1}$ (cf. Definition
\ref{def:shallow bot}), $\mathcal{R}^{k}$ and $\mathcal{L}^{k}$
(cf. Table \ref{Tab1})
\item $\{\eta\}$ for each $\eta\in\mathcal{R}_{(i)}^{k}\cup\mathcal{L}_{(i)}^{k}$
for $i\in[2,\,\frac{L}{2}]$ (cf. Table \ref{Tab3})
\end{itemize}
\item \textbf{Energy $\mathbb{H}_{0}+3$}
\begin{itemize}
\item $\{\eta\}$ for each $\eta\in\mathcal{D}_{m}^{k}$ for $m\in[2,\,L-2]$
(cf. Table \ref{Tab2})
\item stable plateaux in $\mathcal{R}_{m,\pm}^{k}\cup\mathcal{L}_{m,\mp}^{k}$
for $m\in[2,\,L-2]$ (cf. Table \ref{Tab2})
\item $\{\eta\}$ for each $\widehat{\mathcal{R}}_{(i)}^{k}\cup\widehat{\mathcal{L}}_{(i)}^{k}$
for $i\in[2,\,\frac{L-2}{2}]$ (cf. Table \ref{Tab3})
\item stable plateaux in each $\mathcal{R}_{(i),\pm}^{k}\cup\mathcal{L}_{(i),\mp}^{k}\cup\widehat{\mathcal{R}}_{(i),\pm}^{k}\cup\widehat{\mathcal{L}}_{(i),\mp}^{k}$
for $i\in[2,\,\frac{L+1}{2}]$ (cf. Table \ref{Tab3})
\end{itemize}
\end{itemize}
\end{thm}

Next, we present the hierarchical decomposition of $\mathscr{P}^{1}$.
\begin{thm}[Hierarchical decomposition of $\mathscr{P}^{1}$]
\label{thm:hier dec}$ $
\begin{enumerate}
\item We have $\Gamma^{\star,1}=1$ and $\mathscr{P}^{1}=\mathscr{P}^{\star,1}$
is given in Theorem \ref{thm:P1 char} and decomposed as (cf. \eqref{eq:P-starh dec}
and \eqref{eq:P-starh C-starh rec})
\begin{align*}
\mathscr{P}_{{\rm rec}}^{\star,1}= & \bigcup_{k\in\mathbb{T}_{K}}\{\{\bm{\sigma^{k}}\}\}\cup\bigcup_{k\in\mathbb{T}_{K}}\bigcup_{m=2}^{L-2}\{\{\bm{\sigma_{m;\ell,\ell'}^{k}}\}:\,\ell,\,\ell'\in\mathbb{T}_{L}\}\\
 & \cup\bigcup_{k\in\mathbb{T}_{K}}\bigcup_{i\in[2,\,\frac{L}{2})}\{\{\eta\}:\,\eta\in\mathcal{R}_{(i)}^{k}\}\cup\bigcup_{k\in\mathbb{T}_{K}}\bigcup_{i\in[2,\,\frac{L}{2})}\{\{\eta\}:\,\eta\in\mathcal{L}_{(i)}^{k}\}\\
 & \cup\bigcup_{k\in\mathbb{T}_{K}}\bigcup_{\eta\in\mathcal{R}_{(\frac{L}{2})}^{k}}\{\{\eta\}\}\cup\bigcup_{k\in\mathbb{T}_{K}}\bigcup_{\eta\in\mathcal{L}_{(\frac{L}{2})}^{k}}\{\{\eta\}\}\quad\text{(if}\quad L\quad\text{is even)}.
\end{align*}
In particular, $\nu_{1}>1$.
\item We have $\Gamma^{\star,2}=2$ and $\mathscr{P}^{\star,2}$ is decomposed
as (cf. \eqref{eq:P-starh dec})
\[
\mathscr{P}_{{\rm rec}}^{\star,2}=\bigcup_{k\in\mathbb{T}_{K}}\{\{\bm{\sigma^{k}}\}\}
\]
and
\begin{align*}
\mathscr{P}_{{\rm tr}}^{\star,2}= & \bigcup_{k\in\mathbb{T}_{K}}\bigcup_{m=2}^{L-2}\{\mathcal{S}_{m}^{k}\}\cup\bigcup_{k\in\mathbb{T}_{K}}\bigcup_{i\in[2,\,\frac{L}{2})}\{\mathcal{R}_{(i)}^{k}\}\cup\bigcup_{k\in\mathbb{T}_{K}}\bigcup_{i\in[2,\,\frac{L}{2})}\{\mathcal{L}_{(i)}^{k}\}\\
 & \cup\bigcup_{k\in\mathbb{T}_{K}}\bigcup_{\eta\in\mathcal{R}_{(\frac{L}{2})}^{k}}\{\{\eta\}\}\cup\bigcup_{k\in\mathbb{T}_{K}}\bigcup_{\eta\in\mathcal{L}_{(\frac{L}{2})}^{k}}\{\{\eta\}\}\quad\text{(if}\quad L\quad\text{is even)}.
\end{align*}
In particular, $\nu_{2}=k>1$.
\item We have $\Gamma^{\star,3}=4$ and $\mathscr{P}^{\star,3}$ is decomposed
as
\[
\mathscr{P}^{\star,3}=\mathscr{P}_{{\rm rec}}^{\star,3}=\{\{\bm{\sigma^{k}}\}:\,k\in\mathbb{T}_{K}\}.
\]
In particular, $\nu_{3}=1$ thus $\mathfrak{m}=3$.
\end{enumerate}
\end{thm}

Finally, the tunneling transitions are characterized as follows.
\begin{thm}[Metastable hierarchy between stable plateaux]
\label{thm:Kawasaki}For each $h\in[1,\,3]$, there exists a limiting
Markov chain $\{\mathfrak{X}^{\star,h}(t)\}_{t\ge0}$ in $\mathscr{P}^{\star,h}$
such that the $e^{\beta\Gamma^{\star,h}}$-accelerated Kawasaki dynamics
converges to $\{\mathfrak{X}^{\star,h}(t)\}_{t\ge0}$ in the sense
of Theorem \ref{thm:main}.
\end{thm}

According to the remark after Theorem \ref{thm:ground states rec},
the level-$3$ transitions occur between the ground states $\{\bm{\sigma^{k}}\}$
for each $k\in\mathbb{T}_{K}$, which are indeed the deepest tunneling
transitions between the stable states. We record a characteristic
of this level-$3$ Markov chain which demonstrates the complicated
energy landscape of the Kawasaki dynamics. Recall that $\mathfrak{R}^{\star,h}(\cdot,\,\cdot)$
is the transition rate function of $\{\mathfrak{X}^{\star,h}(t)\}_{t\ge0}$.
\begin{thm}
\label{thm:pos rates}For all distinct $k,\,k'\in\mathbb{T}_{K}$,
it holds that $\mathfrak{R}^{\star,3}(\{\bm{\sigma^{k}}\},\,\{\bm{\sigma^{k'}}\})>0$.
\end{thm}

We prove Theorems \ref{thm:P1 char}, \ref{thm:hier dec}, \ref{thm:Kawasaki}
and \ref{thm:pos rates} in Section \ref{sec8}.
\begin{rem}[Uniformly positive rates]
\label{rem:pos rates}By Theorem \ref{thm:pos rates}, the limiting
dynamics $\{\mathfrak{X}^{\star,3}(t)\}_{t\ge0}$ in $\mathscr{P}^{\star,3}=\{\{\bm{\sigma^{k}}\}:\,k\in\mathbb{T}_{K}\}$
has uniformly positive rates. Geometrically, this means that a strip
starting from column $k\in\mathbb{T}_{K}$ can move an arbitrary distance
to another strip starting from column $k'\in\mathbb{T}_{K}$ with
a positive transition rate. This phenomenon is due to the highly complicated
geometry of the saddle structure subject to the Kawasaki dynamics,
which was also the case in \cite[Theorem 2.1]{BL Kawasaki}. In detail,
a single particle can move freely in the empty space without changing
the energy, which makes possible strange but typical metastable transitions.
See Figure \ref{Fig7.15} for such paths and also the proof of Theorem
\ref{thm:pos rates} at the end of Section \ref{sec8}.
\end{rem}

\begin{rem}[Case of $K=L$]
\label{rem:K=00003DL}Suppose here that $K=L$. Then, each stable
plateau classified in Theorem \ref{thm:P1 char} has a reflected counterpart
(by $\mathbb{T}_{K}\leftrightarrow\mathbb{T}_{L}$, which is possible
since $K=L$) with the same Hamiltonian value. Moreover, the hierarchical
metastable transitions in this reflected world remain the same as
characterized in Theorems \ref{thm:hier dec} and \ref{thm:Kawasaki},
but level $3$ would not be the terminal level since the original
ground states are not connected to the reflected ground states at
this level. Thus, the difference here occurs due to this additional
final transition, at level $\mathfrak{m}=4$, between the collection
of original ground states in $\mathcal{S}$ and the collection of
reflected new ground states. To see this, for simplicity we fix $k\in\mathbb{T}_{k}$,
recall $\bm{\sigma^{k}}\in\mathcal{S}$, and define a new configuration
$\bm{\widehat{\sigma}^{k}}$ defined as
\[
\{x\in V:\,\bm{\widehat{\sigma}^{k}}(x)=1\}=\mathfrak{r}^{k}\cup\mathfrak{r}^{k+1}\cup\cdots\cup\mathfrak{r}^{k+\mathscr{N}_{0}-1},
\]
where $\mathfrak{r}^{\ell}$ is the $\ell$-th row in $\Lambda$.
Then, according to the full analysis of the energy landscape given
in Section \ref{sec7}, it readily holds that
\[
\Phi(\bm{\sigma^{k}},\,\bm{\widehat{\sigma}^{k}})\ge\mathbb{H}_{0}+5.
\]
It would be quite an interesting but technical question to characterize exactly this energy barrier $\Phi(\bm{\sigma^{k}},\,\bm{\widehat{\sigma}^{k}})$, and also to dive into the exact energy landscape lying between these transitions. We decide not to go any further in this direction and leave it as a future research topic.
\end{rem}

\section{\label{sec4}Initial Step}

In this section, we prove the main results for the initial step $h=1$
by the following procedure. In Section \ref{sec4.1}, we provide some
key properties of the limiting Markov chain $\{\mathfrak{X}^{\star,1}(t)\}_{t\ge0}$
in $\mathscr{P}^{\star,1}=\mathscr{P}^{1}$ (cf. \eqref{eq:P-star1 P-sharp1 prop}).
Then, in Section \ref{sec4.2}, we prove Theorems \ref{thm:nu dec}
and \ref{thm:ground states rec} for $h=1$. Finally, in Section \ref{sec4.3},
we prove Theorem \ref{thm:main} for $h=1$.

In this section, we omit the superscripts $1$ in $\mathcal{P}_{i}^{1}$
and $\mathcal{V}_{i}^{1}$.

\subsection{\label{sec4.1}Properties of the limiting Markov chain}

We record some key properties of $\{\mathfrak{X}^{\star,1}(t)\}_{t\ge0}$
in $\mathscr{P}^{1}$. Recall \eqref{eq:P-star1 dec} and \eqref{eq:P-starh C-starh rec}.
In Section \ref{sec4.1}, we prove the following theorem.
\begin{thm}
\label{thm:P-star1 class}The following classification holds.
\begin{enumerate}
\item If $\mathcal{P}_{i}\in\mathscr{P}_{m}^{\star,1}$ for some $m\in[1,\,\nu_{1}]$
with $|\mathscr{P}_{m}^{\star,1}|=1$, then $\Gamma_{i}^{1}>\Gamma^{\star,1}$.
\item If $\mathcal{P}_{i}\in\mathscr{P}_{{\rm tr}}^{\star,1}$, then $\Gamma_{i}^{1}=\Gamma^{\star,1}$
and there exists $\mathcal{P}_{j}\in\mathscr{P}^{1}$ such that $\Phi(\mathcal{P}_{i},\,\mathcal{P}_{j})-\mathbb{H}(\mathcal{P}_{i})=\Gamma^{\star,1}$
and $\mathbb{H}(\mathcal{P}_{j})<\mathbb{H}(\mathcal{P}_{i})$.
\item If $\mathcal{P}_{i}\in\mathscr{P}_{m}^{\star,1}$ for some $m\in[1,\,\nu_{1}]$
with $|\mathscr{P}_{m}^{\star,1}|\ge2$, then $\Gamma_{i}^{1}=\Gamma^{\star,1}$
and $\mathbb{H}(\mathcal{P}_{j})=\mathbb{H}(\mathcal{P}_{i})$ for
all $\mathcal{P}_{j}\in\mathscr{P}^{1}$ such that $\Phi(\mathcal{P}_{i},\,\mathcal{P}_{j})-\mathbb{H}(\mathcal{P}_{i})=\Gamma^{\star,1}$.
Moreover, in this case,
\[
\mathscr{P}_{m}^{\star,1}\setminus\{\mathcal{P}_{i}\}=\{\mathcal{P}_{j}\in\mathscr{P}^{1}:\,\Phi(\mathcal{P}_{i},\,\mathcal{P}_{j})-\mathbb{H}(\mathcal{P}_{i})=\Gamma^{\star,1}\}.
\]
\end{enumerate}
\end{thm}

We start with two lemmas.
\begin{lem}
\label{lem:4.1-1}For $i\ne j$ such that $\mathfrak{R}^{\star,1}(\mathcal{P}_{i},\,\mathcal{P}_{j})>0$,
we have $\mathbb{H}(\mathcal{P}_{i})\ge\mathbb{H}(\mathcal{P}_{j})$.
Moreover, if $\mathbb{H}(\mathcal{P}_{i})=\mathbb{H}(\mathcal{P}_{j})$
then we also have $\mathfrak{R}^{\star,1}(\mathcal{P}_{j},\,\mathcal{P}_{i})>0$.
\end{lem}

\begin{proof}
By \eqref{eq:RC def} and \eqref{eq:R-Cstar def}, $\mathfrak{R}^{\star,1}(\mathcal{P}_{i},\,\mathcal{P}_{j})>0$
implies that $\Gamma_{i}^{1}=\Gamma^{\star,1}$ and there exists a
sequence $\mathcal{P}_{i},\,\omega_{1},\,\dots,\,\omega_{N},\,\mathcal{P}_{j}$
along which $\mathfrak{R}^{1}(\cdot,\,\cdot)$ is positive, where
$\omega_{1},\,\dots,\,\omega_{N}\in\Delta^{1}$ . Then, by \eqref{eq:RC def}
we calculate
\[
\mathbb{H}(\mathcal{P}_{i})+\Gamma^{\star,1}=\mathbb{H}(\omega_{1})\ge\cdots\ge\mathbb{H}(\omega_{N})\ge\mathbb{H}(\mathcal{P}_{j})+\Gamma_{j}^{1}\ge\mathbb{H}(\mathcal{P}_{j})+\Gamma^{\star,1}.
\]
Thus, we obtain that $\mathbb{H}(\mathcal{P}_{i})\ge\mathbb{H}(\mathcal{P}_{j})$.

Further suppose that $\mathbb{H}(\mathcal{P}_{i})=\mathbb{H}(\mathcal{P}_{j})$.
Then, all equality holds in the previous display, thus $\Gamma_{j}^{1}=\Gamma^{\star,1}$,
$\mathbb{H}(\omega_{N})=\mathbb{H}(\mathcal{P}_{j})+\Gamma_{j}^{1}$
and $\mathbb{H}(\omega_{1})=\cdots=\mathbb{H}(\omega_{N})$. We then
deduce that $\mathfrak{R}^{1}(\cdot,\,\cdot)>0$ along the reversed
sequence $\mathcal{P}_{j},\,\omega_{N},\,\dots,\,\omega_{1},\,\mathcal{P}_{i}$,
thus $\mathfrak{R}^{\star,1}(\mathcal{P}_{j},\,\mathcal{P}_{i})>0$.
\end{proof}
A path $(\omega_{n})_{n=0}^{N}$ is a \emph{downhill} path if $\mathbb{H}(\omega_{n+1})\le\mathbb{H}(\omega_{n})$
for all $n\in[0,\,N-1]$.
\begin{lem}
\label{lem:4.1-2}It holds that $\Gamma_{i}^{1}=\Gamma^{\star,1}$
if and only if there exists $j\ne i$ with $\mathfrak{R}^{\star,1}(\mathcal{P}_{i},\,\mathcal{P}_{j})>0$.
\end{lem}

\begin{proof}
For the if part, if $\mathfrak{R}^{\star,1}(\mathcal{P}_{i},\,\mathcal{P}_{j})>0$
then \eqref{eq:RC def} and \eqref{eq:R-Cstar def} readily imply
that $\Gamma_{i}^{1}=\Gamma^{\star,1}$.

For the only if part, fix $\mathcal{P}_{i}\in\mathscr{P}^{1}$ with
$\Gamma_{i}^{1}=\Gamma^{\star,1}$. For notational convenience, we
abbreviate $\mathbb{H}_{i}^{\star,1}:=\mathbb{H}(\mathcal{P}_{i})+\Gamma^{\star,1}$.
By Lemma \ref{lem:Vi1 cycle}, $\mathbb{H}(\xi)=\mathbb{H}_{i}^{\star,1}$
for all $\xi\in\partial^{\star}\mathcal{V}_{i}$. Now, define
\[
\mathcal{A}_{i}:=\{\eta\in\Omega\setminus\mathcal{V}_{i}:\,\text{there exists a downhill path from}\quad\partial^{\star}\mathcal{V}_{i}\quad\text{to}\quad\eta\}.
\]
By definition, $\emptyset\ne\partial^{\star}\mathcal{V}_{i}\subseteq\mathcal{A}_{i}$
and $\mathbb{H}(\eta)\le\mathbb{H}_{i}^{\star,1}$ for all $\eta\in\mathcal{A}_{i}$.
We claim that $\mathcal{P}_{j}\subseteq\mathcal{A}_{i}$ for some
$j\ne i$. Indeed, first suppose that $\mathbb{H}(\eta)=\mathbb{H}_{i}^{\star,1}$
for all $\eta\in\mathcal{A}_{i}$. This implies that $\mathcal{V}_{i}\cup\mathcal{A}_{i}=\{\xi\in\Omega:\,\Phi(\mathcal{P}_{i},\,\xi)-\mathbb{H}(\mathcal{P}_{i})\le\Gamma^{\star,1}\}$
and $(\mathcal{V}_{i}\cup\mathcal{A}_{i})\cap\breve{\mathcal{P}}_{i}^{1}=\emptyset$,
which contradict the definition of $\Gamma_{i}^{1}$. Thus, it holds
that $\mathbb{H}(\mathcal{F}(\mathcal{A}_{i}))<\mathbb{H}_{i}^{\star,1}$.
Decomposing $\mathcal{F}(\mathcal{A}_{i})$ into connected components
$\mathcal{A}_{1}'\cup\cdots\cup\mathcal{A}_{N}'$, each $\mathcal{A}_{n}'$
is a stable plateau since $\mathbb{H}(\zeta)>\mathbb{H}(\mathcal{F}(\mathcal{A}_{i}))$
for all $\zeta\in\partial\mathcal{A}_{n}'$ by the definition of $\mathcal{A}_{i}$.
This proves the claim.

The above claim guarantees the existence of a downhill path $\omega=(\omega_{n})_{n=0}^{N}$
such that $\omega_{0}\in\partial^{\star}\mathcal{V}_{i}$, $\omega_{1},\,\dots,\,\omega_{N-1}\in\Delta^{1}$
and $\omega_{N}\in\mathcal{V}_{j}$. Then, by \eqref{eq:RC def},
$\mathfrak{R}^{1}(\cdot,\,\cdot)$ is positive along $\mathcal{P}_{i}\to\omega_{0}\to\cdots\to\omega_{N-1}\to\mathcal{P}_{j}$,
thus ${\bf P}_{\omega_{0}}^{1}[\mathcal{T}_{\mathcal{P}_{j}}=\mathcal{T}_{\mathscr{P}^{1}}]>0$.
Therefore, by \eqref{eq:R-Cstar def},
\[
\mathfrak{R}^{\star,1}(\mathcal{P}_{i},\,\mathcal{P}_{j})\ge\mathfrak{R}^{1}(\mathcal{P}_{i},\,\omega_{0})\cdot{\bf P}_{\omega_{0}}^{1}[\mathcal{T}_{\mathcal{P}_{j}}=\mathcal{T}_{\mathscr{P}^{1}}]>0,
\]
which completes the proof of Lemma \ref{lem:4.1-2}.
\end{proof}
Now, for each $i\in[1,\,\nu_{0}]$ with $\Gamma_{i}^{1}=\Gamma^{\star,1}$,
define
\begin{equation}
\mathscr{Z}_{i}:=\{\mathcal{P}_{j}\in\mathscr{P}^{1}\setminus\{\mathcal{P}_{i}\}:\,\mathcal{P}_{j}\quad\text{is reachable from}\quad\mathcal{P}_{i}\quad\text{via}\quad\{\mathfrak{X}^{\star,1}(t)\}_{t\ge0}\}.\label{eq:Zi def}
\end{equation}
Here, $\mathscr{Z}_{i}$ is nonempty by Lemma \ref{lem:4.1-2} and
$\mathbb{H}(\mathcal{P}_{j})\le\mathbb{H}(\mathcal{P}_{i})$ for all
$\mathcal{P}_{j}\in\mathscr{Z}_{i}$ by Lemma \ref{lem:4.1-1}. In
particular, by an inductive argument,
\begin{equation}
\Phi(\mathcal{P}_{i},\,\mathcal{P}_{j})-\mathbb{H}(\mathcal{P}_{i})=\Gamma^{\star,1}\quad\text{for all}\quad\mathcal{P}_{j}\in\mathscr{Z}_{i}.\label{eq:Zi prop}
\end{equation}

\begin{lem}
\label{lem:4.1-3}Suppose that $\Gamma_{i}^{1}=\Gamma^{\star,1}$
and $\mathbb{H}(\mathcal{P}_{j})<\mathbb{H}(\mathcal{P}_{i})$ for
some $\mathcal{P}_{j}\in\mathscr{Z}_{i}$. Then, $\mathcal{P}_{i}\in\mathscr{P}_{{\rm tr}}^{\star,1}$.
\end{lem}

\begin{proof}
$\mathcal{P}_{j}$ is reachable from $\mathcal{P}_{i}$ via $\{\mathfrak{X}^{\star,1}(t)\}_{t\ge0}$
since $\mathcal{P}_{j}\in\mathscr{Z}_{i}$, but $\mathcal{P}_{i}$
is not reachable from $\mathcal{P}_{j}$ by Lemma \ref{lem:4.1-1}.
Thus, we readily conclude that $\mathcal{P}_{i}\in\mathscr{P}_{{\rm tr}}^{\star,1}$.
\end{proof}
\begin{lem}
\label{lem:4.1-4}Suppose that $\Gamma_{i}^{1}=\Gamma^{\star,1}$
and $\mathbb{H}(\mathcal{P}_{j})=\mathbb{H}(\mathcal{P}_{i})$ for
all $\mathcal{P}_{j}\in\mathscr{Z}_{i}$. Then, $\mathcal{P}_{i}\in\mathscr{P}_{m}^{\star,1}$
for some $m\in[1,\,\nu_{1}]$ with $|\mathscr{P}_{m}^{\star,1}|\ge2$.
In this case, for each $\mathcal{P}_{j'}\in\mathscr{P}^{1}$ such
that $\Phi(\mathcal{P}_{i},\,\mathcal{P}_{j'})-\mathbb{H}(\mathcal{P}_{i})=\Gamma^{\star,1}$,
it holds that $\mathcal{P}_{j'}\in\mathscr{Z}_{i}$.
\end{lem}

\begin{proof}
By Lemma \ref{lem:4.1-1}, from each $\mathcal{P}_{j}\in\mathscr{Z}_{i}$
it is possible to return to $\mathcal{P}_{i}$ via $\mathfrak{R}^{\star,1}(\cdot,\,\cdot)$.
Thus, $\{\mathcal{P}_{i}\}\cup\mathscr{Z}_{i}$ forms an irreducible
class with respect to $\{\mathfrak{X}^{\star,1}(t)\}_{t\ge0}$. This
implies that $\{\mathcal{P}_{i}\}\cup\mathscr{Z}_{i}=\mathscr{P}_{m}^{\star,1}$
for some $m\in[1,\,\nu_{0}]$ which implies that $|\mathscr{P}_{m}^{\star,1}|\ge2$.
This proves the first statement.

Moreover, take another $\mathcal{P}_{j'}\in\mathscr{P}^{1}$ such
that $\Phi(\mathcal{P}_{i},\,\mathcal{P}_{j'})-\mathbb{H}(\mathcal{P}_{i})=\Gamma^{\star,1}$.
Then, there exist a sequence $\mathcal{P}_{i}=\mathcal{P}_{i_{0}},\,\mathcal{P}_{i_{1}},\,\dots,\,\mathcal{P}_{i_{N}}=\mathcal{P}_{j'}$
in $\mathscr{P}^{1}$ and configurations $\omega_{n,n'}$ in $\Delta^{1}$
such that for each $n\in[0,\,N-1]$, there exists a path
\begin{equation}
\mathcal{V}_{i_{n}}\to\omega_{n,1}\to\cdots\to\omega_{n,M_{n}}\to\mathcal{V}_{i_{n+1}}\quad\text{with height at most}\quad\mathbb{H}_{i}^{\star,1}:=\mathbb{H}(\mathcal{P}_{i})+\Gamma^{\star,1}.\label{eq:4.1-4 pf}
\end{equation}
We claim that $\mathcal{P}_{i_{n}}$ is reachable from $\mathcal{P}_{i}$
by $\{\mathfrak{X}^{\star,1}(t)\}_{t\ge0}$ and $\mathbb{H}(\mathcal{P}_{i_{n}})=\mathbb{H}(\mathcal{P}_{i})$
for each $n$, which concludes the proof of the second statement of
Lemma \ref{lem:4.1-4} by substituting $n=N$. Since the claim is
obvious for $n=0$, we verify the claim by proving the inductive step
$n\to n+1$.

Since $\mathbb{H}(\mathcal{P}_{i_{n}})=\mathbb{H}(\mathcal{P}_{i})$,
$\Gamma_{i_{n}}^{1}\ge\Gamma^{\star,1}$ and the path in \eqref{eq:4.1-4 pf}
has height at most $\mathbb{H}_{i}^{\star,1}$, it readily holds that
$\Gamma_{i_{n}}^{1}=\Gamma^{\star,1}$ and $\mathbb{H}(\omega_{n,1})=\mathbb{H}_{i}^{\star,1}$.
Moreover, suppose that there exists $n'$ such that $\mathbb{H}(\omega_{n,n'})<\mathbb{H}_{i}^{\star,1}$.
Take minimal such $n'$ so that $\mathbb{H}(\omega_{n,1})=\cdots=\mathbb{H}(\omega_{n,n'-1})=\mathbb{H}_{i}^{\star,1}$
and $\mathbb{H}(\omega_{n,n'})<\mathbb{H}_{i}^{\star,1}$. Then, the
fact that $\omega_{n,n'}\in\Delta^{1}$ implies that there exists
a downhill path from $\omega_{n,n'}$ to some other $\mathcal{P}_{j''}\in\mathscr{P}^{1}$.
From this we deduce that $\mathbb{H}(\omega_{n,n'})\ge\mathbb{H}(\mathcal{P}_{j''})+\Gamma_{j''}^{1}\ge\mathbb{H}(\mathcal{P}_{j''})+\Gamma^{\star,1}$.
Thus,
\[
\mathbb{H}_{i}^{\star,1}=\mathbb{H}(\omega_{n,1})>\mathbb{H}(\omega_{n,n'})\ge\mathbb{H}(\mathcal{P}_{j''})+\Gamma^{\star,1},\quad\text{thus}\quad\mathbb{H}(\mathcal{P}_{i})>\mathbb{H}(\mathcal{P}_{j''}).
\]
Moreover, there exists a downhill path $\omega_{n,1}\to\cdots\to\omega_{n,n'}\to\cdots\to\mathcal{V}_{j''}$,
thus we have $\mathfrak{R}^{\star,1}(\mathcal{P}_{i_{n}},\,\mathcal{P}_{j''})>0$
and $\mathcal{P}_{j''}\in\mathscr{Z}_{i}$ by the induction hypothesis.
This contradicts the original assumption in the lemma. Hence, $\mathbb{H}(\omega_{n,n'})=\mathbb{H}_{i}^{\star,1}$
for all $n'\in[1,\,M_{n}]$. Finally, if $\mathbb{H}(\mathcal{P}_{i_{n+1}})<\mathbb{H}(\mathcal{P}_{i})$
then the same argument implies that $\mathcal{P}_{i_{n+1}}\in\mathscr{Z}_{i}$
and thus we again obtain a contradiction. Therefore, we conclude that
$\mathbb{H}(\mathcal{P}_{i_{n+1}})=\mathbb{H}(\mathcal{P}_{i})$ and
also $\mathcal{P}_{i_{n+1}}\in\mathscr{Z}_{i}$, which completes the
induction statement.
\end{proof}
Now, we are ready to present a proof of Theorem \ref{thm:P-star1 class}.
\begin{proof}[Proof of Theorem \ref{thm:P-star1 class}]
 By Lemma \ref{lem:4.1-2}, it holds that $\Gamma_{i}^{1}>\Gamma^{\star,1}$
if and only if $\mathcal{P}_{i}$ is an absorbing state of $\{\mathfrak{X}^{\star,1}(t)\}_{t\ge0}$,
which is equivalent to $\mathcal{P}_{i}\in\mathscr{P}_{m}^{\star,1}$
for some $m\in[1,\,\nu_{1}]$ with $|\mathscr{P}_{m}^{\star,1}|=1$.
Thus, Theorem \ref{thm:P-star1 class}-(1) is readily verified.

Next, suppose that $\mathcal{P}_{i}\in\mathscr{P}_{{\rm tr}}^{\star,1}$.
Then, Lemmas \ref{lem:4.1-2}, \ref{lem:4.1-3} and \ref{lem:4.1-4}
constitute the full characterization of $\mathscr{P}^{1}$ such that
for $\mathcal{P}_{i}\in\mathscr{P}_{{\rm tr}}^{\star,1}$, we have
$\Gamma_{i}^{1}=\Gamma^{\star,1}$ and $\mathbb{H}(\mathcal{P}_{j})<\mathbb{H}(\mathcal{P}_{i})$
for some $\mathcal{P}_{j}\in\mathscr{Z}_{i}$, where \eqref{eq:Zi prop}
implies that $\Phi(\mathcal{P}_{i},\,\mathcal{P}_{j})-\mathbb{H}(\mathcal{P}_{i})=\Gamma^{\star,1}$.
Thus, Theorem \ref{thm:P-star1 class}-(2) is proved.

Finally, suppose that $\mathcal{P}_{i}\in\mathscr{P}_{m}^{\star,1}$
for some $m\in[1,\,\nu_{1}]$ with $|\mathscr{P}_{m}^{\star,1}|\ge2$.
Then, by the same classification, for each $\mathcal{P}_{j}\in\mathscr{P}^{1}$
with $\Phi(\mathcal{P}_{i},\,\mathcal{P}_{j})-\mathbb{H}(\mathcal{P}_{i})=\Gamma^{\star,1}$,
the second statement of Lemma \ref{lem:4.1-4} implies that $\mathcal{P}_{j}\in\mathscr{Z}_{i}$
and also $\mathbb{H}(\mathcal{P}_{j})=\mathbb{H}(\mathcal{P}_{i})$.
All that remains to be proved is that each $\mathcal{P}_{j'}\in\mathscr{P}_{m}^{\star,1}\setminus\{\mathcal{P}_{i}\}$
satisfies $\Phi(\mathcal{P}_{i},\,\mathcal{P}_{j'})-\mathbb{H}(\mathcal{P}_{i})=\Gamma^{\star,1}$
and $\mathbb{H}(\mathcal{P}_{j'})=\mathbb{H}(\mathcal{P}_{i})$, which
is again clear by \eqref{eq:Zi prop} and Lemma \ref{lem:4.1-4}.
This proves Theorem \ref{thm:P-star1 class}-(3).
\end{proof}

\subsection{\label{sec4.2}Proof of Theorems \ref{thm:nu dec} and \ref{thm:ground states rec}}

Now, according to the precise analysis conducted in the previous subsection,
here we prove Theorems \ref{thm:nu dec} and \ref{thm:ground states rec}
for $h=1$.

It readily holds by \eqref{eq:C-star C-sharp def} and Lemma \ref{lem:4.1-2}
that $\{\mathfrak{X}^{\star,1}(t)\}_{\ge0}$ is a nonzero Markov chain
in $\mathscr{P}^{1}$. This readily verifies via \eqref{eq:P-star1 dec}
that $\nu_{1}<\nu_{0}$, thus Theorem \ref{thm:nu dec} follows.

To prove Theorem \ref{thm:ground states rec} for $h=1$, suppose
the contrary. Then, there exist $\bm{s}\in\mathcal{S}$ and $\mathcal{P}_{i}\in\mathscr{P}_{{\rm tr}}^{\star,1}$
such that $\bm{s}\in\mathcal{P}_{i}$. Then, by Theorem \ref{thm:P-star1 class}-(2),
there exists another $\mathcal{P}_{j}\in\mathscr{P}^{1}$ such that
$\mathbb{H}(\mathcal{P}_{j})<\mathbb{H}(\mathcal{P}_{i})$. This contradicts
the fact that $\mathcal{S}$ is the collection of minimizers of the
Hamiltonian. This completes the proof of Theorem \ref{thm:ground states rec}
for $h=1$.

\subsection{\label{sec4.3}Proof of Theorem \ref{thm:main}}

In this subsection, we prove Theorem \ref{thm:main} in the case of
$h=1$. Main tool is the resolvent approach to metastability \cite{LMS resolvent}.

We denote by $\mathfrak{L}^{1}$ the infinitesimal generator of the
limiting dynamics $\{\mathfrak{X}^{\star,1}(t)\}_{t\ge0}$ in $\mathscr{P}^{1}$.
For any $\lambda>0$ and $g:\,\mathscr{P}^{1}\to\mathbb{R}$, denote
by $f:\,\mathscr{P}^{1}\to\mathbb{R}$ the unique solution to the
\emph{macroscopic} resolvent equation: $(\lambda-\mathfrak{L}^{1})f=g$.
Denote by $G:\,\Omega\to\mathbb{R}$ the \emph{lift} of $g$ given
as (cf. \eqref{eq:V-starh def} and \eqref{eq:Psih def})
\begin{equation}
G(\eta):=\begin{cases}
g(\Psi^{1}(\eta)) & \text{if}\quad\eta\in\mathcal{V}^{\star,1},\\
0 & \text{if}\quad\eta\in\Omega\setminus\mathcal{V}^{\star,1}.
\end{cases}\label{eq:G def}
\end{equation}
Then, for each positive real number $\Gamma$, denote by $F_{\beta}^{\Gamma}:\,\Omega\to\mathbb{R}$
the unique solution to the \emph{microscopic} resolvent equation,
where $L_{\beta}$ denotes the original infinitesimal generator:
\begin{equation}
(\lambda-e^{\Gamma\beta}L_{\beta})F_{\beta}^{\Gamma}=G.\label{eq:res micro}
\end{equation}
The explicit formulas for the resolvent solutions $f$ and $F_{\beta}^{\Gamma}$
are given as (cf. \cite[(4.1)]{LMS resolvent})
\[
f(\mathcal{P}_{i})={\bf E}_{\mathcal{P}_{i}}^{\star,1}\Big[\int_{0}^{\infty}e^{-\lambda t}g(\mathfrak{X}^{\star,1}(t)){\rm d}t\Big]\quad\text{and}\quad F_{\beta}^{\Gamma}(\eta)=\mathbb{E}_{\eta}\Big[\int_{0}^{\infty}e^{-\lambda t}G(\eta_{\beta}(e^{\Gamma\beta}t)){\rm d}t\Big],
\]
where ${\bf E}_{\cdot}^{\star,1}$ is the expectation of the law ${\bf P}_{\cdot}^{\star,1}$
of $\{\mathfrak{X}^{\star,1}(t)\}_{t\ge0}$. In particular, $f$ and
$F_{\beta}^{\Gamma}$ are uniformly bounded:
\begin{equation}
\|f\|_{\infty}\le\frac{1}{\lambda}\|g\|_{\infty}\quad\text{and}\quad\|F_{\beta}^{\Gamma}\|_{\infty}\le\frac{1}{\lambda}\|g\|_{\infty},\label{eq:f Fbeta unif bdd}
\end{equation}
where $\|\cdot\|_{\infty}$ denotes the supremum norm. By \cite[Theorem 2.3]{LMS resolvent},
proving Theorem \ref{thm:main} is equivalent to proving that
\begin{equation}
\lim_{\beta\to\infty}\sup_{\eta\in\mathcal{V}_{i}}|F_{\beta}^{\star,1}(\eta)-f(\mathcal{P}_{i})|=0\quad\text{for each}\quad i\in[1,\,\nu_{0}],\label{eq:res cond}
\end{equation}
where we abbreviated $F_{\beta}^{\star,1}:=F_{\beta}^{\Gamma^{\star,1}}$.

For $\Gamma>0$ and nonempty $\mathcal{A}\subseteq\Omega$, denote
by $\overline{F}_{\beta}^{\Gamma}(\mathcal{A})$ the average of $F_{\beta}^{\Gamma}$
in $\mathcal{A}$:
\begin{equation}
\overline{F}_{\beta}^{\Gamma}(\mathcal{A}):=\sum_{\xi\in\mathcal{A}}\frac{\mu_{\beta}(\xi)}{\mu_{\beta}(\mathcal{A})}F_{\beta}^{\Gamma}(\xi).\label{eq:Fbeta-bar def}
\end{equation}
First, we prove that for every $\Gamma>0$, $F_{\beta}^{\Gamma}$
is asymptotically constant in $\mathcal{V}_{i}$ for each $i\in[1,\,\nu_{0}]$.
\begin{lem}
\label{lem:Fbeta flat valley-1}For all $\Gamma>0$ and $i\in[1,\,\nu_{0}]$,
it holds that
\[
\lim_{\beta\to\infty}\sup_{\eta\in\mathcal{V}_{i}}|F_{\beta}^{\Gamma}(\eta)-\overline{F}_{\beta}^{\Gamma}(\mathcal{P}_{i})|=0.
\]
\end{lem}

\begin{proof}
It suffices to prove that
\begin{equation}
\lim_{\beta\to\infty}\sup_{\eta,\,\xi\in\mathcal{V}_{i}}|F_{\beta}^{\Gamma}(\eta)-F_{\beta}^{\Gamma}(\xi)|=0.\label{eq:claim-1}
\end{equation}
By \cite[Proposition 6.7]{LMS resolvent}, it suffices to check the
mixing condition (Condition $\mathfrak{M}$ therein):\footnote{In this article, $f(\beta)=o(g(\beta))$ or $f(\beta)\ll g(\beta)$
indicates that $\lim_{\beta\to\infty}\frac{f(\beta)}{g(\beta)}=0$
and $f(\beta)=O(g(\beta))$ indicates that $|f(\beta)|\le Cg(\beta)$
for some constant $C>0$ independent of $\beta$.}\smallskip{}
\begin{itemize}
\item There exists $\theta_{\beta}\ll e^{\Gamma\beta}$ such that $\lim_{\beta\to\infty}\sup_{\eta\in\mathcal{V}_{i}}\mathbb{P}_{\eta}[\mathcal{T}_{\mathcal{V}_{i}^{c}}\le\theta_{\beta}]=0$.
\item The reflected process in $\mathcal{V}_{i}$ is ergodic and for each
$\epsilon>0$ its $\epsilon$-mixing time $t_{{\rm mix}}(\epsilon)$
is bounded by $\theta_{\beta}$ for all sufficiently large $\beta$.\smallskip{}
\end{itemize}
Take $\theta_{\beta}:=e^{\Gamma'\beta}$ where $\Gamma'\in(0,\,\Gamma)$.
Then since $\Gamma'<\Gamma$, it is clear that $\theta_{\beta}\ll e^{\Gamma\beta}$.
By \cite[Theorem 6.23-(i)]{OV}, $\lim_{\beta\to\infty}\mathbb{P}_{\eta}[\mathcal{T}_{\mathcal{V}_{i}^{c}}\le\theta_{\beta}]=0$
uniformly over $\eta\in\mathcal{V}_{i}$, thus the first condition
is verified. Moreover, since $\mathcal{V}_{i}$ is connected it is
clear that the reflected process in $\mathcal{V}_{i}$ is ergodic,
and by \cite[Proposition 3.24 and Lemma 3.6]{NZB}, it holds that
$\lim_{\beta\to\infty}\beta^{-1}\log t_{{\rm mix}}(\epsilon)=\widetilde{\Gamma}(\mathcal{V}_{i})$
where
\begin{equation}
\widetilde{\Gamma}(\mathcal{V}_{i}):=\max_{\eta\in\mathcal{V}_{i}\setminus\{\eta_{0}\}}(\Phi(\eta,\,\eta_{0})-\mathbb{H}(\eta))\quad\text{for any fixed}\quad\eta_{0}\in\mathcal{F}(\mathcal{V}_{i})=\mathcal{P}_{i}.\label{eq:Gamma-tilde def}
\end{equation}
If $\eta\in\mathcal{P}_{i}\setminus\{\eta_{0}\}$ then clearly $\Phi(\eta,\,\eta_{0})-\mathbb{H}(\eta)=0$,
and if $\eta\in\mathcal{V}_{i}\setminus\mathcal{P}_{i}$ then there
exists a downhill path from $\eta$ to $\mathcal{P}_{i}$. Along this
path, we obtain that $\Phi(\eta,\,\eta_{0})-\mathbb{H}(\eta)=0$.
Thus, $\widetilde{\Gamma}(\mathcal{V}_{i})=0$ and $t_{{\rm mix}}(\epsilon)\le e^{\beta\Gamma'}=\theta_{\beta}$
for all sufficiently large $\beta$, which verifies the second condition
and proves \eqref{eq:claim-1}.
\end{proof}
The next step is to prove that the value of $F_{\beta}^{\Gamma}(\eta)$
for $\eta\in\Delta^{1}$ can be approximated by a linear combination
of $\overline{F}_{\beta}^{\Gamma}(\mathcal{P}_{i})$ for $i\in[1,\,\nu_{0}]$.
\begin{lem}
\label{lem:Fbeta Delta1}For all $\Gamma>0$ and $\eta\in\Delta^{1}$,
it holds that
\[
\lim_{\beta\to\infty}\Big|F_{\beta}^{\Gamma}(\eta)-\sum_{i\in[1,\,\nu_{0}]}{\bf P}_{\eta}^{1}[\mathcal{T}_{\mathcal{P}_{i}}=\mathcal{T}_{\mathscr{P}^{1}}]\cdot\overline{F}_{\beta}^{\Gamma}(\mathcal{P}_{i})\Big|=0.
\]
\end{lem}

\begin{proof}
Fix $\Gamma'\in(0,\,\Gamma)$. We may rewrite $F_{\beta}^{\Gamma}(\eta)$
as
\[
F_{\beta}^{\Gamma}(\eta)=\mathbb{E}_{\eta}\Big[\int_{0}^{\infty}e^{-\lambda t}G(\eta_{\beta}(e^{\Gamma\beta}t)){\rm d}t\Big]=e^{-\Gamma\beta}\cdot\mathbb{E}_{\eta}\Big[\int_{0}^{\infty}e^{-\lambda e^{-\Gamma\beta}s}G(\eta_{\beta}(s)){\rm d}s\Big].
\]
Divide the integrand in the right-hand side into regimes $\mathcal{T}_{\mathcal{V}^{\star,1}}>e^{\Gamma'\beta}$
and $\mathcal{T}_{\mathcal{V}^{\star,1}}\le e^{\Gamma'\beta}$. First,
observe that starting from $\eta\in\Delta^{1}$, following the original
dynamics $\{\eta_{\beta}(t)\}_{t\ge0}$ but neglecting the $O(e^{-\beta})$-transition
rates, the holding rates along the trajectory are of order $1$ until
it arrives at $\mathcal{V}^{\star,1}$, and in particular the hitting
time of $\mathcal{V}^{\star,1}$ is also of order $1$ with probability
tending to $1$. Then, by a standard coupling argument it holds that
$\lim_{\beta\to\infty}\mathbb{P}_{\eta}[\mathcal{T}_{\mathcal{V}^{\star,1}}>e^{\Gamma'\beta}]=0$,
thus as $\beta\to\infty$,
\[
e^{-\Gamma\beta}\cdot\mathbb{E}_{\eta}\Big[\int_{0}^{\infty}e^{-\lambda e^{-\Gamma\beta}s}G(\eta_{\beta}(s)){\bf 1}\{\mathcal{T}_{\mathcal{V}^{\star,1}}>e^{\Gamma'\beta}\}{\rm d}s\Big]\le\lambda^{-1}\|g\|_{\infty}\cdot\mathbb{P}_{\eta}[\mathcal{T}_{\mathcal{V}^{\star,1}}>e^{\Gamma'\beta}]\to0.
\]
In turn, consider the remaining part associated to $\mathcal{T}_{\mathcal{V}^{\star,1}}\le e^{\Gamma'\beta}$.
We divide the integral $\int_{0}^{\infty}$ into $\int_{0}^{\mathcal{T}_{\mathcal{V}^{\star,1}}}$
and $\int_{\mathcal{T}_{\mathcal{V}^{\star,1}}}^{\infty}$. The first
part can be estimated as
\[
e^{-\Gamma\beta}\cdot\mathbb{E}_{\eta}\Big[\int_{0}^{\mathcal{T}_{\mathcal{V}^{\star,1}}}e^{-\lambda e^{-\Gamma\beta}s}G(\eta_{\beta}(s)){\bf 1}\{\mathcal{T}_{\mathcal{V}^{\star,1}}\le e^{\Gamma'\beta}\}{\rm d}s\Big]\le e^{-\Gamma\beta}\cdot\|g\|_{\infty}e^{\Gamma'\beta}\to0=o(1).
\]
By the strong Markov property at $\mathcal{T}_{\mathcal{V}^{\star,1}}$
and the definition of $F_{\beta}^{\Gamma}$, the second part equals
\begin{align*}
 & e^{-\Gamma\beta}\cdot\mathbb{E}_{\eta}\Big[\int_{\mathcal{T}_{\mathcal{V}^{\star,1}}}^{\infty}e^{-\lambda e^{-\Gamma\beta}s}G(\eta_{\beta}(s)){\bf 1}\{\mathcal{T}_{\mathcal{V}^{\star,1}}\le e^{\Gamma'\beta}\}{\rm d}s\Big]\\
 & =\mathbb{E}_{\eta}\big[{\bf 1}\{\mathcal{T}_{\mathcal{V}^{\star,1}}\le e^{\Gamma'\beta}\}F_{\beta}^{\Gamma}(\eta_{\beta}(\mathcal{T}_{\mathcal{V}^{\star,1}}))\big].
\end{align*}
Thus so far, we proved that
\[
F_{\beta}^{\Gamma}(\eta)=\mathbb{E}_{\eta}\big[{\bf 1}\{\mathcal{T}_{\mathcal{V}^{\star,1}}\le e^{\Gamma'\beta}\}F_{\beta}^{\Gamma}(\eta_{\beta}(\mathcal{T}_{\mathcal{V}^{\star,1}}))\big]+o(1).
\]
Since $\mathbb{P}_{\eta}[\mathcal{T}_{\mathcal{V}^{\star,1}}\le e^{\Gamma'\beta}]=1-o(1)$,
it holds that
\begin{equation}
F_{\beta}^{\Gamma}(\eta)=\mathbb{E}_{\eta}[F_{\beta}^{\Gamma}(\eta_{\beta}(\mathcal{T}_{\mathcal{V}^{\star,1}}))]+o(1).\label{eq:Delta1-1}
\end{equation}
Now, we calculate the expectation in the right-hand side. By the same
coupling argument, it is routine to see that (see e.g. \cite[Lemma 4.1]{BL Kawasaki})
\begin{equation}
\lim_{\beta\to\infty}\mathbb{P}_{\eta}[\mathcal{T}_{\mathcal{V}_{i}}=\mathcal{T}_{\mathcal{V}^{\star,1}}]={\bf P}_{\eta}^{1}[\mathcal{T}_{\mathcal{P}_{i}}=\mathcal{T}_{\mathscr{P}^{1}}]\quad\text{for each}\quad i\in[1,\,\nu_{0}].\label{eq:Delta1-2}
\end{equation}
Moreover, we may expand $\mathbb{E}_{\eta}[F_{\beta}^{\Gamma}(\eta_{\beta}(\mathcal{T}_{\mathcal{V}^{\star,1}}))]$
as
\begin{equation}
\sum_{i\in[1,\,\nu_{0}]}\mathbb{E}_{\eta}[{\bf 1}\{\mathcal{T}_{\mathcal{V}_{i}}=\mathcal{T}_{\mathcal{V}^{\star,1}}\}F_{\beta}^{\Gamma}(\eta_{\beta}(\mathcal{T}_{\mathcal{V}_{i}}))]=\sum_{i\in[1,\,\nu_{0}]}\mathbb{P}_{\eta}[\mathcal{T}_{\mathcal{V}_{i}}=\mathcal{T}_{\mathcal{V}^{\star,1}}]\cdot\overline{F}_{\beta}^{\Gamma}(\mathcal{P}_{i})+o(1),\label{eq:Delta1-3}
\end{equation}
where the last equality holds by Lemma \ref{lem:Fbeta flat valley-1}.
Therefore, by \eqref{eq:Delta1-1}, \eqref{eq:Delta1-2} and \eqref{eq:Delta1-3},
we conclude that
\[
F_{\beta}^{\Gamma}(\eta)-\sum_{i\in[1,\,\nu_{0}]}{\bf P}_{\eta}^{1}[\mathcal{T}_{\mathcal{P}_{i}}=\mathcal{T}_{\mathscr{P}^{1}}]\cdot\overline{F}_{\beta}^{\Gamma}(\mathcal{P}_{i})=o(1)\to0.
\]
This concludes the proof.
\end{proof}
Now to prove \eqref{eq:res cond}, by Lemma \ref{lem:Fbeta flat valley-1}
it suffices to prove that
\begin{equation}
\lim_{\beta\to\infty}|\overline{F}_{\beta}^{\star,1}(\mathcal{P}_{i})-f(\mathcal{P}_{i})|=0\quad\text{for each}\quad i,\label{eq:res WTS}
\end{equation}
where $\overline{F}_{\beta}^{\star,1}:=\overline{F}_{\beta}^{\Gamma^{\star,1}}$.
The idea is to prove that
\[
(\lambda-\mathfrak{L}^{1})\overline{F}_{\beta}^{\star,1}=g+o(1)\quad\text{in}\quad\mathscr{P}^{1}.
\]
Provided that the displayed identity holds, we may calculate using
the explicit formula for the resolvent solution as
\[
\overline{F}_{\beta}^{\star,1}(\mathcal{P}_{i})={\bf E}_{\mathcal{P}_{i}}^{\star,1}\Big[\int_{0}^{\infty}e^{-\lambda t}g(\mathfrak{X}^{\star,1}(t)){\rm d}t\Big]+o(1)=f(\mathcal{P}_{i})+o(1),
\]
which proves \eqref{eq:res WTS}. Thus, all that remains to be proved
is that
\begin{equation}
(\lambda-\mathfrak{L}^{1})\overline{F}_{\beta}^{\star,1}(\mathcal{P}_{i})=g(\mathcal{P}_{i})+o(1)\quad\text{for each}\quad i.\label{eq:res asymp}
\end{equation}
We finish the proof of Theorem \ref{thm:main} by verifying \eqref{eq:res asymp}.

We integrate both sides of \eqref{eq:res micro} with $\Gamma=\Gamma^{\star,1}$
with respect to weight $\mu_{\beta}$ in the set $\mathcal{V}_{i}$:
\[
\lambda\sum_{\eta\in\mathcal{V}_{i}}\mu_{\beta}(\eta)F_{\beta}^{\star,1}(\eta)-e^{\Gamma^{\star,1}\beta}\sum_{\eta\in\mathcal{V}_{i}}\mu_{\beta}(\eta)L_{\beta}F_{\beta}^{\star,1}(\eta)=\sum_{\eta\in\mathcal{V}_{i}}\mu_{\beta}(\eta)G(\eta).
\]
By the definition of $G$ in \eqref{eq:G def}, the right-hand side
becomes
\[
\sum_{\eta\in\mathcal{V}_{i}}\mu_{\beta}(\eta)g(\mathcal{P}_{i})=\mu_{\beta}(\mathcal{V}_{i})g(\mathcal{P}_{i}),
\]
thus we obtain that
\begin{equation}
\lambda\sum_{\eta\in\mathcal{V}_{i}}\mu_{\beta}(\eta)F_{\beta}^{\star,1}(\eta)-e^{\Gamma^{\star,1}\beta}\sum_{\eta\in\mathcal{V}_{i}}\mu_{\beta}(\eta)L_{\beta}F_{\beta}^{\star,1}(\eta)=\mu_{\beta}(\mathcal{V}_{i})g(\mathcal{P}_{i}).\label{eq:4.3-1}
\end{equation}
By Lemma \ref{lem:Fbeta flat valley-1}, the first term in the left-hand
side of \eqref{eq:4.3-1} equals
\begin{equation}
\lambda\sum_{\eta\in\mathcal{V}_{i}}\mu_{\beta}(\eta)(\overline{F}_{\beta}^{\star,1}(\mathcal{P}_{i})+o(1))=\lambda\mu_{\beta}(\mathcal{V}_{i})(\overline{F}_{\beta}^{\star,1}(\mathcal{P}_{i})+o(1)).\label{eq:4.3-2}
\end{equation}
The second term in the left-hand side of \eqref{eq:4.3-1} becomes
\[
\sum_{\eta\in\mathcal{V}_{i}}\sum_{\xi\in\mathcal{V}_{i}\cup\partial\mathcal{V}_{i}}e^{\Gamma^{\star,1}\beta}\mu_{\beta}(\eta)r_{\beta}(\eta,\,\xi)(F_{\beta}^{\star,1}(\xi)-F_{\beta}^{\star,1}(\eta)).
\]
The double summation for $\eta,\,\xi\in\mathcal{V}_{i}$ vanishes
since $\mu_{\beta}(\eta)r_{\beta}(\eta,\,\xi)=\mu_{\beta}(\xi)r_{\beta}(\xi,\,\eta)$
by \eqref{eq:det bal}. Thus, this becomes
\begin{equation}
\sum_{\eta\in\mathcal{V}_{i}}\sum_{\xi\in\partial\mathcal{V}_{i}}e^{\Gamma^{\star,1}\beta}\mu_{\beta}(\eta)r_{\beta}(\eta,\,\xi)(F_{\beta}^{\star,1}(\xi)-F_{\beta}^{\star,1}(\eta)).\label{eq:4.3-3}
\end{equation}
For $\eta\in\mathcal{V}_{i}$ and $\xi\in\partial\mathcal{V}_{i}$
with $\eta\sim\xi$, by \eqref{eq:det bal},
\[
\mu_{\beta}(\eta)r_{\beta}(\eta,\,\xi)=\mu_{\beta}(\xi)=|\mathcal{P}_{i}|^{-1}\mu_{\beta}(\mathcal{P}_{i})\cdot e^{-\beta(\mathbb{H}(\xi)-\mathbb{H}(\mathcal{P}_{i}))},
\]
where $\mathbb{H}(\xi)-\mathbb{H}(\mathcal{P}_{i})\ge\Gamma_{i}^{1}\ge\Gamma^{\star,1}$.
Thus, the summand in \eqref{eq:4.3-3} is non-negligible with respect
to $\mu_{\beta}(\mathcal{P}_{i})$ only if $\mathbb{H}(\xi)-\mathbb{H}(\mathcal{P}_{i})=\Gamma_{i}^{1}=\Gamma^{\star,1}$.
Since $\mathcal{F}(\mathcal{V}_{i})=\mathcal{P}_{i}$ by Lemma \ref{lem:Vi1 cycle},
this implies $\mathfrak{R}^{1}(\mathcal{P}_{i},\,\xi)>0$ by \eqref{eq:RC def}.
According to this observation, we may rewrite \eqref{eq:4.3-3} as
\begin{equation}
|\mathcal{P}_{i}|^{-1}\mu_{\beta}(\mathcal{P}_{i}){\bf 1}\{\Gamma_{i}^{1}=\Gamma^{\star,1}\}\sum_{\xi\in\partial^{\star}\mathcal{V}_{i}}\sum_{\eta\in\mathcal{V}_{i}:\,\eta\sim\xi}(F_{\beta}^{\star,1}(\xi)-F_{\beta}^{\star,1}(\eta))+o(\mu_{\beta}(\mathcal{P}_{i})).\label{eq:4.1-4}
\end{equation}
By \eqref{eq:RC def}, the summation part in \eqref{eq:4.1-4} regarding
$F_{\beta}^{\star,1}(\xi)$ becomes
\begin{equation}
|\mathcal{P}_{i}|^{-1}\mu_{\beta}(\mathcal{P}_{i}){\bf 1}\{\Gamma_{i}^{1}=\Gamma^{\star,1}\}\sum_{\xi\in\partial^{\star}\mathcal{V}_{i}}\sum_{\eta\in\mathcal{V}_{i}:\,\eta\sim\xi}F_{\beta}^{\star,1}(\xi)=\mu_{\beta}(\mathcal{P}_{i})\sum_{\xi\in\Delta^{1}}\mathfrak{R}^{1}(\mathcal{P}_{i},\,\xi)F_{\beta}^{\star,1}(\xi).\label{eq:4.1-5}
\end{equation}
By Lemma \ref{lem:Fbeta Delta1}, we may rewrite this as
\[
\mu_{\beta}(\mathcal{P}_{i})\sum_{\xi\in\Delta^{1}}\mathfrak{R}^{1}(\mathcal{P}_{i},\,\xi)\Big(\sum_{j}{\bf P}_{\xi}^{1}[\mathcal{T}_{\mathcal{P}_{j}}=\mathcal{T}_{\mathscr{P}^{1}}]\cdot\overline{F}_{\beta}^{\star,1}(\mathcal{P}_{j})+o(1)\Big).
\]
Recalling \eqref{eq:R-Cstar def}, this becomes
\[
\mu_{\beta}(\mathcal{P}_{i})\sum_{j}\mathfrak{R}^{\star,1}(\mathcal{P}_{i},\,\mathcal{P}_{j})\overline{F}_{\beta}^{\star,1}(\mathcal{P}_{j})+o(\mu_{\beta}(\mathcal{P}_{i})).
\]
Since $F_{\beta}^{\star,1}$ is asymptotically constant in $\mathcal{V}_{i}$
by Lemma \ref{lem:Fbeta flat valley-1}, the summation part in \eqref{eq:4.1-4}
regarding $F_{\beta}^{\star,1}(\eta)$ becomes
\[
\mu_{\beta}(\mathcal{P}_{i})\overline{F}_{\beta}^{\star,1}(\mathcal{P}_{i})\sum_{\xi\in\Delta^{1}}\mathfrak{R}^{1}(\mathcal{P}_{i},\,\xi)+o(\mu_{\beta}(\mathcal{P}_{i})).
\]
Thus, by Lemma \ref{lem:R1 R-star1 eq} stated below, we may summarize
the second term in the left-hand side of \eqref{eq:4.3-1} as
\begin{equation}
e^{\Gamma^{\star,1}\beta}\sum_{\eta\in\mathcal{V}_{i}}\mu_{\beta}(\eta)L_{\beta}F_{\beta}^{\star,1}(\eta)=\mu_{\beta}(\mathcal{P}_{i})\sum_{j}\mathfrak{R}^{\star,1}(\mathcal{P}_{i},\,\mathcal{P}_{j})(\overline{F}_{\beta}^{\star,1}(\mathcal{P}_{j})-\overline{F}_{\beta}^{\star,1}(\mathcal{P}_{i}))+o(\mu_{\beta}(\mathcal{P}_{i})).\label{eq:4.1-6}
\end{equation}
By \eqref{eq:4.3-1}, \eqref{eq:4.3-2} and \eqref{eq:4.1-6} and
noting that $\mu_{\beta}(\mathcal{V}_{i})=(1+o(1))\mu_{\beta}(\mathcal{P}_{i})$
since $\mathcal{F}(\mathcal{V}_{i})=\mathcal{P}_{i}$ by Lemma \ref{lem:Vi1 cycle},
we conclude that
\[
\lambda\overline{F}_{\beta}^{\star,1}(\mathcal{P}_{i})-\sum_{j}\mathfrak{R}^{\star,1}(\mathcal{P}_{i},\,\mathcal{P}_{j})(\overline{F}_{\beta}^{\star,1}(\mathcal{P}_{j})-\overline{F}_{\beta}^{\star,1}(\mathcal{P}_{i}))=g(\mathcal{P}_{i})+o(1).
\]
By the definition of $\mathfrak{L}^{1}$, this is exactly \eqref{eq:res asymp}.
Therefore, to complete the proof of Theorem \ref{thm:main}, we are
left to verify Lemma \ref{lem:R1 R-star1 eq} below.
\begin{lem}
\label{lem:R1 R-star1 eq}For each $i\in[1,\,\nu_{0}]$,
\[
\sum_{\eta\in\Delta^{1}}\mathfrak{R}^{1}(\mathcal{P}_{i},\,\eta)=\sum_{j}\mathfrak{R}^{\star,1}(\mathcal{P}_{i},\,\mathcal{P}_{j}).
\]
\end{lem}

\begin{proof}
Starting from a configuration in $\Delta^{1}$, the process $\{\mathfrak{X}^{1}(t)\}_{t\ge0}$
jumps to configurations with lower or equal energy until it escapes
$\Delta^{1}$. Thus, $\mathcal{T}_{\mathscr{P}^{1}}<\infty$ ${\bf P}_{\eta}^{1}$-almost
surely for any $\eta\in\Delta^{1}$, which guarantees that 
\[
\sum_{j}{\bf P}_{\eta}^{1}[\mathcal{T}_{\mathcal{P}_{j}}=\mathcal{T}_{\mathscr{P}^{1}}]=1.
\]
Thus by \eqref{eq:R-Cstar def}, it holds that
\[
\sum_{j}\mathfrak{R}^{\star,1}(\mathcal{P}_{i},\,\mathcal{P}_{j})=\sum_{\eta\in\Delta^{1}}\mathfrak{R}^{1}(\mathcal{P}_{i},\,\eta)\cdot\sum_{j}{\bf P}_{\eta}^{1}[\mathcal{T}_{\mathcal{P}_{j}}=\mathcal{T}_{\mathscr{P}^{1}}]=\sum_{\eta\in\Delta^{1}}\mathfrak{R}^{1}(\mathcal{P}_{i},\,\eta).
\]
\end{proof}

\section{\label{sec5}Inductive Step}

In Section \ref{sec5}, we fix $h\ge2$ and prove the inductive step
from level $h-1$ to level $h$ of the main results.

\subsubsection*{Inductive hypothesis}

Since we are dealing with the inductive step, we may assume that the
following \emph{a priori} lemma and theorem hold.
\begin{lem}
\label{lem:gen ind hyp}The following statements are valid for each
$h'\in[1,\,h-1]$.
\begin{enumerate}
\item For each $i\in[1,\,\nu_{h'-1}]$, it holds that $\mathbb{H}(\eta)=\mathbb{H}(\xi)$
for all $\eta,\,\xi\in\mathcal{P}_{i}^{h'}$.
\item Collections $\mathcal{V}_{i}^{h'}$ for $i\in[1,\,\nu_{h'-1}]$ are
disjoint cycles in $\overline{\Omega}$ with depth $\Gamma_{i}^{h'}$
and bottom $\mathcal{P}_{i}^{h'}$.
\item For every $\mathcal{P}\in\mathscr{P}^{1}$, there exists $\mathcal{C}\in\mathscr{C}^{h'}$
such that $\mathcal{P}\subseteq\mathcal{C}$.
\item It holds that $\mathscr{C}^{\star,h'}=\{\mathcal{V}_{i}^{h'}:\,i\in[1,\,\nu_{h'-1}]\}$
and
\[
\mathscr{C}^{\sharp,h'}=\{\mathcal{C}\in\mathscr{C}_{{\rm tr}}^{\star,h'-1}\cup\mathscr{C}^{\sharp,h'-1}:\,\mathcal{C}\cap\mathcal{V}_{i}^{h'}=\emptyset\quad\text{for all}\quad i\in[1,\,\nu_{h'-1}]\}.
\]
\item For all $i\in[1,\,\nu_{h'-1}]$ and $\eta,\,\eta'\in\mathcal{P}_{i}^{h'}$,
it holds that $\Phi(\eta,\,\eta')-\mathbb{H}(\mathcal{P}_{i}^{h'})\le\Gamma^{\star,h'-1}$.
\item It holds that $\Gamma^{\star,h'}>\Gamma^{\star,h'-1}$, where $\Gamma^{\star,0}:=0$.
\item It holds that $\widetilde{\Gamma}(\mathcal{V}_{i}^{h'})=\Gamma^{\star,h'-1}$
(cf. \eqref{eq:Gamma-tilde def}).
\item For $\Gamma>\Gamma^{\star,h'-1}$ and $i\in[1,\,\nu_{h'-1}]$, it
holds that $\lim_{\beta\to\infty}\sup_{\eta\in\mathcal{V}_{i}^{h'}}|F_{\beta}^{\Gamma}(\eta)-\overline{F}_{\beta}^{\Gamma}(\mathcal{P}_{i}^{h'})|=0$.
\item For every $v\in\Delta^{h'}\cap\mathscr{P}^{\sharp,h'}$, it holds
that $\mathcal{T}_{\mathscr{P}^{\star,h'}}<\infty$ ${\bf P}_{v}^{h'}$-almost
surely.
\end{enumerate}
\end{lem}

\begin{thm}
\label{thm:P-star class ind hyp}For each $h'\in[1,\,h-1]$, the following
classification holds.
\begin{enumerate}
\item If $\mathcal{P}_{i}^{h'}\in\mathscr{P}_{m}^{\star,h'}$ for some $m\in[1,\,\nu_{h'}]$
with $|\mathscr{P}_{m}^{\star,h'}|=1$, then $\Gamma_{i}^{h'}>\Gamma^{\star,h'}$.
\item If $\mathcal{P}_{i}^{h'}\in\mathscr{P}_{{\rm tr}}^{\star,h'}$, then
$\Gamma_{i}^{h'}=\Gamma^{\star,h'}$ and there exists $\mathcal{P}_{j}^{h'}\in\mathscr{P}^{\star,h'}$
such that $\Phi(\mathcal{P}_{i}^{h'},\,\mathcal{P}_{j}^{h'})-\mathbb{H}(\mathcal{P}_{i}^{h'})=\Gamma^{\star,h'}$
and $\mathbb{H}(\mathcal{P}_{j}^{h'})<\mathbb{H}(\mathcal{P}_{i}^{h'})$.
\item If $\mathcal{P}_{i}^{h'}\in\mathscr{P}_{m}^{\star,h'}$ for some $m\in[1,\,\nu_{h'}]$
with $|\mathscr{P}_{m}^{\star,h'}|\ge2$, then $\Gamma_{i}^{h'}=\Gamma^{\star,h'}$
and $\mathbb{H}(\mathcal{P}_{j}^{h'})=\mathbb{H}(\mathcal{P}_{i}^{h'})$
for all $\mathcal{P}_{j}^{h'}\in\mathscr{P}^{\star,h'}$ such that
$\Phi(\mathcal{P}_{i}^{h'},\,\mathcal{P}_{j}^{h'})-\mathbb{H}(\mathcal{P}_{i}^{h'})=\Gamma^{\star,h'}$.
Moreover, in this case,
\[
\mathscr{P}_{m}^{\star,h'}\setminus\{\mathcal{P}_{i}^{h'}\}=\{\mathcal{P}_{j}^{h'}\in\mathscr{P}^{\star,h'}:\,\Phi(\mathcal{P}_{i}^{h'},\,\mathcal{P}_{j}^{h'})-\mathbb{H}(\mathcal{P}_{i}^{h'})=\Gamma^{\star,h'}\}.
\]
\end{enumerate}
\end{thm}

\begin{rem}
\label{rem:ini ind hyp}For the initial step $h'=1$, these induction
hypotheses were readily verified as follows:\smallskip{}
\begin{itemize}
\item Lemma \ref{lem:gen ind hyp}-(1): holds by \eqref{eq:P1 def}.
\item Lemma \ref{lem:gen ind hyp}-(2): exactly Lemma \ref{lem:Vi1 cycle}.
\item Lemma \ref{lem:gen ind hyp}-(3): proved in Lemma \ref{lem:Vi1 cycle}.
\item Lemma \ref{lem:gen ind hyp}-(4): holds by \eqref{eq:C1 def} and
\eqref{eq:C-star1 C-sharp1 prop}.
\item Lemma \ref{lem:gen ind hyp}-(5): obvious.
\item Lemma \ref{lem:gen ind hyp}-(6): proved in \eqref{eq:Gamma-star1 def}.
\item Lemma \ref{lem:gen ind hyp}-(7): proved in the proof of Lemma \ref{lem:Fbeta flat valley-1}.
\item Lemma \ref{lem:gen ind hyp}-(8): exactly Lemma \ref{lem:Fbeta flat valley-1}.
\item Lemma \ref{lem:gen ind hyp}-(9): proved in the proof of Lemma \ref{lem:R1 R-star1 eq}.
\item Theorem \ref{thm:P-star class ind hyp}: exactly Theorem \ref{thm:P-star1 class}.
\end{itemize}
\end{rem}

Now, we proceed as follows. First, in Section \ref{sec5.1}, we prove
Lemmas \ref{lem:Pih energy}, \ref{lem:Gamma-star inc}, \ref{lem:Vih cycle}
and \ref{lem:Ch prop} that are crucial to proceed from level $h-1$
to level $h$. In Section \ref{sec5.2}, we prove key properties of
$\{\mathfrak{X}^{\star,h}(t)\}_{t\ge0}$ in $\mathscr{P}^{\star,h}$
in an analogous manner as in Section \ref{sec4.1}. Using these properties,
in Section \ref{sec5.3} we prove Theorems \ref{thm:nu dec} and \ref{thm:ground states rec}.
Next, in Section \ref{sec5.4}, we prove Theorem \ref{thm:main}.
Finally, in Section \ref{sec5.5}, we check that Lemma \ref{lem:gen ind hyp}
and Theorem \ref{thm:P-star class ind hyp} are valid also for $h'=h$,
which completes the full inductive procedure.

\subsection{\label{sec5.1}Proof of Lemmas \ref{lem:Pih energy}, \ref{lem:Gamma-star inc},
\ref{lem:Vih cycle} and \ref{lem:Ch prop}}
\begin{proof}[Proof of Lemma \ref{lem:Pih energy}]
 Recall from \eqref{eq:P-starh def} that for each $m\in[1,\,\nu_{h-1}]$,
\begin{equation}
\mathcal{P}_{m}^{h}=\bigcup_{i\in[1,\,\nu_{h-2}]:\,\mathcal{P}_{i}^{h-1}\in\mathscr{P}_{m}^{\star,h-1}}\mathcal{P}_{i}^{h-1}.\label{eq:Pih energy pf1}
\end{equation}
For fixed $i\in[1,\,\nu_{h-2}]$ such that $\mathcal{P}_{i}^{h-1}\in\mathscr{P}_{m}^{\star,h-1}$,
by Lemma \ref{lem:gen ind hyp}-(1) with $h'=h-1$, it holds that
\begin{equation}
\mathbb{H}(\eta)=\mathbb{H}(\mathcal{P}_{i}^{h-1})\quad\text{for all}\quad\eta\in\mathcal{P}_{i}^{h-1}.\label{eq:Pih energy pf2}
\end{equation}
If $m$ is such that $|\mathscr{P}_{m}^{\star,h-1}|=1$, then Lemma
\ref{lem:Pih energy} readily holds from \eqref{eq:Pih energy pf2}.
If $|\mathscr{P}_{m}^{\star,h-1}|\ge2$, then by Theorem \ref{thm:P-star class ind hyp}-(3)
with $h'=h-1$, any two sets $\mathcal{P},\,\mathcal{P}'\in\mathscr{P}_{m}^{\star,h-1}$
share the common energy value, i.e. $\mathbb{H}(\mathcal{P})=\mathbb{H}(\mathcal{P}')$.
Thus, in this case Lemma \ref{lem:Pih energy} holds again by \eqref{eq:Pih energy pf2}.
\end{proof}
Before proceeding, we need a lemma.
\begin{lem}
\label{lem:Gamma inc}Fix $m\in[1,\,\nu_{h-1}]$.
\begin{enumerate}
\item If $|\mathscr{P}_{m}^{\star,h-1}|=1$ with $\mathscr{P}_{m}^{\star,h-1}=\{\mathcal{P}_{i}^{h-1}\}$,
then $\Gamma_{m}^{h}\ge\Gamma_{i}^{h-1}$.
\item If $|\mathscr{P}_{m}^{\star,h-1}|\ge2$, then $\Gamma_{m}^{h}>\Gamma^{\star,h-1}$.
\end{enumerate}
\end{lem}

\begin{proof}
(1) By \eqref{eq:Gammaih def}, we may write
\[
\Gamma_{m}^{h}=\Phi(\mathcal{P}_{m}^{h},\,\breve{\mathcal{P}}_{m}^{h})-\mathbb{H}(\mathcal{P}_{m}^{h})=\Phi(\mathcal{P}_{i}^{h-1},\,\breve{\mathcal{P}}_{m}^{h})-\mathbb{H}(\mathcal{P}_{i}^{h-1}).
\]
Moreover, we may rewrite $\breve{\mathcal{P}}_{m}^{h}$ as
\begin{equation}
\bigcup_{m'\in[1,\,\nu_{h-1}]:\,m'\ne m}\mathcal{P}_{m'}^{h}=\bigcup_{m':\,m'\ne m}\bigcup_{j\in[1,\,\nu_{h-2}]:\,\mathcal{P}_{j}^{h-1}\in\mathscr{P}_{m'}^{h-1}}\mathcal{P}_{j}^{h-1}\subseteq\bigcup_{j\in[1,\,\nu_{h-2}]:\,j\ne i}\mathcal{P}_{j}^{h-1}=\breve{\mathcal{P}}_{i}^{h-1}.\label{eq:Gamma inc pf1}
\end{equation}
Thus, by the previous two displays, we have $\Gamma_{\ell}^{h}\ge\Phi(\mathcal{P}_{i}^{h-1},\,\breve{\mathcal{P}}_{i}^{h-1})-\mathbb{H}(\mathcal{P}_{i}^{h-1})=\Gamma_{i}^{h-1}$.

\noindent (2) By Theorem \ref{thm:P-star class ind hyp}-(3) for $h'=h-1$,
we have $\Gamma_{j}^{h-1}=\Gamma^{\star,h-1}$ for all $j\in[1,\,\nu_{h-2}]$
such that $\mathcal{P}_{j}^{h-1}\in\mathscr{P}_{m}^{\star,h-1}$.
Then by \eqref{eq:Gamma inc pf1}, for all such $j$,
\[
\Phi(\mathcal{P}_{j}^{h-1},\,\breve{\mathcal{P}}_{m}^{h})\ge\Phi(\mathcal{P}_{j}^{h-1},\,\breve{\mathcal{P}}_{j}^{h-1})=\Gamma_{j}^{h-1}+\mathbb{H}(\mathcal{P}_{j}^{h-1})=\Gamma^{\star,h-1}+\mathbb{H}(\mathcal{P}_{m}^{h}),
\]
where the last equality holds by Lemma \ref{lem:Pih energy}. Thus,
\[
\Gamma_{m}^{h}+\mathbb{H}(\mathcal{P}_{m}^{h})=\Phi(\mathcal{P}_{m}^{h},\,\breve{\mathcal{P}}_{m}^{h})=\min_{j:\,\mathcal{P}_{j}^{h-1}\in\mathscr{P}_{m}^{\star,h-1}}\Phi(\mathcal{P}_{j}^{h-1},\,\breve{\mathcal{P}}_{m}^{h})\ge\Gamma^{\star,h-1}+\mathbb{H}(\mathcal{P}_{m}^{h}),
\]
which proves that $\Gamma_{m}^{h}\ge\Gamma^{\star,h-1}$. Finally,
suppose that the equality holds. Then, there exist $\mathcal{P}_{j}^{h-1}\in\mathscr{P}_{m}^{\star,h-1}$
and another $\mathcal{P}\in\mathscr{P}_{m'}^{\star,h-1}$ with $m'\ne m$
such that $\Phi(\mathcal{P}_{j}^{h-1},\,\mathcal{P})-\mathbb{H}(\mathcal{P}_{j}^{h-1})=\Gamma^{\star,h-1}$.
Via Theorem \ref{thm:P-star class ind hyp}-(3), this implies that
$\mathcal{P}\in\mathscr{P}_{m}^{\star,h-1}$ which contradicts $m'\ne m$.
Therefore, we obtain $\Gamma_{m}^{h}>\Gamma^{\star,h-1}$ which concludes
the proof.
\end{proof}
\begin{proof}[Proof of Lemma \ref{lem:Gamma-star inc}]
 It suffices to prove that $\Gamma_{m}^{h}>\Gamma^{\star,h-1}$ for
all $m\in[1,\,\nu_{h-1}]$. If $|\mathscr{P}_{m}^{\star,h-1}|=1$
such that $\mathcal{P}_{m}^{h}=\mathcal{P}_{i}^{h-1}$, then by Theorem
\ref{thm:P-star class ind hyp}-(1) and Lemma \ref{lem:Gamma inc}-(1),
$\Gamma_{m}^{h}\ge\Gamma_{i}^{h-1}>\Gamma^{\star,h-1}$. If $|\mathscr{P}_{m}^{\star,h-1}|\ge2$,
then by Lemma \ref{lem:Gamma inc}-(2), we conclude that $\Gamma_{m}^{h}>\Gamma^{\star,h-1}$.
\end{proof}
\begin{proof}[Proof of Lemma \ref{lem:Vih cycle}]
 By \eqref{eq:Gammaih def} and \eqref{eq:Vih def}, we obtain similarly
as in the proof of Lemma \ref{lem:Vi1 cycle} that $\mathcal{V}_{i}^{h}\cap\mathcal{V}_{j}^{h}=\emptyset$
for all $i\ne j$. Moreover, it is routine to see that $\mathbb{H}(\eta)<\mathbb{H}(\mathcal{P}_{i}^{h})+\Gamma_{i}^{h}$
for all $\eta\in\mathcal{V}_{i}^{h}$ and $\mathbb{H}(\zeta)\ge\mathbb{H}(\mathcal{P}_{i}^{h})+\Gamma_{i}^{h}$
for all $\zeta\in\partial\mathcal{V}_{i}^{h}$ where equality holds
if and only if $\zeta\in\partial^{\star}\mathcal{V}_{i}^{h}$.

In the next step, we prove that each $\mathcal{V}_{i}^{h}$ is connected.
To see this, it suffices to prove that all elements of $\mathcal{P}_{i}^{h}$
are connected strictly below energy level $\mathbb{H}(\mathcal{P}_{i}^{h})+\Gamma_{i}^{h}$.
If $|\mathscr{P}_{i}^{\star,h-1}|=1$ such that $\mathscr{P}_{i}^{\star,h-1}=\{\mathcal{P}_{j}^{h-1}\}$,
then by Lemma \ref{lem:gen ind hyp}-(2), all elements of $\mathcal{P}_{i}^{h}=\mathcal{P}_{j}^{h-1}$
are connected strictly below energy level $\mathbb{H}(\mathcal{P}_{j}^{h-1})+\Gamma_{j}^{h-1}$,
where $\Gamma_{j}^{h-1}\le\Gamma_{i}^{h}$ by Lemma \ref{lem:Gamma inc}-(1).
This readily proves that $\mathcal{V}_{i}^{h}$ is connected if $|\mathscr{P}_{i}^{\star,h-1}|=1$.
Suppose on the contrary that $|\mathscr{P}_{i}^{\star,h-1}|\ge2$.
Then by Theorem \ref{thm:P-star class ind hyp}-(3), each $\mathcal{P}_{j}^{h-1}\in\mathscr{P}_{i}^{\star,h-1}$
satisfies $\Gamma_{j}^{h-1}=\Gamma^{\star,h-1}$ and they are connected
within barrier $\mathbb{H}(\mathcal{P}_{i}^{h})+\Gamma^{\star,h-1}$,
where $\Gamma^{\star,h-1}<\Gamma_{i}^{h}$ by Lemma \ref{lem:Gamma inc}-(2).
Thus, in any case, we conclude that $\mathcal{V}_{i}^{h}$ is connected.
Collecting all the above observations, $\mathcal{V}_{i}^{h}$ is a
cycle for every $i\in[1,\,\nu_{h-1}]$.

It remains to prove that $\mathcal{F}(\mathcal{V}_{i}^{h})=\mathcal{P}_{i}^{h}$.
Note that $\mathcal{F}(\mathcal{V}_{i}^{h})$ is a union of stable
plateaux in $\mathscr{P}^{1}$. If $\mathcal{F}(\mathcal{V}_{i}^{h})\ne\mathcal{P}_{i}^{h}$,
then there exists $\mathcal{P}\in\mathscr{P}^{1}$ such that $\mathcal{P}\subseteq\mathcal{F}(\mathcal{V}_{i}^{h})$
and $\mathcal{P}\cap\mathcal{P}_{i}^{h}=\emptyset$. By Lemma \ref{lem:gen ind hyp}-(3),
$\mathcal{P}\subseteq\mathcal{C}$ for some $\mathcal{C}\in\mathscr{C}^{h-1}$.
Recall \eqref{eq:Ch-1 Ph-1 ind hyp}. If $\mathcal{C}\in\mathscr{C}_{i}^{\star,h-1}$
such that $\mathcal{C}=\mathcal{V}_{j}^{h-1}$ for some $j\in[1,\,\nu_{h-2}]$,
then by Lemma \ref{lem:gen ind hyp}-(2), $\mathcal{F}(\mathcal{V}_{j}^{h-1})=\mathcal{P}_{j}^{h-1}$
thus $\mathbb{H}(\mathcal{P})\ge\mathbb{H}(\mathcal{P}_{j}^{h-1})=\mathbb{H}(\mathcal{P}_{i}^{h})$.
If $\mathbb{H}(\mathcal{P})>\mathbb{H}(\mathcal{P}_{i}^{h})$ then
we have a contradiction to $\mathcal{P}\subseteq\mathcal{F}(\mathcal{V}_{i}^{h})$.
If $\mathbb{H}(\mathcal{P})=\mathbb{H}(\mathcal{P}_{i}^{h})$ then
$\mathcal{P}\subseteq\mathcal{P}_{j}^{h-1}\subseteq\mathcal{P}_{i}^{h}$,
which contradicts $\mathcal{P}\cap\mathcal{P}_{i}^{h}=\emptyset$.
Moreover, $\mathcal{C}\notin\mathscr{C}_{i'}^{\star,h-1}$ for any
$i'\ne i$ since collections $\mathcal{V}_{j}^{h-1}$ for $j\in[1,\,\nu_{h-2}]$
are disjoint by Lemma \ref{lem:gen ind hyp}-(2). Hence, we obtain
that
\[
\mathcal{C}\in\mathscr{C}_{{\rm tr}}^{\star,h-1}\cup\mathscr{C}^{\sharp,h-1}.
\]
By Lemma \ref{lem:gen ind hyp}-(4), $\mathcal{C}\in\mathscr{C}_{{\rm tr}}^{\star,h'}$
for some $h'\in[1,\,h-1]$. Write $\mathcal{C}=\mathcal{V}_{j}^{h'}$
such that $\mathbb{H}(\mathcal{P})\ge\mathbb{H}(\mathcal{P}_{j}^{h'})$.
By Theorem \ref{thm:P-star class ind hyp}-(2), $\Gamma_{j}^{h'}=\Gamma^{\star,h'}$
thus $\Phi(\mathcal{P},\,\mathcal{P}_{j}^{h'})<\mathbb{H}(\mathcal{P}_{j}^{h'})+\Gamma^{\star,h'}$.
The elements of $\mathcal{P}_{j}^{h'}$ are connected below barrier
$\mathbb{H}(\mathcal{P}_{j}^{h'})+\Gamma^{\star,h'-1}$ by Lemma \ref{lem:gen ind hyp}-(5).
In addition, again by Theorem \ref{thm:P-star class ind hyp}-(2),
there exists another $\mathcal{P}_{j'}^{h'}\in\mathscr{P}^{\star,h'}$
such that
\[
\Phi(\mathcal{P}_{j}^{h'},\,\mathcal{P}_{j'}^{h'})-\mathbb{H}(\mathcal{P}_{j}^{h'})=\Gamma^{\star,h'}\quad\text{and}\quad\mathbb{H}(\mathcal{P}_{j'}^{h'})<\mathbb{H}(\mathcal{P}_{j}^{h'}).
\]
Moreover, by Lemmas \ref{lem:gen ind hyp}-(6) and \ref{lem:Gamma inc},
we have $\Gamma^{\star,h'-1}<\Gamma^{\star,h'}\le\Gamma^{\star,h-1}<\Gamma_{i}^{h}$.
Hence, by concatenating through $\mathcal{P}\to\mathcal{P}_{j}^{h'}\to\mathcal{P}_{j'}^{h'}$,
we obtain that $\mathcal{P}_{j'}^{h'}\subseteq\mathcal{V}_{i}^{h}$.
This contradicts the fact that $\mathcal{P}\subseteq\mathcal{F}(\mathcal{V}_{i}^{h})$
since $\mathbb{H}(\mathcal{P})\ge\mathbb{H}(\mathcal{P}_{j}^{h'})>\mathbb{H}(\mathcal{P}_{j'}^{h'})$.
Thus, we conclude that $\mathcal{F}(\mathcal{V}_{i}^{h})=\mathcal{P}_{i}^{h}$
which completes the proof of Lemma \ref{lem:Vih cycle}.
\end{proof}
\begin{proof}[Proof of Lemma \ref{lem:Ch prop}]
 It suffices to prove that $\Gamma^{\mathcal{V}_{i}^{h}}\ge\Gamma^{\star,h}$
for all $i\in[1,\,\nu_{h-1}]$ and $\Gamma^{\mathcal{C}}<\Gamma^{\star,h}$
for all $\mathcal{C}\in\mathscr{C}_{{\rm tr}}^{\star,h-1}\cup\mathscr{C}^{\sharp,h-1}$.
The former inequality is obvious by Lemma \ref{lem:Vih cycle} and
\eqref{eq:Gamma-starh def}. For the latter inequality, if $\mathcal{C}\in\mathscr{C}_{{\rm tr}}^{\star,h-1}\cup\mathscr{C}^{\sharp,h-1}$
then by Lemma \ref{lem:gen ind hyp}-(4), it holds that $\mathcal{C}\in\mathscr{C}_{{\rm tr}}^{\star,h'}$
for some $h'\in[1,\,h-1]$. Then, by Lemma \ref{lem:gen ind hyp}-(2)
and Theorem \ref{thm:P-star class ind hyp}-(2),
\[
\Gamma^{\mathcal{C}}=\Gamma^{\star,h'}<\Gamma^{\star,h},
\]
where the strict inequality follows from Lemma \ref{lem:gen ind hyp}-(6).
This concludes the proof.
\end{proof}

\subsection{\label{sec5.2}Classification of $\mathscr{P}^{\star,h}$}

In this subsection, we prove the following theorem.
\begin{thm}
\label{thm:P-starh class}The following classification holds.
\begin{enumerate}
\item If $\mathcal{P}_{i}^{h}\in\mathscr{P}_{m}^{\star,h}$ for some $m\in[1,\,\nu_{h}]$
with $|\mathscr{P}_{m}^{\star,h}|=1$, then $\Gamma_{i}^{h}>\Gamma^{\star,h}$.
\item If $\mathcal{P}_{i}^{h}\in\mathscr{P}_{{\rm tr}}^{\star,h}$, then
$\Gamma_{i}^{h}=\Gamma^{\star,h}$ and there exists $\mathcal{P}_{j}^{h}\in\mathscr{P}^{\star,h}$
such that $\Phi(\mathcal{P}_{i}^{h},\,\mathcal{P}_{j}^{h})-\mathbb{H}(\mathcal{P}_{i}^{h})=\Gamma^{\star,h}$
and $\mathbb{H}(\mathcal{P}_{j}^{h})<\mathbb{H}(\mathcal{P}_{i}^{h})$.
\item If $\mathcal{P}_{i}^{h}\in\mathscr{P}_{m}^{\star,h}$ for some $m\in[1,\,\nu_{h}]$
with $|\mathscr{P}_{m}^{\star,h}|\ge2$, then $\Gamma_{i}^{h}=\Gamma^{\star,h}$
and $\mathbb{H}(\mathcal{P}_{j}^{h})=\mathbb{H}(\mathcal{P}_{i}^{h})$
for all $\mathcal{P}_{j}^{h}\in\mathscr{P}^{\star,h}$ such that $\Phi(\mathcal{P}_{i}^{h},\,\mathcal{P}_{j}^{h})-\mathbb{H}(\mathcal{P}_{i}^{h})=\Gamma^{\star,h}$.
Moreover, in this case,
\[
\mathscr{P}_{m}^{\star,h}\setminus\{\mathcal{P}_{i}^{h}\}=\{\mathcal{P}_{j}^{h}\in\mathscr{P}^{\star,h}:\,\Phi(\mathcal{P}_{i}^{h},\,\mathcal{P}_{j}^{h})-\mathbb{H}(\mathcal{P}_{i}^{h})=\Gamma^{\star,h}\}.
\]
\end{enumerate}
\end{thm}

The overall outline of the proof of Theorem \ref{thm:P-starh class}
is same as in Section \ref{sec4.1}; we focus on the differences and
omit the details.
\begin{lem}
\label{lem:5.2-1}For distinct $i,\,j$ such that $\mathfrak{R}^{\star,h}(\mathcal{P}_{i}^{h},\,\mathcal{P}_{j}^{h})>0$,
we have $\mathbb{H}(\mathcal{P}_{i}^{h})\ge\mathbb{H}(\mathcal{P}_{j}^{h})$.
Moreover, if $\mathbb{H}(\mathcal{P}_{i}^{h})=\mathbb{H}(\mathcal{P}_{j}^{h})$
then we also have $\mathfrak{R}^{\star,h}(\mathcal{P}_{j}^{h},\,\mathcal{P}_{i}^{h})>0$.
\end{lem}

\begin{proof}
As in the proof of Lemma \ref{lem:4.1-1}, we have $\Gamma_{i}^{h}=\Gamma^{\star,h}$
and there exists a sequence $\mathcal{P}_{i}^{h},\,\omega_{1},\,\dots,\,\omega_{N},\,\mathcal{P}_{j}^{h}$
with $\omega_{1},\,\dots,\,\omega_{N}\in\Delta^{h}\cup\mathscr{P}^{\sharp,h}$
along which $\mathfrak{R}^{h}(\cdot,\,\cdot)$ is positive. Here,
first it is clear that $\omega_{1}\in\partial^{\star}\mathcal{V}_{i}^{h}\subseteq\Delta^{h}$
and $\mathbb{H}(\omega_{1})=\mathbb{H}(\mathcal{P}_{i}^{h})+\Gamma_{i}^{h}$.
Given $\omega_{n}\in\Delta^{h}$ for $n\in[1,\,N-1]$, if $\omega_{n+1}\in\Delta^{h}$
then by \eqref{eq:RC def}, $\mathbb{H}(\omega_{n})\ge\mathbb{H}(\omega_{n+1})$.
If $\omega_{n+1}\in\mathscr{P}^{\sharp,h}$ such that $\omega_{n+1}\in\mathscr{P}_{{\rm tr}}^{\star,h'}$
for some $h'\in[1,\,h-1]$ by Lemma \ref{lem:gen ind hyp}-(4), say
$\omega_{n+1}=\mathcal{P}_{j'}^{h'}$, then it readily holds that
$\omega_{n+2}\in\Delta^{h}$ and $\omega_{n+2}\in\partial^{\star}\mathcal{V}_{j'}^{h'}$,
thus
\[
\mathbb{H}(\omega_{n})\ge\mathbb{H}(\partial^{\star}\mathcal{V}_{j'}^{h'})=\mathbb{H}(\omega_{n+2}).
\]
Finally, it is clear that $\omega_{N}\in\Delta^{h}$ and $\mathbb{H}(\omega_{N})\ge\mathbb{H}(\mathcal{P}_{j}^{h})+\Gamma_{j}^{h}$.
Hence, by an inductive argument, it holds that
\[
\mathbb{H}(\mathcal{P}_{i}^{h})+\Gamma^{\star,h}=\mathbb{H}(\omega_{1})\ge\cdots\ge\mathbb{H}(\omega_{N})\ge\mathbb{H}(\mathcal{P}_{j}^{h})+\Gamma_{j}^{h}.
\]
The rest of the proof is exactly same as the proof of Lemma \ref{lem:4.1-1}.
\end{proof}
\begin{lem}
\label{lem:5.2-2}It holds that $\Gamma_{i}^{h}=\Gamma^{\star,h}$
if and only if there exists $j\ne i$ with $\mathfrak{R}^{\star,h}(\mathcal{P}_{i}^{h},\,\mathcal{P}_{j}^{h})>0$.
\end{lem}

\begin{proof}
We only deal with the only if part. Fix $\mathcal{P}_{i}^{h}\in\mathscr{P}^{\star,h}$
with $\Gamma_{i}^{h}=\Gamma^{\star,h}$. We abbreviate $\mathbb{H}_{i}^{\star,h}:=\mathbb{H}(\mathcal{P}_{i}^{h})+\Gamma^{\star,h}$
and define $\mathcal{A}_{i}$ as the collection of configurations
in $\Omega\setminus\mathcal{V}_{i}^{h}$ reachable from $\partial^{\star}\mathcal{V}_{i}^{h}$
by a downhill path. Then, as in the proof of Lemma \ref{lem:4.1-2},
$\mathbb{H}(\mathcal{F}(\mathcal{A}_{i}))<\mathbb{H}_{i}^{\star,h}$
and there exists $\mathcal{P}\in\mathscr{P}^{1}$ such that $\mathcal{P}\subseteq\mathcal{F}(\mathcal{A}_{i})$.
According to Lemma \ref{lem:gen ind hyp}-(3), $\mathcal{P}\subseteq\mathcal{C}$
for some $\mathcal{C}\in\mathscr{C}^{h-1}$. If $\mathcal{C}\in\mathscr{C}_{i'}^{\star,h-1}$
(cf. \eqref{eq:Ch-1 Ph-1 ind hyp}) for $i'\ne i$, then we have a
path $\mathcal{P}_{i}^{h},\,\omega_{1},\,\dots,\,\omega_{N-1},\,\mathcal{P}_{i'}^{h}$
along which $\mathfrak{R}^{h}(\cdot,\,\cdot)$ is positive (where
$\omega_{1},\,\dots,\,\omega_{N-1}\in\Delta^{h}$), thus $\mathfrak{R}^{\star,h}(\mathcal{P}_{i}^{h},\,\mathcal{P}_{i'}^{h})>0$.
If $\mathcal{C}\in\mathscr{C}_{{\rm tr}}^{\star,h-1}\cup\mathscr{C}^{\sharp,h-1}$
such that $\mathcal{C}\in\mathscr{C}_{{\rm tr}}^{\star,h'}$ for some
$h'\in[1,\,h-1]$ (cf. Lemma \ref{lem:gen ind hyp}-(4)), say $\mathcal{C}=\mathcal{V}_{j}^{h'}$,
then by Theorem \ref{thm:P-star class ind hyp}-(2), we can find another
$\mathcal{P}'\in\mathscr{P}^{\star,h'}$ reachable from $\mathcal{P}_{j}^{h'}$
via $\{\mathfrak{X}^{\star,h'}(t)\}_{t\ge0}$ such that 
\begin{equation}
\mathbb{H}(\mathcal{P}')<\mathbb{H}(\mathcal{P}_{j}^{h'}).\label{eq:5.2-2 pf}
\end{equation}
If $\mathcal{P}'$ belongs to a recurrent collection $\mathscr{P}_{{\rm rec}}^{\star,h-1}$,
i.e. if $\mathcal{P}'\subseteq\mathcal{P}_{i'}^{h}$ for some $i'\ne i$,
then we are done. If not, we may iterate this procedure to end up
in the recurrent collection $\mathscr{P}_{{\rm rec}}^{\star,h-1}$
in finite steps due to the strict inequality in \eqref{eq:5.2-2 pf},
thus we conclude the proof.
\end{proof}
For each $i\in[1,\,\nu_{h-1}]$ with $\Gamma_{i}^{h}=\Gamma^{\star,h}$,
define
\begin{equation}
\mathscr{Z}_{i}^{h}:=\{\mathcal{P}_{j}^{h}\in\mathscr{P}^{\star,h}\setminus\{\mathcal{P}_{i}^{h}\}:\,\mathcal{P}_{j}^{h}\quad\text{is reachable from}\quad\mathcal{P}_{i}^{h}\quad\text{via}\quad\{\mathfrak{X}^{\star,h}(t)\}_{t\ge0}\},\label{eq:Zih def}
\end{equation}
which is nonempty by Lemma \ref{lem:5.2-2}. Moreover, $\mathbb{H}(\mathcal{P}_{j}^{h})\le\mathbb{H}(\mathcal{P}_{i}^{h})$
for all $\mathcal{P}_{j}^{h}\in\mathscr{Z}_{i}^{h}$ by Lemma \ref{lem:5.2-1}
and
\begin{equation}
\Phi(\mathcal{P}_{i}^{h},\,\mathcal{P}_{j}^{h})-\mathbb{H}(\mathcal{P}_{i}^{h})=\Gamma^{\star,h}\quad\text{for each}\quad\mathcal{P}_{j}^{h}\in\mathscr{Z}_{i}^{h}.\label{eq:Zih prop}
\end{equation}

\begin{lem}
\label{lem:5.2-3}Suppose that $\Gamma_{i}^{h}=\Gamma^{\star,h}$
and $\mathbb{H}(\mathcal{P}_{j}^{h})<\mathbb{H}(\mathcal{P}_{i}^{h})$
for some $\mathcal{P}_{j}^{h}\in\mathscr{Z}_{i}^{h}$. Then, $\mathcal{P}_{i}^{h}\in\mathscr{P}_{{\rm tr}}^{\star,h}$.
\end{lem}

\begin{proof}
By Lemma \ref{lem:5.2-1}, $\mathcal{P}_{i}^{h}$ is not reachable
from $\mathcal{P}_{j}^{h}$ via $\{\mathfrak{X}^{\star,h}(t)\}_{t\ge0}$.
Thus by \eqref{eq:Zih def}, $\mathcal{P}_{i}^{h}\in\mathscr{P}_{{\rm tr}}^{\star,h}$.
\end{proof}
\begin{lem}
\label{lem:5.2-4}Suppose that $\Gamma_{i}^{h}=\Gamma^{\star,h}$
and $\mathbb{H}(\mathcal{P}_{j}^{h})=\mathbb{H}(\mathcal{P}_{i}^{h})$
for all $\mathcal{P}_{j}^{h}\in\mathscr{Z}_{i}^{h}$. Then, $\mathcal{P}_{i}^{h}\in\mathscr{P}_{m}^{\star,h}$
for some $m\in[1,\,\nu_{h-1}]$ with $|\mathscr{P}_{m}^{\star,h}|\ge2$.
In this case, for each $\mathcal{P}_{j'}^{h}\in\mathscr{P}^{\star,h}$
with $\Phi(\mathcal{P}_{i}^{h},\,\mathcal{P}_{j'}^{h})-\mathbb{H}(\mathcal{P}_{i}^{h})=\Gamma^{\star,h}$,
it holds that $\mathcal{P}_{j'}^{h}\in\mathscr{Z}_{i}^{h}$.
\end{lem}

\begin{proof}
The first statement follows in the same way as in the proof of Lemma
\ref{lem:4.1-4}. Moreover, the second statement can also be proved
in the same way by demonstrating inductively that in the path from
$\mathcal{V}_{i}^{h}$ to $\mathcal{V}_{j'}^{h}$ as written in \eqref{eq:4.1-4 pf},
each jump inside $\Delta^{h}$ preserves the energy and each jump
from $\eta\in\Delta^{h}$ to $\mathcal{C}\in\mathscr{C}^{h}$ or vice
versa happens only when $\eta\in\partial^{\star}\mathcal{C}$. We
leave the tedious details to the readers.
\end{proof}
Now, we present a proof of Theorem \ref{thm:P-starh class}.
\begin{proof}[Proof of Theorem \ref{thm:P-starh class}]
 By Lemma \ref{lem:5.2-2}, $\Gamma_{i}^{h}>\Gamma^{\star,h}$ if
and only if $\mathcal{P}_{i}^{h}\in\mathscr{P}_{m}^{\star,h}$ with
$|\mathscr{P}_{m}^{\star,h}|=1$, thus Theorem \ref{thm:P-starh class}-(1)
follows.

For $\mathcal{P}_{i}^{h}\in\mathscr{P}_{{\rm tr}}^{\star,h}$, Lemmas
\ref{lem:5.2-2}, \ref{lem:5.2-3} and \ref{lem:5.2-4} imply that
$\mathbb{H}(\mathcal{P}_{j}^{h})<\mathbb{H}(\mathcal{P}_{i}^{h})$
for some $\mathcal{P}_{j}^{h}\in\mathscr{Z}_{i}^{h}$ where $\Phi(\mathcal{P}_{i}^{h},\,\mathcal{P}_{j}^{h})-\mathbb{H}(\mathcal{P}_{i}^{h})=\Gamma^{\star,h}$
(cf. \eqref{eq:Zih prop}), thus Theorem \ref{thm:P-starh class}-(2)
is verified.

Finally, suppose that $\mathcal{P}_{i}^{h}\in\mathscr{P}_{m}^{\star,h}$
with $|\mathscr{P}_{m}^{\star,h}|\ge2$. Then for each $\mathcal{P}_{j}^{h}\in\mathscr{P}^{\star,h}\setminus\{\mathcal{P}_{i}^{h}\}$
with $\Phi(\mathcal{P}_{i}^{h},\,\mathcal{P}_{j}^{h})-\mathbb{H}(\mathcal{P}_{i}^{h})=\Gamma^{\star,h}$,
the second statement of Lemma \ref{lem:5.2-4} tells us that $\mathcal{P}_{j}^{h}\in\mathscr{Z}_{i}^{h}$
and $\mathbb{H}(\mathcal{P}_{j}^{h})=\mathbb{H}(\mathcal{P}_{i}^{h})$.
Moreover, Lemma \ref{lem:5.2-4} implies again that each $\mathcal{P}_{j'}^{h}\in\mathscr{P}_{m}^{\star,h}\setminus\{\mathcal{P}_{i}^{h}\}$
satisfies $\Phi(\mathcal{P}_{i}^{h},\,\mathcal{P}_{j'}^{h})-\mathbb{H}(\mathcal{P}_{i}^{h})=\Gamma^{\star,h}$
and $\mathbb{H}(\mathcal{P}_{j'}^{h})=\mathbb{H}(\mathcal{P}_{i}^{h})$.
This concludes the proof of Theorem \ref{thm:P-starh class}-(3).
\end{proof}

\subsection{\label{sec5.3}Proof of Theorems \ref{thm:nu dec} and \ref{thm:ground states rec}}

In this subsection, we prove Theorems \ref{thm:nu dec} and \ref{thm:ground states rec}.
To prove Theorem \ref{thm:nu dec}, note from \eqref{eq:P-starh def}
that $\nu_{h-1}=|\mathscr{P}^{\star,h}|$. Lemma \ref{lem:5.2-2}
implies that $\{\mathfrak{X}^{\star,h}(t)\}_{\ge0}$ is a nonzero
Markov chain in $\mathscr{P}^{\star,h}$. This readily proves $\nu_{h}<\nu_{h-1}$.

To prove Theorem \ref{thm:ground states rec}, suppose the contrary.
Then, there exist $\bm{s}\in\mathcal{S}$ and $\mathcal{P}\in\mathscr{P}^{1}$
such that $\bm{s}\in\mathcal{P}$ and $\mathcal{P}\nsubseteq\mathcal{P}_{i}^{h}$
for any $\mathcal{P}_{i}^{h}\in\mathscr{P}_{{\rm rec}}^{\star,h}$.
By Lemma \ref{lem:gen ind hyp}-(3)(4) and the facts that $\mathcal{P}\nsubseteq\mathcal{P}_{i}^{h}$
and $\mathcal{S}=\mathcal{F}(\Omega)$, it must hold that $\mathcal{P}\subseteq\mathcal{C}$
for some $\mathcal{C}\in\mathscr{C}_{{\rm tr}}^{\star,h'}$ where
$h'\in[1,\,h]$. Then, by Theorems \ref{thm:P-star class ind hyp}-(2)
and \ref{thm:P-starh class}-(2), there exists $\mathcal{P}_{j}^{h'}\in\mathscr{P}^{\star,h'}$
such that $\mathbb{H}(\mathcal{P}_{j}^{h'})<\mathbb{H}(\mathcal{P})$,
which contradicts $\mathcal{S}=\mathcal{F}(\Omega)$.

\subsection{\label{sec5.4}Proof of Theorem \ref{thm:main}}

Denote by $\mathfrak{L}^{h}$ the generator of $\{\mathfrak{X}^{\star,h}(t)\}_{t\ge0}$
in $\mathscr{P}^{\star,h}$. For given $\lambda>0$ and $g:\,\mathscr{P}^{\star,h}\to\mathbb{R}$,
denote by $f:\,\mathscr{P}^{\star,h}\to\mathbb{R}$ the unique solution
to the $h$-th macroscopic resolvent equation $(\lambda-\mathfrak{L}^{h})f=g$
and denote by $G:\,\Omega\to\mathbb{R}$ the lift of $g$ given as
$G(\eta):={\bf 1}\{\eta\in\mathcal{V}^{\star,h}\}g(\Psi^{h}(\eta))$.
Then, abbreviating $F_{\beta}^{\star,h}:=F_{\beta}^{\Gamma^{\star,h}}$,
again by \cite[Theorem 2.3]{LMS resolvent}, we are left to prove
that
\begin{equation}
\lim_{\beta\to\infty}\sup_{\eta\in\mathcal{V}_{i}^{h}}|F_{\beta}^{\star,h}(\eta)-f(\mathcal{P}_{i}^{h})|=0\quad\text{for each}\quad i\in[1,\,\nu_{h-1}].\label{eq:res cond-h}
\end{equation}
First, we prove that for $\Gamma>\Gamma^{\star,h-1}$, $F_{\beta}^{\Gamma}$
is flat in each $\mathcal{V}_{i}^{h}$ for $i\in[1,\,\nu_{h-1}]$.
\begin{lem}
\label{lem:Fbeta flat valley-h}For all $\Gamma>\Gamma^{\star,h-1}$
and $i\in[1,\,\nu_{h-1}]$, it holds that
\[
\lim_{\beta\to\infty}\sup_{\eta\in\mathcal{V}_{i}^{h}}|F_{\beta}^{\Gamma}(\eta)-\overline{F}_{\beta}^{\Gamma}(\mathcal{P}_{i}^{h})|=0.
\]
\end{lem}

\begin{proof}
The strategy is the same as in the proof of Lemma \ref{lem:Fbeta flat valley-1};
define $\theta_{\beta}:=e^{\Gamma'\beta}$ for any fixed $\Gamma'\in(\Gamma^{\star,h-1},\,\Gamma)$.
The only difference here is that by \cite[Proposition 3.24 and Lemma 3.6]{NZB},
\[
\lim_{\beta\to\infty}\beta^{-1}\log t_{{\rm mix}}(\epsilon)=\widetilde{\Gamma}(\mathcal{V}_{i}^{h}),
\]
where $\widetilde{\Gamma}(\mathcal{V}_{i}^{h})$ is as defined in
\eqref{eq:Gamma-tilde def}. Thus, all that remains to be proved is
that $\widetilde{\Gamma}(\mathcal{V}_{i}^{h})\le\Gamma^{\star,h-1}$.

To measure $\widetilde{\Gamma}(\mathcal{V}_{i}^{h})$, without loss
of generality, we may assume that we start from a stable plateau $\mathcal{P}\in\mathscr{P}^{1}$
in $\mathcal{V}_{i}^{h}$. First, we prove that we reach the bottom
$\mathcal{F}(\mathcal{V}_{i}^{h})=\mathcal{P}_{i}^{h}$ with a path
of depth at most $\Gamma^{\star,h-1}$. Recall from Lemma \ref{lem:gen ind hyp}-(3)
that $\mathcal{P}\subseteq\mathcal{C}$ for some cycle $\mathcal{C}\in\mathscr{C}^{h-1}$.
If $\mathcal{C}\in\mathscr{C}_{{\rm tr}}^{\star,h-1}\cup\mathscr{C}^{\sharp,h-1}$,
then by Lemma \ref{lem:gen ind hyp}-(4) and Theorem \ref{thm:P-star class ind hyp}-(2)
there exists another $\mathcal{P}'\in\mathscr{P}^{1}$ such that
\[
\Phi(\mathcal{P},\,\mathcal{P}')-\mathbb{H}(\mathcal{P})\le\Phi(\mathcal{P},\,\mathcal{P}')-\mathbb{H}(\mathcal{F}(\mathcal{C}))\le\Gamma^{\star,h-1}
\]
and $\mathbb{H}(\mathcal{P}')<\mathbb{H}(\mathcal{F}(\mathcal{C}))\le\mathbb{H}(\mathcal{P})$,
thus with a strictly lower energy so that we may iterate the argument.
If $\mathcal{C}\in\mathscr{C}_{{\rm rec}}^{\star,h-1}$, then from
$\mathcal{P}$ we arrive at $\mathcal{F}(\mathcal{C})$ with a path
of depth at most $\Gamma^{\star,h-1}$ by Lemma \ref{lem:gen ind hyp}-(7),
where $\mathcal{F}(\mathcal{C})\subseteq\mathcal{F}(\mathcal{V}_{i}^{h})=\mathcal{P}_{i}^{h}$
by \eqref{eq:P-starh C-starh rec} and \eqref{eq:P-starh def}.

Thus, it remains to prove that all elements in $\mathcal{P}_{i}^{h}$
communicate with depth at most $\Gamma^{\star,h-1}$, i.e.,
\begin{equation}
\Phi(\eta,\,\eta')-\mathbb{H}(\mathcal{P}_{i}^{h})\le\Gamma^{\star,h-1}\quad\text{for all}\quad\eta,\,\eta'\in\mathcal{P}_{i}^{h}.\label{eq:flat claim}
\end{equation}
To prove this, if $|\mathscr{P}_{i}^{\star,h-1}|=1$ then $\mathcal{P}_{i}^{h}=\mathcal{P}_{j}^{h-1}$
for some $j\in[1,\,\nu_{h-2}]$, thus \eqref{eq:flat claim} follows
directly from Lemma \ref{lem:gen ind hyp}-(5)(6). If $|\mathscr{P}_{i}^{\star,h-1}|\ge2$,
then configurations in each $\mathcal{P}_{j}^{h-1}\in\mathscr{P}_{i}^{\star,h-1}$
are connected within $\mathbb{H}(\mathcal{P}_{i}^{h})+\Gamma^{\star,h-2}$
by Lemma \ref{lem:gen ind hyp}-(5) and any two $\mathcal{P}_{j}^{h-1},\,\mathcal{P}_{j'}^{h-1}\in\mathscr{P}_{i}^{\star,h-1}$
are connected within $\mathbb{H}(\mathcal{P}_{i}^{h})+\Gamma^{\star,h-1}$
by Theorem \ref{thm:P-star class ind hyp}-(3). Therefore, the proof
of Lemma \ref{lem:Fbeta flat valley-h} is completed.
\end{proof}
Now, we provide a lemma analogous to Lemma \ref{lem:Fbeta Delta1}.
We remark that Theorem \ref{thm:exit H1} is a crucial element to
prove the following lemma.
\begin{lem}
\label{lem:Fbeta Deltah}For $\Gamma>\Gamma^{\star,h-1}$ and $v\in\Delta^{h}\cup\mathscr{P}^{\sharp,h}$,
it holds that
\[
\lim_{\beta\to\infty}\Big|\overline{F}_{\beta}^{\Gamma}(v)-\sum_{i\in[1,\,\nu_{h-1}]}{\bf P}_{v}^{h}[\mathcal{T}_{\mathcal{P}_{i}^{h}}=\mathcal{T}_{\mathscr{P}^{\star,h}}]\cdot\overline{F}_{\beta}^{\Gamma}(\mathcal{P}_{i}^{h})\Big|=0.
\]
\end{lem}

\begin{proof}
The idea is identical as in the proof of Lemma \ref{lem:Fbeta Delta1}.
Fix $\Gamma'\in(\Gamma^{\star,h-1},\,\Gamma)$. We claim that $\lim_{\beta\to\infty}\mathbb{P}_{\eta}[\mathcal{T}_{\mathcal{V}^{\star,h}}>e^{\Gamma'\beta}]=0$
for any $\eta\in\overline{\Omega}\setminus\mathcal{V}^{\star,h}$.
Indeed, each configuration belonging to $\overline{\Omega}\setminus\mathcal{V}^{\star,h}$
is either non-stable ($\Delta^{h}$) or contained in some $\mathcal{C}\in\mathscr{C}^{\sharp,h}$.
By Lemma \ref{lem:gen ind hyp}-(2)(4)(6) and Theorem \ref{thm:P-star class ind hyp}-(2),
for such $\mathcal{C}$ we have $\Gamma^{\mathcal{C}}\le\Gamma^{\star,h-1}$.
Thus by \cite[Theorem 6.23-(ii)]{OV}, the escape time of such a non-stable
configuration in $\Delta^{h}$ or $\mathcal{C}\in\mathscr{C}^{\sharp,h}$
is at most $e^{(\Gamma^{\star,h-1}+\epsilon)\beta}$ with probability
tending to $1$ for any fixed $\epsilon>0$. Thus, by the usual coupling
argument, for any $\eta\in\overline{\Omega}\setminus\mathcal{V}^{\star,h}$
it holds that $\lim_{\beta\to\infty}\mathbb{P}_{\eta}[\mathcal{T}_{\mathcal{V}^{\star,h}}>e^{\Gamma'\beta}]=0$.

Now to prove Lemma \ref{lem:Fbeta Deltah}, it suffices to fix $v\in\Delta^{h}\cup\mathscr{P}^{\sharp,h}$
and a configuration $\eta$ such that $\eta=v$ if $v\in\Delta^{h}$
and $\eta\in\mathcal{C}$ if $v=\mathcal{F}(\mathcal{C})\in\mathscr{P}^{\sharp,h}$,
and prove that
\[
F_{\beta}^{\Gamma}(\eta)=\sum_{i\in[1,\,\nu_{h-1}]}{\bf P}_{v}^{h}[\mathcal{T}_{\mathcal{P}_{i}^{h}}=\mathcal{T}_{\mathscr{P}^{\star,h}}]\cdot\overline{F}_{\beta}^{\Gamma}(\mathcal{P}_{i}^{h})+o(1).
\]
Indeed, if $v=\mathcal{F}(\mathcal{C})\in\mathscr{P}^{\sharp,h}$,
it suffices to pick any $\eta\in\mathcal{C}$ since by Lemma \ref{lem:gen ind hyp}-(6)(8),
$F_{\beta}^{\Gamma}$ is flat in $\mathcal{C}$.

It has been verified in the proof of Lemma \ref{lem:Fbeta Delta1}
that the jump probability from $\xi\in\Delta^{h}$ of the original
process $\{\eta_{\beta}(t)\}_{t\ge0}$ is asymptotically equal to
the jump probability from $\xi$ of $\{\mathfrak{X}^{h}(t)\}_{t\ge0}$.
Moreover, by Theorem \ref{thm:exit H1} and \eqref{eq:RC def}, for
each $\mathcal{C}\in\mathscr{C}^{\sharp,h}$, the exit distribution
on $\partial\mathcal{C}$ from $\mathcal{C}$ of $\{\eta_{\beta}(t)\}_{t\ge0}$
is asymptotically equal to the probability distribution on $\partial\mathcal{C}$
of the first jump from $\mathcal{C}$ according to $\{\mathfrak{X}^{h}(t)\}_{t\ge0}$.
Hence, the usual coupling technique between these two processes (as
in \eqref{eq:Delta1-2}) implies that
\[
\lim_{\beta\to\infty}\mathbb{P}_{\eta}[\mathcal{T}_{\mathcal{V}_{i}^{h}}=\mathcal{T}_{\mathcal{V}^{\star,h}}]={\bf P}_{v}^{h}[\mathcal{T}_{\mathcal{P}_{i}^{h}}=\mathcal{T}_{\mathscr{P}^{\star,h}}]\quad\text{for each}\quad\mathcal{V}_{i}^{h}\in\mathscr{C}^{\star,h}.
\]
Using this asymptotic identity, we may proceed in a similar way as
in the proof of Lemma \ref{lem:Fbeta Delta1} to obtain that
\[
F_{\beta}^{\Gamma}(\eta)=\sum_{i\in[1,\,\nu_{h-1}]}{\bf P}_{v}^{h}[\mathcal{T}_{\mathcal{P}_{i}^{h}}=\mathcal{T}_{\mathscr{P}^{\star,h}}]\cdot\overline{F}_{\beta}^{\Gamma}(\mathcal{P}_{i}^{h})+o(1),
\]
using again the fact that $F_{\beta}^{\Gamma}$ is flat in each $\mathcal{V}_{i}^{h}\in\mathscr{C}^{\star,h}$
from Lemma \ref{lem:Fbeta flat valley-h}. This concludes the proof.
\end{proof}
Write $\overline{F}_{\beta}^{\star,h}:=\overline{F}_{\beta}^{\Gamma^{\star,h}}$.
As in Section \ref{sec4.3}, to conclude the proof of Theorem \ref{thm:main},
it suffices to prove that
\begin{equation}
(\lambda-\mathfrak{L}^{h})\overline{F}_{\beta}^{\star,h}(\mathcal{P}_{i}^{h})=g(\mathcal{P}_{i}^{h})+o(1)\quad\text{for}\quad\mathcal{P}_{i}^{h}\in\mathscr{P}^{\star,h}.\label{eq:res asymp-h}
\end{equation}
By integrating both sides of \eqref{eq:res micro} with $\Gamma=\Gamma^{\star,h}$
with respect to $\mu_{\beta}$ in $\mathcal{V}_{i}^{h}$ and substituting
the definition of $G$, it holds that
\begin{equation}
\lambda\sum_{\eta\in\mathcal{V}_{i}^{h}}\mu_{\beta}(\eta)F_{\beta}^{\star,h}(\eta)-e^{\Gamma^{\star,h}\beta}\sum_{\eta\in\mathcal{V}_{i}^{h}}\mu_{\beta}(\eta)L_{\beta}F_{\beta}^{\star,h}(\eta)=\mu_{\beta}(\mathcal{V}_{i}^{h})g(\mathcal{P}_{i}^{h}).\label{eq:5.4-1}
\end{equation}
By the definition of $\overline{F}_{\beta}^{\star,h}$ and Lemma \ref{lem:Fbeta flat valley-h},
the first term in the left-hand side of \eqref{eq:5.4-1} equals
\begin{equation}
\lambda\mu_{\beta}(\mathcal{V}_{i}^{h})(\overline{F}_{\beta}^{\star,h}(\mathcal{P}_{i}^{h})+o(1)).\label{eq:5.4-2}
\end{equation}
As calculated in \eqref{eq:4.1-4}, the second term in the left-hand
side of \eqref{eq:5.4-1} becomes
\begin{equation}
|\mathcal{P}_{i}^{h}|^{-1}\mu_{\beta}(\mathcal{P}_{i}^{h}){\bf 1}\{\Gamma_{i}^{h}=\Gamma^{\star,h}\}\sum_{\xi\in\partial^{\star}\mathcal{V}_{i}^{h}}\sum_{\eta\in\mathcal{V}_{i}^{h}:\,\eta\sim\xi}(F_{\beta}^{\star,h}(\xi)-F_{\beta}^{\star,h}(\eta))+o(\mu_{\beta}(\mathcal{P}_{i}^{h})).\label{eq:5.4-3}
\end{equation}
For the $F_{\beta}^{\star,h}(\xi)$-part in \eqref{eq:5.4-3}, by
\eqref{eq:RC def} and Lemma \ref{lem:Fbeta Deltah}, we may rewrite
this as
\[
\mu_{\beta}(\mathcal{P}_{i}^{h})\sum_{\xi\in\Delta^{h}}\mathfrak{R}^{h}(\mathcal{P}_{i}^{h},\,\xi)\Big(\sum_{j\in[1,\,\nu_{h-1}]}{\bf P}_{\xi}^{h}[\mathcal{T}_{\mathcal{P}_{j}^{h}}=\mathcal{T}_{\mathscr{P}^{\star,h}}]\cdot\overline{F}_{\beta}^{\star,h}(\mathcal{P}_{j}^{h})+o(1)\Big).
\]
By \eqref{eq:R-Cstar def}, this equals
\[
\mu_{\beta}(\mathcal{P}_{i}^{h})\sum_{j\in[1,\,\nu_{h-1}]}\mathfrak{R}^{\star,h}(\mathcal{P}_{i}^{h},\,\mathcal{P}_{j}^{h})\overline{F}_{\beta}^{\star,h}(\mathcal{P}_{j}^{h})+o(\mu_{\beta}(\mathcal{P}_{i}^{h})).
\]
For the $F_{\beta}^{\star,h}(\eta)$-part in \eqref{eq:5.4-3}, since
$F_{\beta}^{\star,h}$ is asymptotically constant in $\mathcal{V}_{i}^{h}$
by Lemmas \ref{lem:Gamma-star inc} and \ref{lem:Fbeta flat valley-h},
along with Lemma \ref{lem:Rh R-starh eq} stated below, it holds that

\begin{align*}
 & |\mathcal{P}_{i}^{h}|^{-1}\mu_{\beta}(\mathcal{P}_{i}^{h}){\bf 1}\{\Gamma_{i}^{h}=\Gamma^{\star,h}\}\sum_{\xi\in\partial^{\star}\mathcal{V}_{i}^{h}}\sum_{\eta\in\mathcal{V}_{i}^{h}:\,\eta\sim\xi}F_{\beta}^{\star,h}(\eta)\\
 & =\mu_{\beta}(\mathcal{P}_{i}^{h})\overline{F}_{\beta}^{\star,h}(\mathcal{P}_{i}^{h})\sum_{j}\mathfrak{R}^{\star,h}(\mathcal{P}_{i}^{h},\,\mathcal{P}_{j}^{h})+o(\mu_{\beta}(\mathcal{P}_{i}^{h})).
\end{align*}
Thus, the second term in the left-hand side of \eqref{eq:5.4-1} becomes
\begin{equation}
\mu_{\beta}(\mathcal{P}_{i}^{h})\sum_{j}\mathfrak{R}^{\star,h}(\mathcal{P}_{i}^{h},\,\mathcal{P}_{j}^{h})(\overline{F}_{\beta}^{\star,h}(\mathcal{P}_{j}^{h})-\overline{F}_{\beta}^{\star,h}(\mathcal{P}_{i}^{h}))+o(\mu_{\beta}(\mathcal{P}_{i}^{h})).\label{eq:5.4-4}
\end{equation}
By \eqref{eq:5.4-1}, \eqref{eq:5.4-2}, \eqref{eq:5.4-4} and the
fact that $\mu_{\beta}(\mathcal{V}_{i}^{h})=(1+o(1))\mu_{\beta}(\mathcal{P}_{i}^{h})$,
we conclude that
\[
\lambda\overline{F}_{\beta}^{\star,h}(\mathcal{P}_{i}^{h})-\sum_{j}\mathfrak{R}^{\star,h}(\mathcal{P}_{i}^{h},\,\mathcal{P}_{j}^{h})(\overline{F}_{\beta}^{\star,h}(\mathcal{P}_{j}^{h})-\overline{F}_{\beta}^{\star,h}(\mathcal{P}_{i}^{h}))=g(\mathcal{P}_{i}^{h})+o(1),
\]
which is exactly \eqref{eq:res asymp-h}. Thus, proving the next lemma
completes the proof of Theorem \ref{thm:main}.
\begin{lem}
\label{lem:Rh R-starh eq}For each $i\in[1,\,\nu_{h-1}]$, it holds
that
\[
\sum_{\eta\in\Delta^{h}}\mathfrak{R}^{h}(\mathcal{P}_{i}^{h},\,\eta)=\sum_{j\in[1,\,\nu_{h-1}]}\mathfrak{R}^{\star,h}(\mathcal{P}_{i}^{h},\,\mathcal{P}_{j}^{h}).
\]
\end{lem}

\begin{proof}
Following the idea presented in the proof of Lemma \ref{lem:R1 R-star1 eq},
it suffices to prove that $\mathcal{T}_{\mathscr{P}^{\star,h}}<\infty$
a.s. starting from any element in $\Delta^{h}\cup\mathscr{P}^{\sharp,h}$.

By Lemma \ref{lem:Gamma inc}, it holds that $\mathcal{V}_{j}^{h-1}\subseteq\mathcal{V}_{i}^{h}$
for all $\mathcal{P}_{j}^{h-1}\in\mathscr{P}_{i}^{\star,h-1}$ which
implies $\Delta^{h}\subseteq\Delta^{h-1}$. Thus, we obtain
\[
\Delta^{h}\cup\mathscr{P}^{\sharp,h}\subseteq\Delta^{h-1}\cup\mathscr{P}_{{\rm tr}}^{\star,h-1}\cup\mathscr{P}^{\sharp,h-1}.
\]
By Lemma \ref{lem:gen ind hyp}-(9) and the fact that ${\bf P}_{\cdot}^{h}$
and ${\bf P}_{\cdot}^{h-1}$ follow the same law before hitting $\mathscr{P}^{\star,h}$
and $\mathscr{P}_{{\rm rec}}^{\star,h-1}$, respectively, all it remains
to be proved is that ${\bf P}_{v}^{h-1}[\mathcal{T}_{\mathscr{P}_{{\rm rec}}^{\star,h-1}}<\infty]=1$
for all $v\in\mathscr{P}_{{\rm tr}}^{\star,h-1}$, which is trivial
by the definition of the transient collection $\mathscr{P}_{{\rm tr}}^{\star,h-1}$.
This concludes the proof.
\end{proof}

\subsection{\label{sec5.5}Inductive hypothesis at level $h$}

Finally, we check here that the inductive hypotheses at level $h'=h$
are readily verified along this section (except for Lemma \ref{lem:gen ind hyp}-(3)
which is proved right below), which completes the full logic of induction
in $h$.\smallskip{}

\begin{itemize}
\item Lemma \ref{lem:gen ind hyp}-(1): exactly Lemma \ref{lem:Pih energy}.
\item Lemma \ref{lem:gen ind hyp}-(2): exactly Lemma \ref{lem:Vih cycle}.
\item Lemma \ref{lem:gen ind hyp}-(3): to be proved below.
\item Lemma \ref{lem:gen ind hyp}-(4): exactly Lemma \ref{lem:Ch prop}.
\item Lemma \ref{lem:gen ind hyp}-(5): proved in \eqref{eq:flat claim}.
\item Lemma \ref{lem:gen ind hyp}-(6): exactly Lemma \ref{lem:Gamma-star inc}.
\item Lemma \ref{lem:gen ind hyp}-(7): proved in the proof of Lemma \ref{lem:Fbeta flat valley-h}.
\item Lemma \ref{lem:gen ind hyp}-(8): exactly Lemma \ref{lem:Fbeta flat valley-h}.
\item Lemma \ref{lem:gen ind hyp}-(9): proved in the proof of Lemma \ref{lem:Rh R-starh eq}.
\item Theorem \ref{thm:P-star class ind hyp}: exactly Theorem \ref{thm:P-starh class}.
\end{itemize}
\begin{proof}[Proof of Lemma \ref{lem:gen ind hyp}-(3) for $h'=h$]
 Fix $\mathcal{P}\in\mathscr{P}^{1}$. By Lemma \ref{lem:gen ind hyp}-(3)
for $h'=h-1$, it holds that $\mathcal{P}\subseteq\mathcal{C}$ for
some $\mathcal{C}\in\mathscr{C}^{h-1}$. Recall \eqref{eq:Ch-1 Ph-1 ind hyp}.
If $\mathcal{C}\in\mathscr{C}_{m}^{\star,h-1}$ for some $m\in[1,\,\nu_{h-1}]$,
then by Lemma \ref{lem:Gamma inc}, $\mathcal{C}\subseteq\mathcal{V}_{m}^{h}$,
thus $\mathcal{P}\subseteq\mathcal{V}_{m}^{h}$. If $\mathcal{C}\in\mathscr{C}_{{\rm tr}}^{\star,h-1}\cup\mathscr{C}^{\sharp,h-1}$
such that $\mathcal{C}\cap\mathcal{V}_{i}^{h}\ne\emptyset$ for some
$i\in[1,\,\nu_{h-1}]$, then similarly $\mathcal{C}\subseteq\mathcal{V}_{i}^{h}$,
thus $\mathcal{P}\subseteq\mathcal{V}_{i}^{h}$. Finally, if $\mathcal{C}\in\mathscr{C}_{{\rm tr}}^{\star,h-1}\cup\mathscr{C}^{\sharp,h-1}$
such that $\mathcal{C}\cap\mathcal{V}_{i}^{h}=\emptyset$ for all
$i\in[1,\,\nu_{h-1}]$, then $\mathcal{C}\in\mathscr{C}^{\sharp,h}$
by Lemma \ref{lem:Ch prop}. These three cases complete the proof
of Lemma \ref{lem:gen ind hyp}-(3) for $h'=h$.
\end{proof}

\section{\label{sec6}Ground States and Upper Bound of Energy Barrier}

\subsection{\label{sec6.1}Proof of Theorem \ref{thm:Ham S}}

As a starting point of our analysis of the Kawasaki dynamics, we first
prove Theorem \ref{thm:Ham S}. We need some preliminary definitions.
\begin{notation}[Square bracket in torus]
\label{nota:llb rrb}For an integer $M\ge1$ and $m,\,m'\in\mathbb{T}_{M}$,
we write $[m,\,m']_{M}:=\{m,\,m+1,\,\dots,\,m'\}\subseteq\mathbb{T}_{M}$.
Due to the periodic structure of $\mathbb{T}_{M}$, it is clear that
$[m,\,m']_{M}$ is always nonempty.
\end{notation}

\begin{defn}[Bridges, crosses and sticks]
\label{def:bridge cross stick}Recall from \eqref{eq:ck def} that
$\mathfrak{c}^{k}$ denotes the column $\{k\}\times\mathbb{T}_{L}$
in $\Lambda$ for each $k\in\mathbb{T}_{K}$.
\begin{itemize}
\item For $\ell\in\mathbb{T}_{L}$, denote by $\mathfrak{r}^{\ell}$ the
$\ell$-th row $\mathbb{T}_{K}\times\{\ell\}$ in $\Lambda$.
\item Given $\eta\in\Omega$ and $s\in\{0,\,1\}$, a column (resp. row)
in $\Lambda$ is a \emph{vertical} (resp. \emph{horizontal}) \emph{$s$-bridge}
of $\eta$ if all spins in the column (resp. row) are $s$. If there
exist both vertical and horizontal $s$-bridges, their union is called
an \emph{$s$-cross}.
\item For $k,\,k'\in\mathbb{T}_{K}$, define (cf. Notation \ref{nota:llb rrb})
\[
\mathfrak{c}^{[k,\,k']}:=\mathfrak{c}^{k}\cup\mathfrak{c}^{k+1}\cup\cdots\cup\mathfrak{c}^{k'}=\bigcup_{i\in[k,\,k']_{K}}\mathfrak{c}^{i}.
\]
\item For $k\in\mathbb{T}_{K}$, $\ell\in\mathbb{T}_{L}$ and $m\in[1,\,L-1]$,
define
\[
\mathfrak{s}_{m;\ell}^{k}:=\{k\}\times[\ell,\,\ell+m-1]_{L}
\]
and call it the \emph{$m$-stick} in $\mathfrak{c}^{k}$ starting
from position $\ell$. Moreover, we denote by 
\[
\mathfrak{S}_{m}^{k}:=\{\mathfrak{s}_{m;\ell}^{k}:\,\ell\in\mathbb{T}_{L}\}
\]
the collection of all $m$-sticks in $\mathfrak{c}^{k}$.
\end{itemize}
\end{defn}

\begin{defn}[Interface function $\mathbb{I}$]
\label{def:interface}$ $
\begin{itemize}
\item For each $\eta\in\Omega$, define
\begin{equation}
\mathbb{I}(\eta):=\sum_{\{x,\,y\}\in E}{\bf 1}\{\eta(x)\ne\eta(y)\},\label{eq:I def}
\end{equation}
which counts the number of interfaces between the two spins $0$ and
$1$ in $\eta$.
\item For column $\mathfrak{c}^{k}$ and row $\mathfrak{r}^{\ell}$ (cf.
Definition \ref{def:bridge cross stick}), define
\[
\mathbb{I}_{\mathfrak{c}^{k}}(\eta):=\sum_{\{x,\,y\}\in E:\,\{x,\,y\}\subseteq\mathfrak{c}^{k}}{\bf 1}\{\eta(x)\ne\eta(y)\}
\]
and
\[
\mathbb{I}_{\mathfrak{r}^{\ell}}(\eta):=\sum_{\{x,\,y\}\in E:\,\{x,\,y\}\subseteq\mathfrak{r}^{\ell}}{\bf 1}\{\eta(x)\ne\eta(y)\}.
\]
Then, by \eqref{eq:I def}, it is clear that
\begin{equation}
\mathbb{I}(\eta)=\sum_{k\in\mathbb{T}_{K}}\mathbb{I}_{\mathfrak{c}^{k}}(\eta)+\sum_{\ell\in\mathbb{T}_{L}}\mathbb{I}_{\mathfrak{r}^{\ell}}(\eta).\label{eq:I dec}
\end{equation}
Since the boundary is periodic, if $\mathfrak{c}^{k}$ (resp. $\mathfrak{r}^{\ell}$)
is not a bridge of $\eta$, then $\mathbb{I}_{\mathfrak{c}^{k}}(\eta)\ge2$
(resp. $\mathbb{I}_{\mathfrak{r}^{\ell}}(\eta)\ge2$). Moreover, $\mathbb{I}_{\mathfrak{c}^{k}}(\eta)$
and $\mathbb{I}_{\mathfrak{r}^{\ell}}(\eta)$ are always even integers
since there are exactly two spins in the system. In turn, by \eqref{eq:I dec},
$\mathbb{I}(\eta)$ is also an even integer.
\end{itemize}
\end{defn}

A simple connection between $\mathbb{I}$ and the Hamiltonian $\mathbb{H}$
is that
\begin{equation}
\mathbb{H}(\eta)=-2L\mathscr{N}_{0}+\frac{1}{2}\cdot\mathbb{I}(\eta)\quad\text{for all}\quad\eta\in\Omega.\label{eq:H I identity}
\end{equation}
Indeed, since $\eta(x)\in\{0,\,1\}$ for each $x\in V$, we may rewrite
the Hamiltonian as
\[
\mathbb{H}(\eta)=-\sum_{\{x,\,y\}\in E}{\bf 1}\{\eta(x)=\eta(y)=1\}=-\frac{1}{2}\cdot\sum_{x\in V}{\bf 1}\{\eta(x)=1\}\sum_{y\in V:\,\{x,\,y\}\in E}{\bf 1}\{\eta(y)=1\}.
\]
The second equality holds by a simple double-counting argument. Since
every $x\in V$ has exactly four neighboring sites and there exist
two types of spins in the system, this equals
\[
-\frac{1}{2}\cdot\sum_{x\in V}{\bf 1}\{\eta(x)=1\}\cdot\Big[4-\sum_{y\in V:\,\{x,\,y\}\in E}{\bf 1}\{\eta(y)=0\}\Big].
\]
Rearranging and employing the same double-counting argument once more,
we conclude that
\[
\mathbb{H}(\eta)=-2L\mathscr{N}_{0}+\frac{1}{2}\cdot\sum_{\{x,\,y\}\in E}{\bf 1}\{\eta(x)\ne\eta(y)\}=-2L\mathscr{N}_{0}+\frac{1}{2}\cdot\mathbb{I}(\eta),
\]
which is exactly \eqref{eq:H I identity}.
\begin{proof}[Proof of Theorem \ref{thm:Ham S}]
 By \eqref{eq:H I identity}, Theorem \ref{thm:Ham S} is equivalent
to:
\begin{equation}
\mathbb{I}(\eta)\ge2L\quad\text{for all}\quad\eta\in\Omega\quad\text{and}\quad\mathbb{I}(\eta)=2L\quad\text{if and only if}\quad\eta\in\mathcal{S}.\label{eq:Ham S WTS}
\end{equation}
To verify \eqref{eq:Ham S WTS}, we fix $\eta\in\Omega$ and divide
into four cases depending on whether $\eta$ has vertical or horizontal
bridges.\smallskip{}
\begin{itemize}
\item \textbf{(Case 1: $\eta$ has no bridges)} In this case, we have $\mathbb{I}_{\mathfrak{c}^{k}}(\eta)\ge2$
and $\mathbb{I}_{\mathfrak{r}^{\ell}}(\eta)\ge2$ for all $k\in\mathbb{T}_{K}$
and $\ell\in\mathbb{T}_{L}$. Thus, \eqref{eq:I dec} implies that
$\mathbb{I}(\eta)\ge2K+2L>2L$.\smallskip{}
\item \textbf{(Case 2: $\eta$ has a cross)} Suppose first that $\eta$
has a $0$-cross. Then, $\eta$ can be regarded as a configuration
of spins $1$ in a finite lattice with open boundary conditions. Thus,
by a standard isoperimetric inequality \cite[Corollary 2.5]{AC},
it holds that
\[
\mathbb{I}(\eta)\ge4\sqrt{\sum_{x\in V}{\bf 1}\{\eta(x)=1\}}=4\sqrt{L\mathscr{N}_{0}}>4\sqrt{\frac{L^{2}}{4}}=2L,
\]
where the second inequality follows by \eqref{eq:N0 assump}. On the
other hand, if $\eta$ has a $1$-cross, then the same logic implies
that
\[
\mathbb{I}(\eta)\ge4\sqrt{\sum_{x\in V}{\bf 1}\{\eta(x)=0\}}=4\sqrt{L(K-\mathscr{N}_{0})}>4\sqrt{\frac{L^{2}}{4}}=2L,
\]
where the second inequality again follows by \eqref{eq:N0 assump}.\smallskip{}
\item \textbf{(Case 3: $\eta$ has a horizontal bridge but no vertical bridge)}
Since $\eta$ has no vertical bridge, $\mathbb{I}_{\mathfrak{c}^{k}}(\eta)\ge2$
for all $k\in\mathbb{T}_{K}$. Along with \eqref{eq:I dec}, this
implies that $\mathbb{I}(\eta)\ge2K>2L$.\smallskip{}
\item \textbf{(Case 4: $\eta$ has a vertical bridge but no horizontal bridge)}
Similarly, we have $\mathbb{I}_{\mathfrak{r}^{\ell}}(\eta)\ge2$ for
all $\ell\in\mathbb{T}_{L}$, thus $\mathbb{I}(\eta)\ge2L$. To obtain
the equality here, it is required that
\[
\mathbb{I}_{\mathfrak{r}^{\ell}}(\eta)=2\quad\text{for all}\quad\ell\in\mathbb{T}_{L}\quad\text{and}\quad\mathbb{I}_{\mathfrak{c}^{k}}(\eta)=0\quad\text{for all}\quad k\in\mathbb{T}_{K}.
\]
This is achieved if and only if $\eta\in\mathcal{S}$.\smallskip{}
\end{itemize}
The presented four cases readily verify \eqref{eq:Ham S WTS}, thereby
complete the proof of Theorem \ref{thm:Ham S}.
\end{proof}

\subsection{\label{sec6.2}Upper bound of energy barrier}

We divide the proof of Theorem \ref{thm:E barrier} into the verification
of two lemmas: upper bound in Lemma \ref{lem:E barrier UB} and lower
bound in Lemma \ref{lem:E barrier LB}. In this subsection, we handle
the upper bound.
\begin{lem}
\label{lem:E barrier UB}We have $\Phi(\bm{\sigma^{k}},\,\bm{\sigma^{k'}})\le\mathbb{H}_{0}+4$
for all $k\ne k'$.
\end{lem}

\begin{proof}
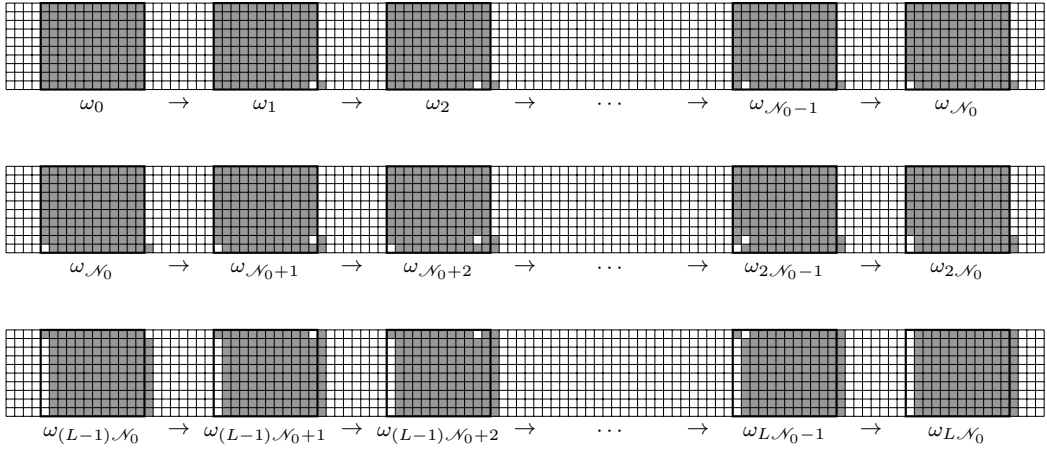
\begin{figure}
\begin{tikzpicture}[scale=0.115]
\fill[white] (-4,0) rectangle (116,10);

\fill[black!40!white] (0,0) rectangle (12,10);
\draw (6,0) node[below]{\scriptsize $\omega_0$};

\draw (16,0) node[below]{\scriptsize $\to$};

\fill[black!40!white] (20,0) rectangle (32,10);
\fill[black!40!white] (32,0) rectangle (33,1);
\fill[white] (31,0) rectangle (32,1);
\draw (26,0) node[below]{\scriptsize $\omega_1$};

\draw (36,0) node[below]{\scriptsize $\to$};

\fill[black!40!white] (40,0) rectangle (52,10);
\fill[black!40!white] (52,0) rectangle (53,1);
\fill[white] (50,0) rectangle (51,1);
\draw (46,0) node[below]{\scriptsize $\omega_2$};

\draw (56,0) node[below]{\scriptsize $\to$};
\draw (66,0) node[below]{\scriptsize $\cdots$};
\draw (76,0) node[below]{\scriptsize $\to$};

\fill[black!40!white] (80,0) rectangle (92,10);
\fill[black!40!white] (92,0) rectangle (93,1);
\fill[white] (81,0) rectangle (82,1);
\draw (86,0) node[below]{\scriptsize $\omega_{\mathscr{N}_0-1}$};

\draw (96,0) node[below]{\scriptsize $\to$};

\fill[black!40!white] (100,0) rectangle (112,10);
\fill[black!40!white] (112,0) rectangle (113,1);
\fill[white] (100,0) rectangle (101,1);
\draw (106,0) node[below]{\scriptsize $\omega_{\mathscr{N}_0}$};

\draw[very thin] (-4,0) grid (116,10);
\foreach \i in {0,1,2,4,5} {
\draw[thick] (0+20*\i,0) rectangle (12+20*\i,10);
}
\end{tikzpicture}\\
\vspace{5mm}
\begin{tikzpicture}[scale=0.115]
\fill[white] (-4,0) rectangle (116,10);

\fill[black!40!white] (0,0) rectangle (12,10);
\fill[black!40!white] (12,0) rectangle (13,1);
\fill[white] (0,0) rectangle (1,1);
\draw (6,0) node[below]{\scriptsize $\omega_{\mathscr{N}_0}$};

\draw (16,0) node[below]{\scriptsize $\to$};

\fill[black!40!white] (20,0) rectangle (32,10);
\fill[black!40!white] (32,0) rectangle (33,2);
\fill[white] (20,0) rectangle (21,1);
\fill[white] (31,1) rectangle (32,2);
\draw (26,0) node[below]{\scriptsize $\omega_{\mathscr{N}_0+1}$};

\draw (36,0) node[below]{\scriptsize $\to$};

\fill[black!40!white] (40,0) rectangle (52,10);
\fill[black!40!white] (52,0) rectangle (53,2);
\fill[white] (40,0) rectangle (41,1);
\fill[white] (50,1) rectangle (51,2);
\draw (46,0) node[below]{\scriptsize $\omega_{\mathscr{N}_0+2}$};

\draw (56,0) node[below]{\scriptsize $\to$};
\draw (66,0) node[below]{\scriptsize $\cdots$};
\draw (76,0) node[below]{\scriptsize $\to$};

\fill[black!40!white] (80,0) rectangle (92,10);
\fill[black!40!white] (92,0) rectangle (93,2);
\fill[white] (80,0) rectangle (81,1);
\fill[white] (81,1) rectangle (82,2);
\draw (86,0) node[below]{\scriptsize $\omega_{2\mathscr{N}_0-1}$};

\draw (96,0) node[below]{\scriptsize $\to$};

\fill[black!40!white] (100,0) rectangle (112,10);
\fill[black!40!white] (112,0) rectangle (113,2);
\fill[white] (100,0) rectangle (101,2);
\draw (106,0) node[below]{\scriptsize $\omega_{2\mathscr{N}_0}$};

\draw[very thin] (-4,0) grid (116,10);
\foreach \i in {0,1,2,4,5} {
\draw[thick] (0+20*\i,0) rectangle (12+20*\i,10);
}
\end{tikzpicture}\\
\vspace{5mm}
\begin{tikzpicture}[scale=0.115]
\fill[white] (-4,0) rectangle (116,10);

\fill[black!40!white] (0,0) rectangle (12,10);
\fill[black!40!white] (12,0) rectangle (13,9);
\fill[white] (0,0) rectangle (1,9);
\draw (6,0) node[below]{\scriptsize $\omega_{(L-1)\mathscr{N}_0}$};

\draw (16,0) node[below]{\scriptsize $\to$};

\fill[black!40!white] (20,0) rectangle (33,10);
\fill[white] (20,0) rectangle (21,9);
\fill[white] (31,9) rectangle (32,10);
\draw (26,0) node[below]{\scriptsize $\omega_{(L-1)\mathscr{N}_0+1}$};

\draw (36,0) node[below]{\scriptsize $\to$};

\fill[black!40!white] (40,0) rectangle (53,10);
\fill[white] (40,0) rectangle (41,9);
\fill[white] (50,9) rectangle (51,10);
\draw (46,0) node[below]{\scriptsize $\omega_{(L-1)\mathscr{N}_0+2}$};

\draw (56,0) node[below]{\scriptsize $\to$};
\draw (66,0) node[below]{\scriptsize $\cdots$};
\draw (76,0) node[below]{\scriptsize $\to$};

\fill[black!40!white] (80,0) rectangle (93,10);
\fill[white] (80,0) rectangle (81,9);
\fill[white] (81,9) rectangle (82,10);
\draw (86,0) node[below]{\scriptsize $\omega_{L\mathscr{N}_0-1}$};

\draw (96,0) node[below]{\scriptsize $\to$};

\fill[black!40!white] (101,0) rectangle (113,10);
\draw (106,0) node[below]{\scriptsize $\omega_{L\mathscr{N}_0}$};

\draw[very thin] (-4,0) grid (116,10);
\foreach \i in {0,1,2,4,5} {
\draw[thick] (0+20*\i,0) rectangle (12+20*\i,10);
}
\end{tikzpicture}\caption{\label{Fig6.1}Path $\omega:\,\bm{\sigma^{k}}\to\bm{\sigma^{k+1}}$
constructed in the proof of Lemma \ref{lem:E barrier UB}.}
\end{figure}

It suffices to prove that
\begin{equation}
\Phi(\bm{\sigma^{k}},\,\bm{\sigma^{k+1}})\le\mathbb{H}_{0}+4\quad\text{for all}\quad k\in\mathbb{T}_{K}.\label{eq:UB WTS}
\end{equation}
Indeed, we may iterate \eqref{eq:UB WTS} along the sequence $\bm{\sigma^{k}}\to\bm{\sigma^{k+1}}\to\cdots\to\bm{\sigma^{k'}}$
and apply \eqref{eq:Phi max ineq} to conclude the proof of Lemma
\ref{lem:E barrier UB}.

Recall \eqref{eq:eta-xy def} and define $\omega_{n+1}:=(\omega_{n})^{(k+\mathscr{N}_{0}-p-1,\,q)\leftrightarrow(k+\mathscr{N}_{0}-p,\,q)}$
for $n=p+\mathscr{N}_{0}q$ where $p\in[0,\,\mathscr{N}_{0}-1]$ and
$q\in[0,\,L-1]$ (refer to Figure \ref{Fig6.1} for an illustration).
Then, $\omega$ is a path from $\bm{\sigma^{k}}$ and $\bm{\sigma^{k+1}}$
and it is straightforward to check that:\smallskip{}
\begin{itemize}
\item $\mathbb{H}(\omega_{0})=\mathbb{H}(\omega_{L\mathscr{N}_{0}})=\mathbb{H}_{0}$;
\item $\mathbb{H}(\omega_{p+\mathscr{N}_{0}q})=\mathbb{H}_{0}+2$ if $p=0$
and $q\in[1,\,L-1]$;
\item $\mathbb{H}(\omega_{p+\mathscr{N}_{0}q})=\mathbb{H}_{0}+3$ if $p\in[1,\,\mathscr{N}_{0}-1]$
and $q\in\{0,\,L-1\}$;
\item $\mathbb{H}(\omega_{p+\mathscr{N}_{0}q})=\mathbb{H}_{0}+4$ if $p\in[1,\,\mathscr{N}_{0}-1]$
and $q\in[1,\,L-2]$.\smallskip{}
\end{itemize}
In particular, it holds that $\Phi_{\omega}=\mathbb{H}_{0}+4$ (cf.
\eqref{eq:Phi-omega def}), thus the existence of the path $\omega:\,\bm{\sigma^{k}}\to\bm{\sigma^{k+1}}$
readily proves that $\Phi(\bm{\sigma^{k}},\,\bm{\sigma^{k+1}})\le\mathbb{H}_{0}+4$.
This completes the proof of \eqref{eq:UB WTS}.
\end{proof}

\section{\label{sec7}Energy Landscape and Lower Bound of Energy Barrier}

In this section, we explore more detailed features of the energy landscape
of the model. Unless otherwise specified, we fix $k\in\mathbb{T}_{K}$
throughout this section. First, we introduce some notation.
\begin{notation}
\label{nota:allowed jump path}$ $
\begin{itemize}
\item A jump from $\eta$ to $\xi$ is an \emph{allowed} jump if $\eta\sim\xi$
and $\mathbb{H}(\eta),\,\mathbb{H}(\xi)\le\mathbb{H}_{0}+4$.
\item A path $\omega=(\omega_{n})_{n=0}^{N}$ is an \emph{allowed} path
if $\Phi_{\omega}\le\mathbb{H}_{0}+4$, i.e., if $\mathbb{H}(\omega_{n})\le\mathbb{H}_{0}+4$
for all $n\in[0,\,N]$.
\end{itemize}
\end{notation}

We adopt the following notation of neighborhoods from \cite[Definition 6.2]{KS IsingPotts-2D}:
\begin{defn}[Neighborhoods]
\label{def:nbd}$ $
\begin{itemize}
\item For $\eta\in\Omega$, define 
\[
\mathcal{N}(\eta):=\{\xi\in\Omega:\,\Phi(\eta,\,\xi)<\mathbb{H}_{0}+4\}\quad\text{and}\quad\widehat{\mathcal{N}}(\eta):=\{\xi\in\Omega:\,\Phi(\eta,\,\xi)\le\mathbb{H}_{0}+4\}.
\]
\item For $\mathcal{B}\subseteq\Omega$ and $\eta\notin\mathcal{B}$, define
\[
\widehat{\mathcal{N}}(\eta;\mathcal{B}):=\{\xi\in\Omega:\,\text{there exists an allowed path}\quad\omega:\,\eta\to\xi\quad\text{such that}\quad\omega\cap\mathcal{B}=\emptyset\}.
\]
In this terminology, it is clear that $\widehat{\mathcal{N}}(\eta;\emptyset)=\widehat{\mathcal{N}}(\eta)$.
\item For disjoint subsets $\mathcal{A},\,\mathcal{B}\subseteq\Omega$,
define
\[
\mathcal{N}(\mathcal{A}):=\bigcup_{\eta\in\mathcal{A}}\mathcal{N}(\eta),\quad\widehat{\mathcal{N}}(\mathcal{A}):=\bigcup_{\eta\in\mathcal{A}}\widehat{\mathcal{N}}(\eta)\quad\text{and}\quad\widehat{\mathcal{N}}(\mathcal{A};\mathcal{B}):=\bigcup_{\eta\in\mathcal{A}}\widehat{\mathcal{N}}(\eta;\mathcal{B}).
\]
\end{itemize}
\end{defn}

According to Definition \ref{def:nbd} and \eqref{eq:Kawa Omega-bar},
it is clear that 
\begin{equation}
\overline{\Omega}=\widehat{\mathcal{N}}(\mathcal{S}).\label{eq:Omega-bar N-hatS}
\end{equation}
The following decomposition lemma is useful along the investigation.
\begin{lem}[Decomposition of neighborhoods]
\label{lem:nbd dec}Fix two disjoint subsets $\mathcal{A},\,\mathcal{B}\subseteq\Omega$.
\begin{enumerate}
\item It holds that $\widehat{\mathcal{N}}(\mathcal{A}\cup\mathcal{B})=\widehat{\mathcal{N}}(\mathcal{A};\mathcal{B})\cup\widehat{\mathcal{N}}(\mathcal{B};\mathcal{A}).$
\item It holds that $\widehat{\mathcal{N}}(\mathcal{A}\cup\mathcal{B})=\widehat{\mathcal{N}}(\mathcal{A};\mathcal{B})\cup\widehat{\mathcal{N}}(\mathcal{B};\widehat{\mathcal{N}}(\mathcal{A};\mathcal{B})).$
\end{enumerate}
\end{lem}

We provide a proof of Lemma \ref{lem:nbd dec} in Appendix \ref{appenA}.

\subsection{\label{sec7.1}Shallow bottoms}

First, we define a specific collection of configurations called \emph{shallow
bottoms}. Refer to Figure \ref{Fig7.1} for illustrations.

\begin{figure}
\begin{tikzpicture}[scale=0.115]
\fill[white] (-4,0) rectangle (56,10);

\fill[black!40!white] (1,0) rectangle (12,10);
\fill[black!40!white] (0,4) rectangle (1,10);
\fill[black!40!white] (12,5) rectangle (13,9);
\draw (6,0) node[below]{\scriptsize $\bm{\sigma_{4;4,5}^k}$};

\fill[black!40!white] (21,0) rectangle (32,10);
\fill[black!40!white] (20,3) rectangle (21,7);
\fill[black!40!white] (32,2) rectangle (33,8);
\draw (26,0) node[below]{\scriptsize $\bm{\sigma_{6;3,2}^k}$};

\fill[black!40!white] (41,0) rectangle (52,10);
\fill[black!40!white] (40,5) rectangle (41,8);
\fill[black!40!white] (52,1) rectangle (53,8);
\draw (46,0) node[below]{\scriptsize $\bm{\sigma_{7;5,1}^k}$};

\draw[very thin] (-4,0) grid (56,10);
\foreach \i in {0,1,2} {
\draw[thick] (0+20*\i,0) rectangle (12+20*\i,10);
}
\end{tikzpicture}\caption{\label{Fig7.1}Examples of shallow bottoms defined in Definition \ref{def:shallow bot}.}
\end{figure}
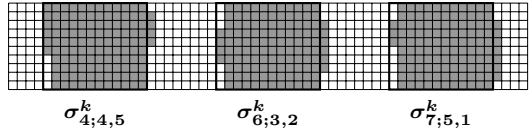

\begin{defn}[Shallow bottoms $\bm{\sigma_{m;\ell,\ell'}^{k}}$]
\label{def:shallow bot}For $m\in[1,\,L-1]$ and $\ell,\,\ell'\in\mathbb{T}_{L}$,
let $\bm{\sigma_{m;\ell,\ell'}^{k}}\in\Omega$ be the configuration
such that (cf. Definition \ref{def:bridge cross stick})
\[
\{x\in V:\,\bm{\sigma_{m;\ell,\ell'}^{k}}(x)=1\}=\mathfrak{c}^{[k+1,\,k+\mathscr{N}_{0}-1]}\cup\mathfrak{s}_{L-m;\ell}^{k}\cup\mathfrak{s}_{m;\ell'}^{k+\mathscr{N}_{0}}.
\]
Then, collect 
\begin{equation}
\mathcal{S}_{m}^{k}:=\{\bm{\sigma_{m;\ell,\ell'}^{k}}:\,\ell,\,\ell'\in\mathbb{T}_{L}\}.\label{eq:Skm def}
\end{equation}
It is easy to check that $\mathbb{H}(\bm{\sigma_{m;\ell,\ell'}^{k}})=\mathbb{H}_{0}+2$.
\end{defn}

First, we investigate some basic properties of $\bm{\sigma_{m;\ell,\ell'}^{k}}$.
Recall Definition \ref{def:nbd}.
\begin{lem}
\label{lem:shallow bot}Fix $\ell,\,\ell'\in\mathbb{T}_{L}$.
\begin{enumerate}
\item For all $m\in[1,\,L-1]$ and $\ell'',\,\ell'''\in\mathbb{T}_{L}$,
it holds that $\bm{\sigma_{m;\ell'',\ell'''}^{k}}\in\mathcal{N}(\bm{\sigma_{m;\ell,\ell'}^{k}})$.
\item We have $\bm{\sigma_{1;\ell,\ell'}^{k}}\in\mathcal{N}(\bm{\sigma^{k}})$
and $\bm{\sigma_{L-1;\ell,\ell'}^{k}}\in\mathcal{N}(\bm{\sigma^{k+1}})$.
\item For all $m\in[1,\,L-1]$, it holds that $\bm{\sigma_{m;\ell,\ell'}^{k}}\in\widehat{\mathcal{N}}(\mathcal{S})$.
\end{enumerate}
\end{lem}

\begin{proof}
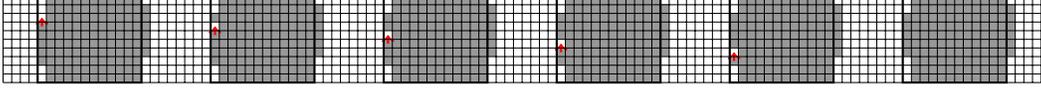
\begin{figure}
\begin{tikzpicture}[scale=0.115]
\fill[white] (-4,0) rectangle (116,10);

\fill[black!40!white] (1,0) rectangle (12,10);
\fill[black!40!white] (0,2) rectangle (1,7);
\fill[black!40!white] (12,3) rectangle (13,9);
\draw[->,thick,red] (0.5,6.5)--(0.5,7.5);

\fill[black!40!white] (21,0) rectangle (32,10);
\fill[black!40!white] (20,2) rectangle (21,6);
\fill[black!40!white] (20,7) rectangle (21,8);
\fill[black!40!white] (32,3) rectangle (33,9);
\draw[->,thick,red] (20.5,5.5)--(20.5,6.5);

\fill[black!40!white] (41,0) rectangle (52,10);
\fill[black!40!white] (40,2) rectangle (41,5);
\fill[black!40!white] (40,6) rectangle (41,8);
\fill[black!40!white] (52,3) rectangle (53,9);
\draw[->,thick,red] (40.5,4.5)--(40.5,5.5);

\fill[black!40!white] (61,0) rectangle (72,10);
\fill[black!40!white] (60,2) rectangle (61,4);
\fill[black!40!white] (60,5) rectangle (61,8);
\fill[black!40!white] (72,3) rectangle (73,9);
\draw[->,thick,red] (60.5,3.5)--(60.5,4.5);

\fill[black!40!white] (81,0) rectangle (92,10);
\fill[black!40!white] (80,2) rectangle (81,3);
\fill[black!40!white] (80,4) rectangle (81,8);
\fill[black!40!white] (92,3) rectangle (93,9);
\draw[->,thick,red] (80.5,2.5)--(80.5,3.5);

\fill[black!40!white] (101,0) rectangle (112,10);
\fill[black!40!white] (100,3) rectangle (101,8);
\fill[black!40!white] (112,3) rectangle (113,9);

\draw[very thin] (-4,0) grid (116,10);
\foreach \i in {0,...,5} {
\draw[thick] (0+20*\i,0) rectangle (12+20*\i,10);
}
\end{tikzpicture}\caption{\label{Fig7.2}Path from $\bm{\sigma_{m;\ell,\ell'}^{k}}$ to $\bm{\sigma_{m;\ell+1,\ell'}^{k}}$
with height $\mathbb{H}_{0}+3$ explained in the proof of Lemma \ref{lem:shallow bot}-(1).}
\end{figure}

\noindent (1) First, we prove that $\bm{\sigma_{m;\ell'',\ell'}^{k}}\in\mathcal{N}(\bm{\sigma_{m;\ell,\ell'}^{k}})$.
By iterating through $\ell\to\ell+1\to\cdots\to\ell''$, it suffices
to demonstrate that $\bm{\sigma_{m;\ell+1,\ell'}^{k}}\in\mathcal{N}(\bm{\sigma_{m;\ell,\ell'}^{k}})$.
To prove this, from $\bm{\sigma_{m;\ell,\ell'}^{k}}$ we move the
particle at $(k,\,\ell+L-m-i)$ to $(k,\,\ell+L-m-i+1)$ consecutively
for $i\in[1,\,L-m]$ as depicted in Figure \ref{Fig7.2}. This path
from $\bm{\sigma_{m;\ell,\ell'}^{k}}$ to $\bm{\sigma_{m;\ell+1,\ell'}^{k}}$
has height $\mathbb{H}_{0}+3$, thus it proves that $\bm{\sigma_{m;\ell+1,\ell'}^{k}}\in\mathcal{N}(\bm{\sigma_{m;\ell,\ell'}^{k}})$
as desired. Moreover, we may construct a similar path of particle
movements inside column $\mathfrak{c}^{k+\mathscr{N}_{0}}$ to deduce
that $\bm{\sigma_{m;\ell'',\ell'''}^{k}}\in\mathcal{N}(\bm{\sigma_{m;\ell'',\ell'}^{k}})$.
Thus, we conclude that $\bm{\sigma_{m;\ell'',\ell'''}^{k}}\in\mathcal{N}(\bm{\sigma_{m;\ell,\ell'}^{k}})$.

\noindent (2) For the first statement, by part (1), it suffices to
prove that $\bm{\sigma_{1;1,0}^{k}}\in\mathcal{N}(\bm{\sigma^{k}})$.
This is demonstrated by the subpath $(\omega_{n})_{n=0}^{\mathscr{N}_{0}}$
of the path $\omega:\,\bm{\sigma^{k}}\to\bm{\sigma^{k+1}}$ defined
in the proof of Lemma \ref{lem:E barrier UB} (see also the first
line of figures in Figure \ref{Fig6.1}), which has height $\mathbb{H}_{0}+3$
as observed in the proof. The second statement can be proved in the
same way.

\noindent (3) By part (1), it suffices to prove that $\bm{\sigma_{m;m,0}^{k}}\in\widehat{\mathcal{N}}(\mathcal{S})$
for all $m\in[1,\,L-1]$. As in the proof of part (2), each $\bm{\sigma_{m;m,0}^{k}}$
appears in the allowed path $\omega:\,\bm{\sigma^{k}}\to\bm{\sigma^{k+1}}$
defined in the proof of Lemma \ref{lem:E barrier UB} and this implies
that $\bm{\sigma_{m;m,0}^{k}}\in\widehat{\mathcal{N}}(\mathcal{S})$
as desired.
\end{proof}

\subsection{\label{sec7.2}Energetic valleys}

In this subsection, we define \emph{energetic valleys} around the
ground states $\bm{\sigma^{k}}$ and the shallow bottoms $\bm{\sigma_{m;\ell,\ell'}^{k}}$.
\begin{defn}
\label{def:energetic valleys}Recall Definition \ref{def:nbd} and
define 
\[
\mathcal{N}^{k}:=\mathcal{N}(\bm{\sigma^{k}}).
\]
Moreover, by Lemma \ref{lem:shallow bot}-(1), we may define 
\[
\mathcal{N}_{m}^{k}:=\mathcal{N}(\bm{\sigma_{m;\ell,\ell'}^{k}})\quad\text{for}\quad m\in[1,\,L-1]\quad\text{and}\quad\ell,\,\ell'\in\mathbb{T}_{L},
\]
which does not depend on the selection of $\ell$ and $\ell'$. Then,
collect
\begin{equation}
\mathcal{V}:=\bigcup_{k\in\mathbb{T}_{K}}\mathcal{N}^{k}\cup\bigcup_{k\in\mathbb{T}_{K}}\bigcup_{m=2}^{L-2}\mathcal{N}_{m}^{k}.\label{eq:V def}
\end{equation}
Note that we may exclude $\mathcal{N}_{1}^{k}$ and $\mathcal{N}_{L-1}^{k}$
in \eqref{eq:V def} since $\mathcal{N}_{1}^{k}=\mathcal{N}^{k}$
and $\mathcal{N}_{L-1}^{k}=\mathcal{N}^{k+1}$ by Lemma \ref{lem:shallow bot}-(2).
\end{defn}

\subsubsection{\label{sec7.2.1}Collections $\mathcal{N}_{m}^{k}$ for $m\in[2,\,L-2]$}

Here, we investigate the local geometry of the valleys that appear
in the right-hand side of \eqref{eq:V def}. First, we handle $\mathcal{N}_{m}^{k}$
for $m\in[2,\,L-2]$.

\begin{figure}
\begin{tikzpicture}[scale=0.115]
\fill[white] (-4,0) rectangle (76,10);

\fill[black!40!white] (1,0) rectangle (12,10);
\fill[black!40!white] (0,0) rectangle (1,3);
\fill[black!40!white] (0,7) rectangle (1,8);
\fill[black!40!white] (12,0) rectangle (13,6);

\fill[black!40!white] (21,0) rectangle (32,10);
\fill[black!40!white] (20,0) rectangle (21,2);
\fill[black!40!white] (20,3) rectangle (21,5);
\fill[black!40!white] (32,0) rectangle (33,6);

\fill[black!40!white] (41,0) rectangle (52,10);
\fill[black!40!white] (40,0) rectangle (41,4);
\fill[black!40!white] (52,0) rectangle (53,5);
\fill[black!40!white] (52,7) rectangle (53,8);

\fill[black!40!white] (61,0) rectangle (72,10);
\fill[black!40!white] (60,0) rectangle (61,4);
\fill[black!40!white] (72,0) rectangle (73,2);
\fill[black!40!white] (72,3) rectangle (73,7);

\draw[very thin] (-4,0) grid (76,10);
\foreach \i in {0,...,3} {
\draw[thick] (0+20*\i,0) rectangle (12+20*\i,10);
}
\end{tikzpicture}\caption{\label{Fig7.3}Examples of configurations in $\mathcal{N}_{m}^{k}$
for $m\in[2,\,L-2]$.}
\end{figure}
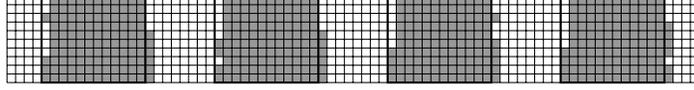

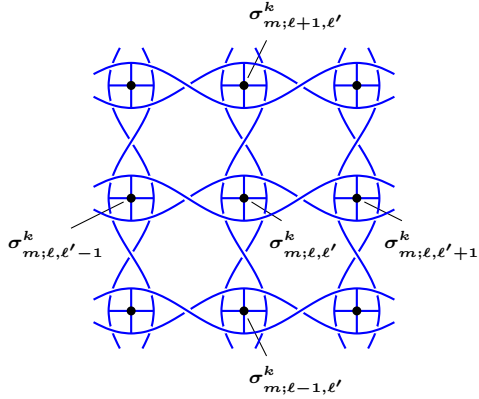
\begin{figure}
\begin{tikzpicture}
\draw[thin] (0.1,-0.05)--(0.5,-0.3); \draw (0.8,-0.3) node[below]{\tiny $\bm{\sigma_{m;\ell,\ell'}^k}$};
\draw[thin] (1.6,-0.05)--(2.1,-0.3); \draw (2.5,-0.3) node[below]{\tiny $\bm{\sigma_{m;\ell,\ell'+1}^k}$};
\draw[thin] (-1.6,-0.05)--(-2.3,-0.4); \draw (-2.5,-0.3) node[below]{\tiny $\bm{\sigma_{m;\ell,\ell'-1}^k}$};
\draw[thin] (0.05,1.6)--(0.3,2.1); \draw (0.7,2.1) node[above]{\tiny $\bm{\sigma_{m;\ell+1,\ell'}^k}$};
\draw[thin] (0.05,-1.6)--(0.3,-2.1); \draw (0.7,-2.1) node[below]{\tiny $\bm{\sigma_{m;\ell-1,\ell'}^k}$};

\foreach \i in {-1.5,0,1.5} {
\draw[blue,thick] plot[domain=-2:2,smooth]({cos(120*\x)*0.3+\i},{\x});
\fill[white] (\i,0.75) circle (0.05);
\fill[white] (\i,-0.75) circle (0.05);
\draw[blue,thick] plot[domain=-2:2,smooth]({-cos(120*\x)*0.3+\i},{\x});
}

\foreach \i in {-1.5,0,1.5} { \foreach \j in {-1.5,0,1.5} {
\fill[white] (0.25-\i,0.25-\j) circle (0.05);
\fill[white] (0.25-\i,-0.25-\j) circle (0.05);
\fill[white] (-0.25-\i,0.25-\j) circle (0.05);
\fill[white] (-0.25-\i,-0.25-\j) circle (0.05);
} }

\foreach \i in {-1.5,0,1.5} {
\draw[blue,thick] plot[domain=-2:2,smooth]({\x},{cos(120*\x)*0.3+\i});
\fill[white] (0.75,\i) circle (0.05);
\fill[white] (-0.75,\i) circle (0.05);
\draw[blue,thick] plot[domain=-2:2,smooth]({\x},{-cos(120*\x)*0.3+\i});
}

\foreach \i in {-1.5,0,1.5} { \foreach \j in {-1.5,0,1.5} {
\draw[blue,thick] (\i-0.3,\j)--(\i+0.3,\j); \draw[blue,thick] (\i,\j-0.3)--(\i,\j+0.3);
\fill (\i,\j) circle (0.06);
} }

\foreach \i in {-0.75,0.75} { \foreach \j in {-1.5,0,1.5} {
} }

\foreach \i in {-1.5,0,1.5} { \foreach \j in {-0.75,0.75} {
} }
\end{tikzpicture}\caption{\label{Fig7.4}Structure of $\mathcal{N}_{m}^{k}$ explained in Section
\ref{sec7.2.1}. Each dot represents the configurations in $\mathcal{S}_{m}^{k}$
(cf. \eqref{eq:Skm def}) and the blue lines denote the allowed jumps
between the configurations.}
\end{figure}

Refer to Figure \ref{Fig7.4} for the structure. To characterize the
configurations in $\mathcal{N}_{m}^{k}$ for $m\in[2,\,L-2]$, we
explore the configurations reachable by a path starting from $\bm{\sigma_{m;\ell,\ell'}^{k}}$
for fixed $\ell,\,\ell'\in\mathbb{T}_{L}$ and with height at most
$\mathbb{H}_{0}+3$. One can easily see that any horizontal particle
jumps from $\bm{\sigma_{m;\ell,\ell'}^{k}}$ raise the energy to at
least $\mathbb{H}_{0}+4$. Thus, particle jumps may only occur vertically
in columns $\mathfrak{c}^{k}$ or $\mathfrak{c}^{k+\mathscr{N}_{0}}$,
in four possible ways. If a particle moves in $\mathfrak{c}^{k}$,
then the energy becomes $\mathbb{H}_{0}+3$. Thereafter, the possible
options are to move the empty site inside the occupied sites or move
the occupied site inside the empty sites, thereby return to energy
$\mathbb{H}_{0}+2$ by arriving at another $\bm{\sigma_{m;\ell'',\ell'}^{k}}$
for $\ell''\in\mathbb{T}_{L}$ (illustrated by blue vertical lines
in Figure \ref{Fig7.4}). Note here that we can arrive at $\bm{\sigma_{m;\ell'',\ell'}^{k}}$
for every $\ell''\in\mathbb{T}_{L}$ before visiting any other configurations
with energy $\mathbb{H}_{0}+2$. The same logic holds for the movements
in $\mathfrak{c}^{k+\mathscr{N}_{0}}$ as well (blue horizontal lines
in Figure \ref{Fig7.4}). Thus, we proved that $\mathcal{N}_{m}^{k}$
consists of the configurations obtained from $\mathcal{S}_{m}^{k}$
by particle jumps in $\mathfrak{c}^{k}$ or $\mathfrak{c}^{k+\mathscr{N}_{0}}$
within energy $\mathbb{H}_{0}+3$. See Figure \ref{Fig7.3} for examples
of configurations in $\mathcal{N}_{m}^{k}$. In particular, we proved
the following lemma.
\begin{lem}
\label{lem:Nkm stab plat}For $m\in[2,\,L-2]$, the $L^{2}$ stable
plateaux in $\mathcal{N}_{m}^{k}$ are $\{\bm{\sigma_{m;\ell,\ell'}^{k}}\}$
for all $\ell,\,\ell'\in\mathbb{T}_{L}$.
\end{lem}

\subsubsection{\label{sec7.2.2}Collections $\mathcal{N}^{k}$}

We investigate $\mathcal{N}^{k}$ in Section \ref{sec7.2.2}.
\begin{notation}
\label{nota:sitting stick}Suppose that $k\in\mathbb{T}_{K}$ and
$k'\in\mathbb{T}_{K}$ satisfy $k'\in\{k-1,\,k+1\}$. If $W,\,W'$
are sticks in $\mathfrak{c}^{k},\,\mathfrak{c}^{k'}$, respectively
(cf. Definition \ref{def:bridge cross stick}), we say that \emph{$W'$
sits on $W$} and write $W\vdash W'$ if the set of vertical positions
of $W'$ is contained in the set of vertical positions of $W$. In
this case, it is automatic that the length of $W$ is bigger than
or equal to the length of $W'$.
\end{notation}

\begin{figure}
\begin{tikzpicture}[scale=0.115]
\fill[white] (-4,0) rectangle (96,10);

\fill[black!40!white] (0,0) rectangle (11,10);
\fill[black!40!white] (11,0) rectangle (12,9);
\fill[black!40!white] (12,9) rectangle (13,10);
\draw (6,0) node[below]{\scriptsize $\mathcal{F}^k$};

\fill[black!40!white] (20,0) rectangle (31,10);
\fill[black!40!white] (31,0) rectangle (32,9);
\fill[black!40!white] (32,0) rectangle (33,1);
\draw (26,0) node[below]{\scriptsize $\mathcal{R}^k$};

\fill[black!40!white] (40,0) rectangle (51,10);
\fill[black!40!white] (51,0) rectangle (52,8);
\fill[black!40!white] (52,7) rectangle (53,9);
\draw (46,0) node[below]{\scriptsize $\mathcal{R}_{\mathfrak{0}}^k$};

\fill[black!40!white] (60,0) rectangle (72,10);
\fill[black!40!white] (72,0) rectangle (73,1);
\fill[white] (65,5) rectangle (66,6);
\draw (66,0) node[below]{\scriptsize $\mathcal{P}_1^k$};

\fill[black!40!white] (80,0) rectangle (92,10);
\fill[black!40!white] (94,5) rectangle (95,6);
\fill[white] (91,9) rectangle (92,10);
\draw (86,0) node[below]{\scriptsize $\mathcal{Q}_L^{k-1}$};

\draw[very thin] (-4,0) grid (96,10);
\foreach \i in {0,...,4} {
\draw[thick] (0+20*\i,0) rectangle (12+20*\i,10);
}
\end{tikzpicture}\\
\vspace{5mm}
\begin{tikzpicture}[scale=0.115]
\fill[white] (-4,0) rectangle (96,10);

\fill[black!40!white] (1,0) rectangle (12,10);
\fill[black!40!white] (-1,9) rectangle (0,10);
\fill[black!40!white] (0,0) rectangle (1,9);
\draw (6,0) node[below]{\scriptsize $\mathcal{G}^k$};

\fill[black!40!white] (21,0) rectangle (32,10);
\fill[black!40!white] (19,0) rectangle (20,1);
\fill[black!40!white] (20,0) rectangle (21,9);
\draw (26,0) node[below]{\scriptsize $\mathcal{L}^k$};

\fill[black!40!white] (41,0) rectangle (52,10);
\fill[black!40!white] (39,7) rectangle (40,9);
\fill[black!40!white] (40,0) rectangle (41,8);
\draw (46,0) node[below]{\scriptsize $\mathcal{L}_{\mathfrak{0}}^k$};

\fill[black!40!white] (60,0) rectangle (72,10);
\fill[black!40!white] (59,0) rectangle (60,1);
\fill[white] (66,5) rectangle (67,6);
\draw (66,0) node[below]{\scriptsize $\mathcal{P}_L^{k-1}$};

\fill[black!40!white] (80,0) rectangle (92,10);
\fill[black!40!white] (77,5) rectangle (78,6);
\fill[white] (80,9) rectangle (81,10);
\draw (86,0) node[below]{\scriptsize $\mathcal{Q}_1^k$};

\draw[very thin] (-4,0) grid (96,10);
\foreach \i in {0,...,4} {
\draw[thick] (0+20*\i,0) rectangle (12+20*\i,10);
}
\end{tikzpicture}\caption{\label{Fig7.5}Examples of configurations in $\mathcal{N}^{k}$.}
\end{figure}
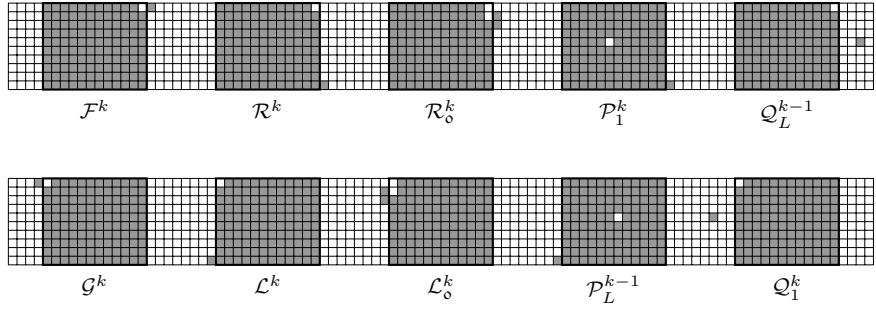

\begin{figure}
\begin{tikzpicture}
\fill[blue!40!white] (-0.5,2) rectangle (0.5,2.5); \draw[thick] (-0.5,2) rectangle (0.5,2.5);
\fill[blue!40!white] (-0.5,-2) rectangle (0.5,-2.5); \draw[thick] (-0.5,-2) rectangle (0.5,-2.5);
\draw (0,2.25) node{\small $\mathcal{R}_{\mathfrak{0}}^k$};
\draw (0,-2.25) node{\small $\mathcal{L}_{\mathfrak{0}}^k$};

\fill[black!20!white,rounded corners] (-0.5,1.5) rectangle (0.5,2); \draw[thick,rounded corners] (-0.5,1.5) rectangle (0.5,2);
\fill[black!20!white,rounded corners] (-0.5,-1.5) rectangle (0.5,-2); \draw[thick,rounded corners] (-0.5,-1.5) rectangle (0.5,-2);
\draw (0,1.75) node{\small $\mathcal{R}^k$};
\draw (0,-1.75) node{\small $\mathcal{L}^k$};

\foreach \j in {-1,1} {
\fill[blue!40!white] (-2,\j-0.5) rectangle (2,\j+0.5); \draw[thick] (-2,\j-0.5) rectangle (2,\j+0.5);
\draw[thick] (-0.25,\j-0.5)--(-0.25,\j+0.5); \draw[thick] (0.25,\j-0.5)--(0.25,\j+0.5); }
\draw (0,1) node{\small $\mathcal{F}^k$};
\draw (0,-1) node{\small $\mathcal{G}^k$};
\draw (1.125,1) node{\small $\mathcal{P}_1^k$};
\draw (-1.125,1) node{\small $\mathcal{Q}_L^{k-1}$};
\draw (-1.125,-1) node{\small $\mathcal{P}_L^{k-1}$};
\draw (1.125,-1) node{\small $\mathcal{Q}_1^k$};

\foreach \i in {-2.5,2.5} {
\fill[black!20!white,rounded corners] (\i-0.5,-1.5) rectangle (\i+0.5,1.5); \draw[thick,rounded corners] (\i-0.5,-1.5) rectangle (\i+0.5,1.5); }
\draw (2.5,0) node{\small $\mathcal{S}_1^k$};
\draw (-2.5,0) node{\small $\mathcal{S}_{L-1}^{k-1}$};

\draw[blue,thick] (0,-0.5)--(0,0.5);
\fill (0,0) circle (0.06);
\draw (0,0) node[right]{\small $\bm{\sigma^k}$};
\end{tikzpicture}\caption{\label{Fig7.6}Structure of $\mathcal{N}^{k}$ explained in Section
\ref{sec7.2.2}.}
\end{figure}
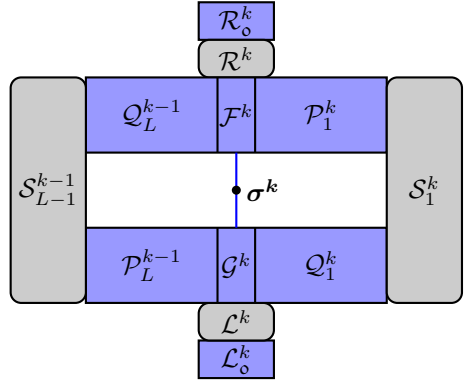

First, we present a list of collections in Table \ref{Tab1} (cf.
Definition \ref{def:bridge cross stick} and Notation \ref{nota:sitting stick})
which are illustrated in Figure \ref{Fig7.5}. In each table in this
article, the first column contains the names of the collections and
the second column, unless otherwise stated, gives the positions of
the $L\mathscr{N}_{0}$ particles of the configurations belonging
to the corresponding collection.

\begin{table}
\begin{tabular}{|c|l|}
\hline 
{\footnotesize{}$\mathcal{F}^{k}$} & {\footnotesize{}Reachable from $\bm{\sigma^{k}}$ by a horizontal
jump in $\mathfrak{c}^{[k+\mathscr{N}_{0}-1,\,k+\mathscr{N}_{0}]}$}\tabularnewline
\hline 
{\footnotesize{}$\mathcal{R}^{k}$} & {\footnotesize{}$\mathfrak{c}^{[k,\,k+\mathscr{N}_{0}-2]}\cup W\cup W'$
where $W\in\mathfrak{S}_{L-1}^{k+\mathscr{N}_{0}-1}$ and $W'\in\mathfrak{S}_{1}^{k+\mathscr{N}_{0}}$
such that $W\vdash W'$}\tabularnewline
\hline 
{\footnotesize{}$\mathcal{R}_{\mathfrak{0}}^{k}$} & {\footnotesize{}Reachable from $\mathcal{R}^{k}$ by a jump $\mathfrak{c}^{k+\mathscr{N}_{0}-1}\to\mathfrak{c}^{k+\mathscr{N}_{0}}$
within energy $\mathbb{H}_{0}+3$}\tabularnewline
\hline 
{\footnotesize{}$\mathcal{P}_{1}^{k}$} & {\footnotesize{}$\mathfrak{c}^{[k,\,k+\mathscr{N}_{0}-1]}\cup\{w\}\setminus\{w'\}$
where $w\in\mathfrak{c}^{k+\mathscr{N}_{0}}$ and $w'\in\mathfrak{c}^{[k+1,\,k+\mathscr{N}_{0}-2]}$}\tabularnewline
\hline 
{\footnotesize{}$\mathcal{Q}_{L}^{k-1}$} & {\footnotesize{}$\mathfrak{c}^{[k,\,k+\mathscr{N}_{0}-1]}\cup\{w\}\setminus\{w'\}$
where $w\in\mathfrak{c}^{[k+\mathscr{N}_{0}+1,\,k-2]}$ and $w'\in\mathfrak{c}^{k+\mathscr{N}_{0}-1}$}\tabularnewline
\hline 
\end{tabular}

\vspace{2mm}

\caption{\label{Tab1}Collections of configurations in $\mathcal{N}^{k}$.
We denote by $\mathcal{G}^{k}$, $\mathcal{L}^{k}$, $\mathcal{L}_{\mathfrak{0}}^{k}$,
$\mathcal{P}_{L}^{k-1}$ and $\mathcal{Q}_{1}^{k}$ the collections
obtained from $\mathcal{F}^{k}$, $\mathcal{R}^{k}$, $\mathcal{R}_{\mathfrak{0}}^{k}$,
$\mathcal{P}_{1}^{k}$ and $\mathcal{Q}_{L}^{k-1}$, respectively,
by horizontally reflecting the configurations with respect to the
vertical center line of the rectangle $\mathfrak{c}^{[k,\,k+\mathscr{N}_{0}-1]}$;
see Figure \ref{Fig7.5}.}
\end{table}

We claim that the following decomposition holds:
\begin{equation}
\mathcal{N}^{k}=\{\bm{\sigma^{k}}\}\cup\mathcal{S}_{1}^{k}\cup\mathcal{S}_{L-1}^{k-1}\cup(\mathcal{F}^{k}\cup\mathcal{R}^{k}\cup\mathcal{R}_{\mathfrak{0}}^{k}\cup\mathcal{P}_{1}^{k}\cup\mathcal{Q}_{L}^{k-1})\cup(\mathcal{G}^{k}\cup\mathcal{L}^{k}\cup\mathcal{L}_{\mathfrak{0}}^{k}\cup\mathcal{P}_{L}^{k-1}\cup\mathcal{Q}_{1}^{k}).\label{eq:Nk-dec}
\end{equation}
Starting from $\bm{\sigma^{k}}$, the first jump is either from a
particle in $\mathfrak{c}^{k+\mathscr{N}_{0}-1}$ to its right into
$\mathfrak{c}^{k+\mathscr{N}_{0}}$ (thereby enter collection $\mathcal{F}^{k}$)
or a particle in $\mathfrak{c}^{k}$ to its left into $\mathfrak{c}^{k-1}$
(enter collection $\mathcal{G}^{k}$). Thus, we divide into two cases.\smallskip{}

\begin{itemize}
\item Suppose that a $\mathfrak{c}^{k+\mathscr{N}_{0}-1}\to\mathfrak{c}^{k+\mathscr{N}_{0}}$-jump
occurs and we obtain a configuration in $\mathcal{F}^{k}$. Then,
excluding the possibility of returning to $\bm{\sigma^{k}}$, the
next jump occurs as one of the three following options: a vertical
jump occurs in $\mathfrak{c}^{[k+\mathscr{N}_{0}-1,\,k+\mathscr{N}_{0}]}$
(collection $\mathcal{R}^{k}$), or the empty site in $\mathfrak{c}^{k+\mathscr{N}_{0}-1}$
moves left and enters the sea of occupied sites (collection $\mathcal{P}_{1}^{k}$),
or the occupied site in $\mathfrak{c}^{k+\mathscr{N}_{0}}$ moves
right and enters the sea of empty sites (collection $\mathcal{Q}_{L}^{k-1}$).
\begin{itemize}
\item If we enter $\mathcal{R}^{k}$, then all configurations in $\mathcal{R}^{k}$
are connected to each other. When we escape $\mathcal{R}^{k}$, other
than returning to $\mathcal{F}^{k}$, we may visit $\mathcal{R}_{\mathfrak{0}}^{k}\cup\mathcal{P}_{1}^{k}\cup\mathcal{Q}_{L}^{k-1}$.
In particular, if we enter $\mathcal{R}_{\mathfrak{0}}^{k}$, the
only option is to return to $\mathcal{R}^{k}$.
\item If we enter $\mathcal{P}_{1}^{k}$, then a new type emerges only if
the empty site travels left and hits $\mathfrak{c}^{k}$, thereby
obtaining $\mathcal{S}_{1}^{k}$. After then, all configurations in
$\mathcal{S}_{1}^{k}$ are reachable without escaping $\mathcal{S}_{1}^{k}$.
When we escape $\mathcal{S}_{1}^{k}$, either the empty site in $\mathfrak{c}^{k}$
may enter the sea of occupied sites (return to $\mathcal{P}_{1}^{k}$),
or the occupied site in $\mathfrak{c}^{k+\mathscr{N}_{0}}$ may enter
the sea of empty sites (visit $\mathcal{Q}_{1}^{k}$).
\item If we enter $\mathcal{Q}_{L}^{k-1}$, then a new type emerges only
if the occupied site travels right and hits $\mathfrak{c}^{k-1}$,
thereby obtaining $\mathcal{S}_{L-1}^{k-1}$. After then, the options
are to return to $\mathcal{Q}_{L}^{k-1}$ or to visit $\mathcal{P}_{L}^{k-1}$.\smallskip{}
\end{itemize}
\item By the same logic, we may handle the other case when a $\mathfrak{c}^{k}\to\mathfrak{c}^{k-1}$-jump
occurs and we obtain a configuration in $\mathcal{G}^{k}$.\smallskip{}
\end{itemize}
The above classification exhausts all possible types of configurations
in $\mathcal{N}^{k}$, thereby verifying \eqref{eq:Nk-dec}. Refer
to Figure \ref{Fig7.6} for the overall structure of $\mathcal{N}^{k}$.
In particular, the following lemma is immediate.
\begin{lem}
\label{lem:Nk stab plat}The five stable plateaux in $\mathcal{N}^{k}$
are $\{\bm{\sigma^{k}}\}$, $\mathcal{S}_{1}^{k}$, $\mathcal{S}_{L-1}^{k-1}$,
$\mathcal{R}^{k}$ and $\mathcal{L}^{k}$. Moreover, $\mathbb{H}(\mathcal{R}^{k})=\mathbb{H}(\mathcal{L}^{k})=\mathbb{H}_{0}+2$.
\end{lem}

As a byproduct, we obtain the following lower-bound part of Theorem
\ref{thm:E barrier}.
\begin{lem}
\label{lem:E barrier LB}For distinct $k,\,k'\in\mathbb{T}_{K}$,
it holds that $\Phi(\bm{\sigma^{k}},\,\bm{\sigma^{k'}})\ge\mathbb{H}_{0}+4$.
\end{lem}

\begin{proof}
All types of configurations classified in Section \ref{sec7.2.2}
are explicit, thus any configuration cannot belong to both $\mathcal{N}^{k}$
and $\mathcal{N}^{k'}$ for $k\ne k'$. This implies that $\mathcal{N}^{k}\cap\mathcal{N}^{k'}=\emptyset$
which is equivalent to $\Phi(\bm{\sigma^{k}},\,\bm{\sigma^{k'}})\ge\mathbb{H}_{0}+4$.
\end{proof}
\begin{proof}[Proof of Theorem \ref{thm:E barrier}]
 Lemmas \ref{lem:E barrier UB} and \ref{lem:E barrier LB} readily
verify the theorem.
\end{proof}

\subsection{\label{sec7.3}Passages}

In this subsection, we define the transition passages. Refer to Figure
\ref{Fig7.7}.

\begin{figure}
\begin{tikzpicture}[scale=0.115]
\fill[white] (-4,0) rectangle (76,10);

\fill[black!40!white] (1,0) rectangle (12,10);
\fill[black!40!white] (0,2) rectangle (1,7);
\fill[black!40!white] (12,2) rectangle (13,8);
\fill[white] (4,3) rectangle (5,4);
\draw (6,0) node[below]{\scriptsize $\mathcal{P}_{6;2,2}^k$};

\fill[black!40!white] (21,0) rectangle (32,10);
\fill[black!40!white] (20,5) rectangle (21,9);
\fill[black!40!white] (32,1) rectangle (33,8);
\fill[white] (29,7) rectangle (30,8);
\draw (26,0) node[below]{\scriptsize $\mathcal{P}_{7;5,1}^k$};

\fill[black!40!white] (41,0) rectangle (52,10);
\fill[black!40!white] (40,4) rectangle (41,9);
\fill[black!40!white] (52,5) rectangle (53,9);
\fill[black!40!white] (54,3) rectangle (55,4);
\draw (46,0) node[below]{\scriptsize $\mathcal{Q}_{5;4,5}^k$};

\fill[black!40!white] (61,0) rectangle (72,10);
\fill[black!40!white] (60,4) rectangle (61,7);
\fill[black!40!white] (72,2) rectangle (73,8);
\fill[black!40!white] (74,1) rectangle (75,2);
\draw (66,0) node[below]{\scriptsize $\mathcal{Q}_{7;4,2}^k$};

\draw[very thin] (-4,0) grid (76,10);
\foreach \i in {0,...,3} {
\draw[thick] (0+20*\i,0) rectangle (12+20*\i,10);
}
\end{tikzpicture}\caption{\label{Fig7.7}Examples of configurations belonging to $\mathcal{P}_{m}^{k}$
(left two) defined in Definition \ref{def:passage P} and to $\mathcal{Q}_{m}^{k}$
(right two) defined in Definition \ref{def:passage Q}.}
\end{figure}
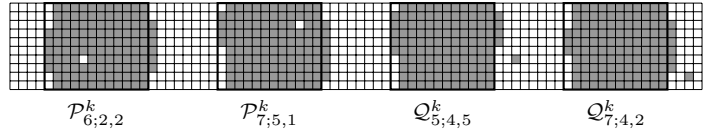

\begin{defn}[Major passages $\mathcal{P}_{m}^{k}$]
\label{def:passage P}For $m\in[2,\,L-1]$ and $\ell,\,\ell'\in\mathbb{T}_{L}$,
let $\mathcal{P}_{m;\ell,\ell'}^{k}\subseteq\Omega$ be the collection
of configurations $\eta$ such that
\[
\{x:\,\eta(x)=1\}=\mathfrak{c}^{[k+1,\,k+\mathscr{N}_{0}-1]}\cup\mathfrak{s}_{L-m+1;\ell}^{k}\cup\mathfrak{s}_{m;\ell'}^{k+\mathscr{N}_{0}}\setminus\{w_{\eta}\},
\]
where $w_{\eta}$ is a site in $\mathfrak{c}^{[k+2,\,k+\mathscr{N}_{0}-2]}$.
Then, we collect
\[
\mathcal{P}_{m}^{k}:=\bigcup_{\ell,\,\ell'\in\mathbb{T}_{L}}\mathcal{P}_{m;\ell,\ell'}^{k}.
\]
By a simple calculation, it holds that $\mathbb{H}(\eta)=\mathbb{H}_{0}+4$
for all $\eta\in\mathcal{P}_{m}^{k}$.
\end{defn}

\begin{defn}[Minor passages $\mathcal{Q}_{m}^{k}$]
\label{def:passage Q}For $m\in[2,\,L-1]$ and $\ell,\,\ell'\in\mathbb{T}_{L}$,
let $\mathcal{Q}_{m;\ell,\ell'}^{k}\subseteq\Omega$ be the collection
of configurations $\eta$ such that
\[
\{x:\,\eta(x)=1\}=\mathfrak{c}^{[k+1,\,k+\mathscr{N}_{0}-1]}\cup\mathfrak{s}_{L-m;\ell}^{k}\cup\mathfrak{s}_{m-1;\ell'}^{k+\mathscr{N}_{0}}\cup\{w_{\eta}\},
\]
where $w_{\eta}$ is a site in $\mathfrak{c}^{[k+\mathscr{N}_{0}+2,\,k-2]}$.
Similarly,
\[
\mathcal{Q}_{m}^{k}:=\bigcup_{\ell,\,\ell'\in\mathbb{T}_{L}}\mathcal{Q}_{m;\ell,\ell'}^{k}.
\]
Again, it holds that $\mathbb{H}(\eta)=\mathbb{H}_{0}+4$ for all
$\eta\in\mathcal{Q}_{m}^{k}$.
\end{defn}

Here, notice that the number of possible horizontal locations of $w_{\eta}$
for each $\eta\in\mathcal{P}_{m;\ell,\ell'}^{k}$ is $\mathscr{N}_{0}-3$,
whereas the corresponding number for $\eta\in\mathcal{Q}_{m;\ell,\ell'}^{k}$
is $K-\mathscr{N}_{0}-3$. Since $\mathscr{N}_{0}-3<K-\mathscr{N}_{0}-3$
according to the assumption in \eqref{eq:N0 assump}, we deduce intuitively
that in the course of typical transitions, $\mathcal{P}_{m;\ell,\ell'}^{k}$
is more likely to be chosen than $\mathcal{Q}_{m;\ell,\ell'}^{k}$.
This is why we call $\mathcal{P}_{m}^{k}$ and $\mathcal{Q}_{m}^{k}$
the collections of \emph{major} and \emph{minor} passage configurations,
respectively.
\begin{lem}
\label{lem:passage}For all $m\in[2,\,L-1]$, it holds that $\mathcal{P}_{m}^{k}\subseteq\widehat{\mathcal{N}}(\mathcal{S})$
and $\mathcal{Q}_{m}^{k}\subseteq\widehat{\mathcal{N}}(\mathcal{S})$.
\end{lem}

\begin{proof}
First, we prove that $\mathcal{P}_{m}^{k}\subseteq\widehat{\mathcal{N}}(\mathcal{S})$
for $m\in[2,\,L-1]$. Starting from a configuration $\eta\in\mathcal{P}_{m;\ell,\ell'}^{k}$
for $\ell,\,\ell'\in\mathbb{T}_{L}$, the empty site in $\mathfrak{c}^{[k+2,\,k+\mathscr{N}_{0}-2]}$
(denoted as $w_{\eta}$ in Definition \ref{def:passage P}) can move
inside the sea of occupied sites without changing the energy. If the
empty site hits the boundary of the sea at $(k,\,\ell+L-m)$, then
the resulting configuration is $\bm{\sigma_{m;\ell,\ell'}^{k}}$.
This implies that there exists an allowed path (cf. Notation \ref{nota:allowed jump path})
from $\eta$ to $\bm{\sigma_{m;\ell,\ell'}^{k}}$. Since $\bm{\sigma_{m;\ell,\ell'}^{k}}\in\widehat{\mathcal{N}}(\mathcal{S})$
by Lemma \ref{lem:shallow bot}-(3), it holds that $\eta\in\widehat{\mathcal{N}}(\mathcal{S})$.
Thus, we conclude that $\mathcal{P}_{m}^{k}\subseteq\widehat{\mathcal{N}}(\mathcal{S})$.
By switching the roles of spins $0$ and $1$, the same strategy applies
to $\mathcal{Q}_{m}^{k}$ as well; we omit the detail.
\end{proof}
\begin{defn}
\label{lem:way}Collect 
\begin{equation}
\mathcal{W}:=\bigcup_{k\in\mathbb{T}_{K}}\bigcup_{m=2}^{L-1}(\mathcal{P}_{m}^{k}\cup\mathcal{Q}_{m}^{k}).\label{eq:W def}
\end{equation}
\end{defn}

By the definitions of $\mathcal{V}$ and $\mathcal{W}$ in Definitions
\ref{def:energetic valleys} and \ref{lem:way}, due to Lemmas \ref{lem:shallow bot}-(3)
and \ref{lem:passage}, we have $\mathcal{V}\cup\mathcal{W}\subseteq\widehat{\mathcal{N}}(\mathcal{S})$.
Thus, since $\mathcal{S}\subseteq\bigcup_{k\in\mathbb{T}_{K}}\mathcal{N}^{k}\subseteq\mathcal{V}$,
it holds that
\[
\widehat{\mathcal{N}}(\mathcal{S})\subseteq\widehat{\mathcal{N}}(\mathcal{V}\cup\mathcal{W})\subseteq\widehat{\mathcal{N}}(\mathcal{S}),\quad\text{thus}\quad\widehat{\mathcal{N}}(\mathcal{S})=\widehat{\mathcal{N}}(\mathcal{V}\cup\mathcal{W}).
\]
Then by \eqref{eq:Omega-bar N-hatS} and Lemma \ref{lem:nbd dec}-(2)
with $\mathcal{A}=\mathcal{V}$ and $\mathcal{B}=\mathcal{W}$, it
holds that
\[
\overline{\Omega}=\widehat{\mathcal{N}}(\mathcal{V};\mathcal{W})\cup\widehat{\mathcal{N}}(\mathcal{W};\widehat{\mathcal{N}}(\mathcal{V};\mathcal{W})).
\]
Rearranging using Definitions \ref{def:energetic valleys} and \ref{lem:way},
we obtain that
\begin{equation}
\begin{aligned}\overline{\Omega}= & \bigcup_{k\in\mathbb{T}_{K}}\widehat{\mathcal{N}}(\mathcal{N}^{k};\mathcal{W})\cup\bigcup_{k\in\mathbb{T}_{K}}\bigcup_{m=2}^{L-2}\widehat{\mathcal{N}}(\mathcal{N}_{m}^{k};\mathcal{W})\\
 & \cup\bigcup_{k\in\mathbb{T}_{K}}\bigcup_{m=2}^{L-1}\widehat{\mathcal{N}}(\mathcal{P}_{m}^{k};\widehat{\mathcal{N}}(\mathcal{V};\mathcal{W}))\cup\bigcup_{k\in\mathbb{T}_{K}}\bigcup_{m=2}^{L-1}\widehat{\mathcal{N}}(\mathcal{Q}_{m}^{k};\widehat{\mathcal{N}}(\mathcal{V};\mathcal{W})).
\end{aligned}
\label{eq:Omega-bar dec}
\end{equation}
According to this expression, we define the following collections.
\begin{defn}
\label{def:E landscape}Define 
\[
\widehat{\mathcal{N}}^{k}:=\widehat{\mathcal{N}}(\mathcal{N}^{k};\mathcal{W})\quad\text{and}\quad\widehat{\mathcal{N}}_{m}^{k}:=\widehat{\mathcal{N}}(\mathcal{N}_{m}^{k};\mathcal{W})\quad\text{for}\quad m\in[1,\,L-1].
\]
Moreover, for each $m\in[2,\,L-1]$, define
\[
\widehat{\mathcal{P}}_{m}^{k}:=\widehat{\mathcal{N}}(\mathcal{P}_{m}^{k};\widehat{\mathcal{N}}(\mathcal{V};\mathcal{W}))\quad\text{and}\quad\widehat{\mathcal{Q}}_{m}^{k}:=\widehat{\mathcal{N}}(\mathcal{Q}_{m}^{k};\widehat{\mathcal{N}}(\mathcal{V};\mathcal{W})).
\]
Then, by \eqref{eq:Omega-bar dec}, it holds that
\begin{equation}
\overline{\Omega}=\bigcup_{k\in\mathbb{T}_{K}}\widehat{\mathcal{N}}^{k}\cup\bigcup_{k\in\mathbb{T}_{K}}\bigcup_{m=2}^{L-2}\widehat{\mathcal{N}}_{m}^{k}\cup\bigcup_{k\in\mathbb{T}_{K}}\bigcup_{m=2}^{L-1}\widehat{\mathcal{P}}_{m}^{k}\cup\bigcup_{k\in\mathbb{T}_{K}}\bigcup_{m=2}^{L-1}\widehat{\mathcal{Q}}_{m}^{k}.\label{eq:E landscape dec}
\end{equation}
\end{defn}

Now, we analyze the local geometry near each collection in the right-hand
side of \eqref{eq:E landscape dec}. In the remainder of this subsection,
we investigate the passage part of $\overline{\Omega}$, which is
\[
\bigcup_{k\in\mathbb{T}_{K}}\bigcup_{m=2}^{L-1}\widehat{\mathcal{P}}_{m}^{k}\cup\bigcup_{k\in\mathbb{T}_{K}}\bigcup_{m=2}^{L-1}\widehat{\mathcal{Q}}_{m}^{k}.
\]
It turns out that the structure is quite simple in the passage part.

\subsubsection{\label{sec7.3.1}Bulk collections}

First, we start with the bulk passage parts $\widehat{\mathcal{P}}_{m}^{k}$
and $\widehat{\mathcal{Q}}_{m}^{k}$ for $m\in[3,\,L-2]$. To start,
we focus on $\widehat{\mathcal{P}}_{m}^{k}$.

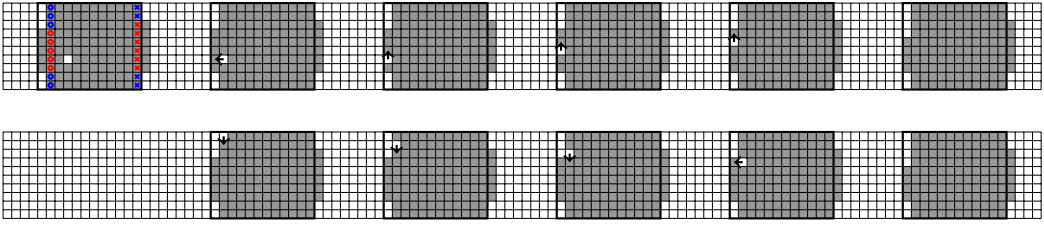
\begin{figure}
\begin{tikzpicture}[scale=0.115]
\fill[white] (-4,0) rectangle (116,10);

\fill[black!40!white] (1,0) rectangle (12,10);
\fill[black!40!white] (0,2) rectangle (1,7);
\fill[black!40!white] (12,2) rectangle (13,8);
\fill[white] (3,3) rectangle (4,4);
\foreach \i in {2,...,6} {
\draw[thick,red] (1.5,0.5+\i) circle (0.25);
}
\foreach \i in {2,...,7} {
\draw[thick,red] (11.25,0.25+\i)--(11.75,0.75+\i); \draw[thick,red] (11.25,0.75+\i)--(11.75,0.25+\i);
}
\foreach \i in {0,1,7,8,9} {
\draw[thick,blue] (1.5,0.5+\i) circle (0.25);
}
\foreach \i in {0,1,8,9} {
\draw[thick,blue] (11.25,0.25+\i)--(11.75,0.75+\i); \draw[thick,blue] (11.25,0.75+\i)--(11.75,0.25+\i);
}

\fill[black!40!white] (21,0) rectangle (32,10);
\fill[black!40!white] (20,2) rectangle (21,7);
\fill[black!40!white] (32,2) rectangle (33,8);
\fill[white] (21,3) rectangle (22,4);
\draw[->,thick] (21.5,3.5)--(20.5,3.5);

\fill[black!40!white] (41,0) rectangle (52,10);
\fill[black!40!white] (40,2) rectangle (41,7);
\fill[black!40!white] (52,2) rectangle (53,8);
\fill[white] (40,3) rectangle (41,4);
\draw[->,thick] (40.5,3.5)--(40.5,4.5);

\fill[black!40!white] (61,0) rectangle (72,10);
\fill[black!40!white] (60,2) rectangle (61,7);
\fill[black!40!white] (72,2) rectangle (73,8);
\fill[white] (60,4) rectangle (61,5);
\draw[->,thick] (60.5,4.5)--(60.5,5.5);

\fill[black!40!white] (81,0) rectangle (92,10);
\fill[black!40!white] (80,2) rectangle (81,7);
\fill[black!40!white] (92,2) rectangle (93,8);
\fill[white] (80,5) rectangle (81,6);
\draw[->,thick] (80.5,5.5)--(80.5,6.5);

\fill[black!40!white] (101,0) rectangle (112,10);
\fill[black!40!white] (100,2) rectangle (101,6);
\fill[black!40!white] (112,2) rectangle (113,8);

\draw[very thin] (-4,0) grid (116,10);
\foreach \i in {0,...,5} {
\draw[thick] (0+20*\i,0) rectangle (12+20*\i,10);
}
\end{tikzpicture}\\
\vspace{5mm}
\begin{tikzpicture}[scale=0.115]
\fill[white] (-4,0) rectangle (116,10);

\fill[black!40!white] (21,0) rectangle (32,10);
\fill[black!40!white] (20,2) rectangle (21,7);
\fill[black!40!white] (32,2) rectangle (33,8);
\fill[white] (21,9) rectangle (22,10);
\draw[->,thick] (21.5,9.5)--(21.5,8.5);

\fill[black!40!white] (41,0) rectangle (52,10);
\fill[black!40!white] (40,2) rectangle (41,7);
\fill[black!40!white] (52,2) rectangle (53,8);
\fill[white] (41,8) rectangle (42,9);
\draw[->,thick] (41.5,8.5)--(41.5,7.5);

\fill[black!40!white] (61,0) rectangle (72,10);
\fill[black!40!white] (60,2) rectangle (61,7);
\fill[black!40!white] (72,2) rectangle (73,8);
\fill[white] (61,7) rectangle (62,8);
\draw[->,thick] (61.5,7.5)--(61.5,6.5);

\fill[black!40!white] (81,0) rectangle (92,10);
\fill[black!40!white] (80,2) rectangle (81,7);
\fill[black!40!white] (92,2) rectangle (93,8);
\fill[white] (81,6) rectangle (82,7);
\draw[->,thick] (81.5,6.5)--(80.5,6.5);

\fill[black!40!white] (101,0) rectangle (112,10);
\fill[black!40!white] (100,2) rectangle (101,6);
\fill[black!40!white] (112,2) rectangle (113,8);

\draw[very thin] (-4,0) grid (116,10);
\foreach \i in {1,...,5} {
\draw[thick] (0+20*\i,0) rectangle (12+20*\i,10);
}
\end{tikzpicture}\caption{\label{Fig7.8}Configuration $\eta\in\mathcal{P}_{m;\ell,\ell'}^{k}$
(first figure) and allowed paths from $\xi$ to $\bm{\sigma_{m+1;\ell,\ell'}^{k}}$
avoiding $\mathcal{W}$ (first and second lines) explained in Section
\ref{sec7.3.1}.}
\end{figure}

We classify all configurations reachable via an allowed path starting
from $\mathcal{P}_{m}^{k}$ and avoiding $\widehat{\mathcal{N}}(\mathcal{V};\mathcal{W})$.
Fix a configuration $\eta\in\mathcal{P}_{m;\ell,\ell'}^{k}$ for some
$\ell,\,\ell'\in\mathbb{T}_{L}$ and note that $\mathbb{H}(\eta)=\mathbb{H}_{0}+4$.
To maintain the energy lower than or equal to $\mathbb{H}_{0}+4$,
only the isolated empty site (indicated as $w_{\eta}$ in Definition
\ref{def:passage P}) can move inside the rectangle $\mathfrak{c}^{[k+2,\,k+\mathscr{N}_{0}-2]}$.
These movements produce configurations still contained in $\mathcal{P}_{m;\ell,\ell'}^{k}$
unless the empty site escapes the rectangle and hits $\mathfrak{c}^{k+1}\cup\mathfrak{c}^{k+\mathscr{N}_{0}-1}$
(marked red and blue in the first figure of Figure \ref{Fig7.8}),
which can be decomposed as follows:
\begin{equation}
\mathfrak{c}^{k+1}\cup\mathfrak{c}^{k+\mathscr{N}_{0}-1}=\mathfrak{s}_{L-m+1;\ell}^{k+1}\cup\mathfrak{s}_{m-1;\ell-m+1}^{k+1}\cup\mathfrak{s}_{m;\ell'}^{k+\mathscr{N}_{0}-1}\cup\mathfrak{s}_{L-m;\ell'+m}^{k+\mathscr{N}_{0}-1}.\label{eq:hit bdry}
\end{equation}
Suppose first that the empty site hits the first set $\mathfrak{s}_{L-m+1;\ell}^{k+1}$
in the right-hand side of \eqref{eq:hit bdry} (marked with red circles
in Figure \ref{Fig7.8}). Denote by $\xi$ the resulting configuration.
Then, we can further move the empty site left to enter $\mathfrak{s}_{L-m+1;\ell}^{k}$
and then move it upwards until it hits $(k,\,\ell+L-m)$, so that
the resulting configuration is $\bm{\sigma_{m;\ell,\ell'}^{k}}$ (see
the first line of figures in Figure \ref{Fig7.8} for the path). This
is an allowed path from $\xi$ to $\bm{\sigma_{m;\ell,\ell'}^{k}}\in\mathcal{N}_{m}^{k}\subseteq\mathcal{V}$
avoiding $\mathcal{W}$. Hence, it holds that $\xi\in\widehat{\mathcal{N}}(\mathcal{V};\mathcal{W})$,
thus $\xi\notin\widehat{\mathcal{N}}(\mathcal{P}_{m}^{k};\widehat{\mathcal{N}}(\mathcal{V};\mathcal{W}))=\widehat{\mathcal{P}}_{m}^{k}$.
In particular, since $m\in[3,\,L-2]$ and $\xi\notin\mathcal{N}_{m}^{k}$,
by Definition \ref{def:E landscape} it holds that
\[
\xi\in\widehat{\mathcal{N}}(\mathcal{N}_{m}^{k};\mathcal{W})\setminus\mathcal{N}_{m}^{k}=\widehat{\mathcal{N}}_{m}^{k}\setminus\mathcal{N}_{m}^{k}.
\]
Similarly, if the empty site hits the third set $\mathfrak{s}_{m;\ell'}^{k+\mathscr{N}_{0}-1}$
in the right-hand side of \eqref{eq:hit bdry} (marked with red crosses
in Figure \ref{Fig7.8}), it holds that $\xi\notin\widehat{\mathcal{P}}_{m}^{k}$
and $\xi\in\widehat{\mathcal{N}}_{m-1}^{k}\setminus\mathcal{N}_{m-1}^{k}$.

Next, suppose that the empty site hits the second set $\mathfrak{s}_{m-1;\ell-m+1}^{k+1}$
in the right-hand side of \eqref{eq:hit bdry} (marked with blue circles
in Figure \ref{Fig7.8}). Then, we can further move the empty site
downwards until it hits $(k+1,\,\ell-m)$ and then move it left, so
that the resulting configuration is again $\bm{\sigma_{m;\ell,\ell'}^{k}}$
(see the second line of figures in Figure \ref{Fig7.8} for the path).
By the same logic, we obtain that $\xi\notin\widehat{\mathcal{P}}_{m}^{k}$
and $\xi\in\widehat{\mathcal{N}}_{m}^{k}\setminus\mathcal{N}_{m}^{k}$.
Finally, if the empty site hits the fourth set $\mathfrak{s}_{L-m;\ell'+m}^{k+\mathscr{N}_{0}-1}$
in the right-hand side of \eqref{eq:hit bdry} (marked with blue crosses
in Figure \ref{Fig7.8}), then the same strategy applies and we obtain
$\xi\notin\widehat{\mathcal{P}}_{m}^{k}$ and $\xi\in\widehat{\mathcal{N}}_{m-1}^{k}\setminus\mathcal{N}_{m-1}^{k}$.

The above deductions imply that any allowed path starting from $\mathcal{P}_{m}^{k}$
and avoiding $\widehat{\mathcal{N}}(\mathcal{V};\mathcal{W})$ either
stays in $\mathcal{P}_{m}^{k}$ or escapes $\widehat{\mathcal{P}}_{m}^{k}$
and visits either $\widehat{\mathcal{N}}_{m}^{k}\setminus\mathcal{N}_{m}^{k}$
or $\widehat{\mathcal{N}}_{m-1}^{k}\setminus\mathcal{N}_{m-1}^{k}$.

Switching the roles of spins $0$ and $1$, collection $\widehat{\mathcal{Q}}_{m}^{k}$
can be analyzed in the exact same way. As a result, any allowed path
from $\mathcal{Q}_{m}^{k}$ and avoiding $\widehat{\mathcal{N}}(\mathcal{V};\mathcal{W})$
either stays in $\mathcal{Q}_{m}^{k}$ or escapes $\widehat{\mathcal{Q}}_{m}^{k}$
and visits either $\widehat{\mathcal{N}}_{m}^{k}\setminus\mathcal{N}_{m}^{k}$
or $\widehat{\mathcal{N}}_{m-1}^{k}\setminus\mathcal{N}_{m-1}^{k}$.

\subsubsection{\label{sec7.3.2}Edge collections}

Here, we consider the edge passage parts $\widehat{\mathcal{P}}_{2}^{k}$,
$\widehat{\mathcal{P}}_{L-1}^{k}$, $\widehat{\mathcal{Q}}_{2}^{k}$
and $\widehat{\mathcal{Q}}_{L-1}^{k}$.

First, consider $\widehat{\mathcal{P}}_{2}^{k}$. In this special
case of $m=2$ in $\widehat{\mathcal{P}}_{m}^{k}$, the same deduction
described in Section \ref{sec7.3.1} works well such that the path
stays in $\mathcal{P}_{2}^{k}$ unless the empty site hits\footnote{In fact, there is an additional possibility that the only empty site
in $\mathfrak{c}^{k}$ slides vertically without changing the energy,
but the resulting configurations are still contained in $\mathcal{P}_{2}^{k}$
unless the empty site hits $\mathfrak{c}^{k+1}\cup\mathfrak{c}^{k+\mathscr{N}_{0}-1}$.}
\begin{equation}
\mathfrak{s}_{L-1;\ell}^{k+1}\cup\mathfrak{s}_{1;\ell-1}^{k+1}\cup\mathfrak{s}_{2;\ell'}^{k+\mathscr{N}_{0}-1}\cup\mathfrak{s}_{L-2;\ell'+2}^{k+\mathscr{N}_{0}-1}.\label{eq:hit bdry 2}
\end{equation}
After hitting the sets in \eqref{eq:hit bdry 2}, the exact same strategy
applies as well to verify that any path starting from $\mathcal{P}_{2}^{k}$
and avoiding $\widehat{\mathcal{N}}(\mathcal{V};\mathcal{W})$ either
stays in $\mathcal{P}_{2}^{k}$ or escapes $\widehat{\mathcal{P}}_{2}^{k}$
and visits either $\widehat{\mathcal{N}}^{k}\setminus\mathcal{N}^{k}$
or $\widehat{\mathcal{N}}_{2}^{k}\setminus\mathcal{N}_{2}^{k}$. Moreover,
the other collection $\widehat{\mathcal{P}}_{L-1}^{k}$ can be analyzed
in the exact same way such that any path starting from $\mathcal{P}_{L-1}^{k}$
and avoiding $\widehat{\mathcal{N}}(\mathcal{V};\mathcal{W})$ either
stays in $\mathcal{P}_{L-1}^{k}$ or escapes $\widehat{\mathcal{P}}_{L-1}^{k}$
and visits either $\widehat{\mathcal{N}}_{L-2}^{k}\setminus\mathcal{N}_{L-2}^{k}$
or $\widehat{\mathcal{N}}^{k+1}\setminus\mathcal{N}^{k+1}$. We may
apply the same deduction to the minor collections as well.

Gathering all the observations, we summarize the results regarding
the passage part as follows.
\begin{lem}
\label{lem:Pkm Qkm stab plat}It holds that
\[
\widehat{\mathcal{P}}_{m}^{k}=\mathcal{P}_{m}^{k}\quad\text{and}\quad\widehat{\mathcal{Q}}_{m}^{k}=\mathcal{Q}_{m}^{k}\quad\text{for each}\quad m\in[2,\,L-1].
\]
Moreover, there is no stable plateau in $\widehat{\mathcal{P}}_{m}^{k}\cup\widehat{\mathcal{Q}}_{m}^{k}$
for all $m\in[2,\,L-1]$.
\end{lem}

\subsection{\label{sec7.4}Bulk part}

In this subsection, we investigate the collections $\widehat{\mathcal{N}}_{m}^{k}=\widehat{\mathcal{N}}(\mathcal{N}_{m}^{k};\mathcal{W})$
for $m\in[2,\,L-2]$ that appear in the right-hand side of \eqref{eq:E landscape dec}.

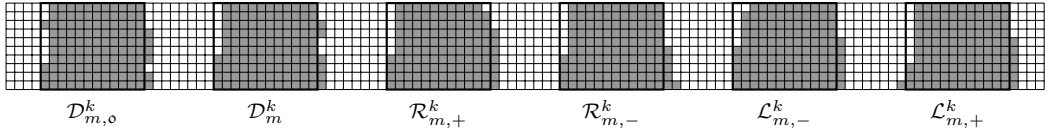
\begin{figure}
\begin{tikzpicture}[scale=0.115]
\fill[white] (-4,0) rectangle (116,10);

\fill[black!40!white] (1,0) rectangle (12,10);
\fill[black!40!white] (0,0) rectangle (1,3);
\fill[black!40!white] (0,7) rectangle (1,8);
\fill[black!40!white] (12,0) rectangle (13,2);
\fill[black!40!white] (12,3) rectangle (13,7);
\draw (6,0) node[below]{\scriptsize $\mathcal{D}_{m,\mathfrak{0}}^k$};

\fill[black!40!white] (21,0) rectangle (32,10);
\fill[black!40!white] (20,0) rectangle (21,4);
\fill[black!40!white] (32,0) rectangle (33,4);
\fill[black!40!white] (32,6) rectangle (33,8);
\draw (26,0) node[below]{\scriptsize $\mathcal{D}_m^k$};

\fill[black!40!white] (41,0) rectangle (52,10);
\fill[black!40!white] (40,0) rectangle (41,4);
\fill[black!40!white] (52,0) rectangle (53,7);
\fill[white] (51,9) rectangle (52,10);
\draw (46,0) node[below]{\scriptsize $\mathcal{R}_{m,+}^k$};

\fill[black!40!white] (61,0) rectangle (72,10);
\fill[black!40!white] (60,0) rectangle (61,4);
\fill[black!40!white] (72,0) rectangle (73,5);
\fill[black!40!white] (73,0) rectangle (74,1);
\draw (66,0) node[below]{\scriptsize $\mathcal{R}_{m,-}^k$};

\fill[black!40!white] (81,0) rectangle (92,10);
\fill[black!40!white] (80,0) rectangle (81,5);
\fill[black!40!white] (92,0) rectangle (93,6);
\fill[white] (81,9) rectangle (82,10);
\draw (86,0) node[below]{\scriptsize $\mathcal{L}_{m,-}^k$};

\fill[black!40!white] (101,0) rectangle (112,10);
\fill[black!40!white] (99,0) rectangle (100,1);
\fill[black!40!white] (100,0) rectangle (101,3);
\fill[black!40!white] (112,0) rectangle (113,6);
\draw (106,0) node[below]{\scriptsize $\mathcal{L}_{m,+}^k$};

\draw[very thin] (-4,0) grid (116,10);
\foreach \i in {0,...,5} {
\draw[thick] (0+20*\i,0) rectangle (12+20*\i,10);
}
\end{tikzpicture}\caption{\label{Fig7.9}Examples of configurations in $\widehat{\mathcal{N}}_{m}^{k}$
for $m=6$.}
\end{figure}

\begin{figure}
\begin{tikzpicture}
\foreach \i in {-2.5,2.5} { \foreach \j in {-1.25,1.25} {
\fill[black!20!white,rounded corners] (\i-0.5,\j-0.25) rectangle (\i+0.5,\j+0.25); \draw[thick,rounded corners] (\i-0.5,\j-0.25) rectangle (\i+0.5,\j+0.25); }}
\draw(2.5,1.25) node{\small $\mathcal{R}_{m,+}^k$};
\draw(-2.5,-1.25) node{\small $\mathcal{R}_{m,-}^k$};
\draw(2.5,-1.25) node{\small $\mathcal{L}_{m,+}^k$};
\draw(-2.5,1.25) node{\small $\mathcal{L}_{m,-}^k$};

\foreach \i in {-2,2} { \foreach \j in {-0.75,0.75} {
\fill[red!40!white] (\i-1,\j-0.25) rectangle (\i+1,\j+0.25); \draw[thick] (\i-1,\j-0.25) rectangle (\i+1,\j+0.25); }}
\draw (2,0.75) node{\small $\mathcal{P}_{m,+}^k$};
\draw (-2,-0.75) node{\small $\mathcal{Q}_{m,-}^k$};
\draw (-2,0.75) node{\small $\mathcal{P}_{m,-}^k$};
\draw (2,-0.75) node{\small $\mathcal{Q}_{m,+}^k$};

\foreach \i in {-4,4} { \foreach \j in {-1,1} {
\fill[white] (\i-1,\j+0.5) rectangle (\i+1,\j-0.5); \draw[thick] (\i-1,\j+0.5)--(\i+1,\j+0.5); \draw[thick] (\i-1,\j-0.5)--(\i+1,\j-0.5); }}
\draw[thick] (3,1.5)--(3,0.5); \draw[thick] (3,-1.5)--(3,-0.5); \draw[thick] (-3,1.5)--(-3,0.5); \draw[thick] (-3,-1.5)--(-3,-0.5);
\draw (4,1) node{\small $\mathcal{P}^k_{m+1}$};
\draw (4,-1) node{\small $\mathcal{Q}^k_{m+1}$};
\draw (-4,1) node{\small $\mathcal{P}^k_m$};
\draw (-4,-1) node{\small $\mathcal{Q}^k_m$};

\fill[black!20!white,rounded corners] (-0.5,2) rectangle (0.5,2.5); \draw[thick,rounded corners] (-0.5,2) rectangle (0.5,2.5);
\draw (0,2.25) node{\small $\mathcal{D}_m^k$};

\fill[red!40!white] (-0.5,1.5) rectangle (0.5,2); \draw[thick] (-0.5,1.5) rectangle (0.5,2);
\draw (0,1.75) node{\small $\mathcal{D}_{m,\mathfrak{0}}^k$};

\fill[black!20!white,rounded corners] (1,1.5) rectangle (-1,-1.5); \draw[thick,rounded corners] (1,1.5) rectangle (-1,-1.5);
\draw (0,0) node{\small $\mathcal{N}_m^k$};
\end{tikzpicture}\caption{\label{Fig7.10}Structure of $\widehat{\mathcal{N}}_{m}^{k}$ explained
in Section \ref{sec7.4}.}
\end{figure}
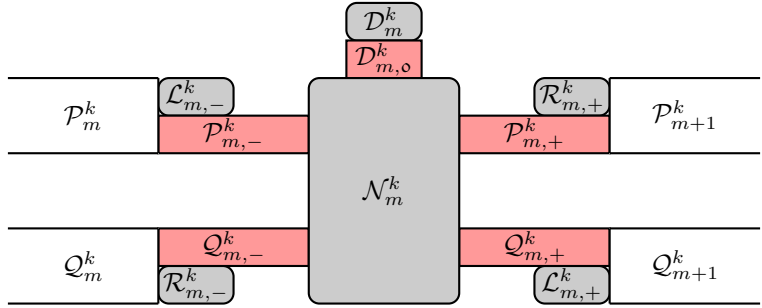

Fix $m\in[2,\,L-2]$. Recall from Section \ref{sec7.2.1} that $\mathcal{N}_{m}^{k}$
consists of the configurations reachable from $\mathcal{S}_{m}^{k}$
by particle jumps in $\mathfrak{c}^{k}\cup\mathfrak{c}^{k+\mathscr{N}_{0}}$
within energy $\mathbb{H}_{0}+3$. We define a list of collections
in Table \ref{Tab2} which are illustrated in Figure \ref{Fig7.9}.

\begin{table}
\begin{tabular}{|c|l|}
\hline 
{\footnotesize{}$\mathcal{D}_{m,\mathfrak{0}}^{k}$} & {\footnotesize{}Reachable from $\mathcal{S}_{m}^{k}$ by allowed jumps
in $\mathfrak{c}^{k}\cup\mathfrak{c}^{k+\mathscr{N}_{0}}$, of energy
$\mathbb{H}_{0}+4$}\tabularnewline
\hline 
{\footnotesize{}$\mathcal{D}_{m}^{k}$} & {\footnotesize{}Reachable from $\mathcal{S}_{m}^{k}$ by allowed jumps
in $\mathfrak{c}^{k}\cup\mathfrak{c}^{k+\mathscr{N}_{0}}$, of energy
$\mathbb{H}_{0}+3$, excluding $\mathcal{N}_{m}^{k}$}\tabularnewline
\hline 
{\footnotesize{}$\mathcal{R}_{m,+}^{k}$} & {\footnotesize{}$\mathfrak{c}^{[k+1,\,k+\mathscr{N}_{0}-2]}\cup W\cup W'\cup W''$
where $W\in\mathfrak{S}_{L-m}^{k}$, $W'\in\mathfrak{S}_{L-1}^{k+\mathscr{N}_{0}-1}$,
$W''\in\mathfrak{S}_{m+1}^{k+\mathscr{N}_{0}}$ s.t. $W'\vdash W''$}\tabularnewline
\hline 
{\footnotesize{}$\mathcal{P}_{m,+}^{k}$} & {\footnotesize{}Reachable from $\mathcal{R}_{m,+}^{k}$ by allowed
jumps in $\mathfrak{c}^{k}\cup\mathfrak{c}^{[k+\mathscr{N}_{0}-1,\,k+\mathscr{N}_{0}]}$,
excluding $\mathcal{N}_{m}^{k}\cup\mathcal{R}_{m,+}^{k}$}\tabularnewline
\hline 
{\footnotesize{}$\mathcal{R}_{m,-}^{k}$} & {\footnotesize{}$\mathfrak{c}^{[k+1,\,k+\mathscr{N}_{0}-1]}\cup W\cup W'\cup W''$
where $W\in\mathfrak{S}_{L-m}^{k}$, $W'\in\mathfrak{S}_{m-1}^{k+\mathscr{N}_{0}}$,
$W''\in\mathfrak{S}_{1}^{k+\mathscr{N}_{0}+1}$ s.t. $W'\vdash W''$}\tabularnewline
\hline 
{\footnotesize{}$\mathcal{Q}_{m,-}^{k}$} & {\footnotesize{}Reachable from $\mathcal{R}_{m,-}^{k}$ by allowed
jumps in $\mathfrak{c}^{k}\cup\mathfrak{c}^{[k+\mathscr{N}_{0},\,k+\mathscr{N}_{0}+1]}$,
excluding $\mathcal{N}_{m}^{k}\cup\mathcal{R}_{m,-}^{k}$}\tabularnewline
\hline 
{\footnotesize{}$\mathcal{L}_{m,-}^{k}$} & {\footnotesize{}$\mathfrak{c}^{[k+2,\,k+\mathscr{N}_{0}-1]}\cup W\cup W'\cup W''$
where $W\in\mathfrak{S}_{L-m+1}^{k}$, $W'\in\mathfrak{S}_{L-1}^{k+1}$,
$W''\in\mathfrak{S}_{m}^{k+\mathscr{N}_{0}}$ s.t. $W'\vdash W$}\tabularnewline
\hline 
{\footnotesize{}$\mathcal{P}_{m,-}^{k}$} & {\footnotesize{}Reachable from $\mathcal{L}_{m,-}^{k}$ by allowed
jumps in $\mathfrak{c}^{[k,\,k+1]}\cup\mathfrak{c}^{k+\mathscr{N}_{0}}$,
excluding $\mathcal{N}_{m}^{k}\cup\mathcal{L}_{m,-}^{k}$}\tabularnewline
\hline 
{\footnotesize{}$\mathcal{L}_{m,+}^{k}$} & {\footnotesize{}$\mathfrak{c}^{[k+1,\,k+\mathscr{N}_{0}-1]}\cup W\cup W'\cup W''$
where $W\in\mathfrak{S}_{1}^{k-1}$, $W'\in\mathfrak{S}_{L-m-1}^{k}$,
$W''\in\mathfrak{S}_{m}^{k+\mathscr{N}_{0}}$ s.t. $W'\vdash W$}\tabularnewline
\hline 
{\footnotesize{}$\mathcal{Q}_{m,+}^{k}$} & {\footnotesize{}Reachable from $\mathcal{L}_{m,+}^{k}$ by allowed
jumps in $\mathfrak{c}^{[k-1,\,k]}\cup\mathfrak{c}^{k+\mathscr{N}_{0}}$,
excluding $\mathcal{N}_{m}^{k}\cup\mathcal{L}_{m,+}^{k}$}\tabularnewline
\hline 
\end{tabular}

\vspace{2mm}
\caption{\label{Tab2}Collections of configurations in $\widehat{\mathcal{N}}_{m}^{k}$
for $m\in[2,\,L-2]$.}
\end{table}

First, suppose that a configuration $\xi$ is obtained from $\mathcal{S}_{m}^{k}$
by vertical particle jumps in $\mathfrak{c}^{k}\cup\mathfrak{c}^{k+\mathscr{N}_{0}}$.
Then, it is easy to see that 
\begin{equation}
\mathbb{H}(\xi)\ge\mathbb{H}_{0}+2\quad\text{and equality holds if and only if}\quad\xi\in\mathcal{S}_{m}^{k}.\label{eq:vertical jumps claim}
\end{equation}
As done before, we start from $\mathcal{N}_{m}^{k}$ and follow an
allowed path that avoids $\mathcal{W}$. If we escape $\mathcal{N}_{m}^{k}$
from vertical movements only and obtain $\eta$, then by definition
$\eta\in\mathcal{D}_{m,\mathfrak{0}}^{k}\cup\mathcal{D}_{m}^{k}$.
If $\eta\in\mathcal{D}_{m,\mathfrak{0}}^{k}$ such that $\mathbb{H}(\eta)=\mathbb{H}_{0}+4$,
then any horizontal movements are not allowed since they raise the
energy. On the other hand, if $\eta\in\mathcal{D}_{m}^{k}$ such that
$\mathbb{H}(\eta)=\mathbb{H}_{0}+3$, then the fact that $\eta\notin\mathcal{N}_{m}^{k}$
implies that any vertical jumps from $\eta$ must raise the energy.\footnote{If not, there exists either an isolated empty site or an isolated
occupied site in $\mathfrak{c}^{k}\cup\mathfrak{c}^{k+\mathscr{N}_{0}}$.
This isolated site can move vertically until it hits the same spin
to arrive at $\mathcal{S}_{m}^{k}$, which implies that $\eta\in\mathcal{N}_{m}^{k}$
contradicting the assumption.} This implies that all spins $0$ and $1$ in $\mathfrak{c}^{k}\cup\mathfrak{c}^{k+\mathscr{N}_{0}}$
have at least one neighbor in the same column with the same spin.
In this situation, any horizontal movements from $\eta$ raise the
energy by at least $2$, thereby attain energy at least $\mathbb{H}_{0}+5$,
thus they are not permitted. In summary, after entering $\mathcal{D}_{m,\mathfrak{0}}^{k}\cup\mathcal{D}_{m}^{k}$,
there are no other new configurations to obtain and the only option
is to return to $\mathcal{N}_{m}^{k}$.

Next, suppose that we escape $\mathcal{N}_{m}^{k}$ from a horizontal
movement. There are exactly four cases: $\mathfrak{c}^{k+\mathscr{N}_{0}-1}\to\mathfrak{c}^{k+\mathscr{N}_{0}}$,
$\mathfrak{c}^{k+\mathscr{N}_{0}}\to\mathfrak{c}^{k+\mathscr{N}_{0}+1}$,
$\mathfrak{c}^{k+1}\to\mathfrak{c}^{k}$ and $\mathfrak{c}^{k}\to\mathfrak{c}^{k-1}$.\smallskip{}

\begin{itemize}
\item If a $\mathfrak{c}^{k+\mathscr{N}_{0}-1}\to\mathfrak{c}^{k+\mathscr{N}_{0}}$-jump
occurs, then it is straightforward to see that we enter $\mathcal{P}_{m,+}^{k}$.
Then, except returning to $\mathcal{N}_{m}^{k}$, the empty site in
$\mathfrak{c}^{k+\mathscr{N}_{0}-1}$ may move vertically to obtain
a configuration in $\mathcal{R}_{m,+}^{k}$ or it may move left and
enter the sea of occupied sites and we visit $\mathcal{P}_{m+1}^{k}$
which must be avoided. From $\mathcal{R}_{m,+}^{k}$, new configurations
are obtained by vertical jumps in $\mathfrak{c}^{k}$, again producing
configurations in $\mathcal{P}_{m,+}^{k}$. From here, the only option
is to return to $\mathcal{R}_{m,+}^{k}$.\smallskip{}
\item The other three cases yield similar results by entering $\mathcal{Q}_{m,-}^{k}$,
$\mathcal{P}_{m,-}^{k}$, or $\mathcal{Q}_{m,+}^{k}$. The only differences
are that the escape sets are now $\mathcal{R}_{m,-}^{k}\cup\mathcal{Q}_{m}^{k}$,
$\mathcal{L}_{m,-}^{k}\cup\mathcal{P}_{m}^{k}$, or $\mathcal{L}_{m,+}^{k}\cup\mathcal{Q}_{m+1}^{k}$,
respectively.\smallskip{}
\end{itemize}
Therefore, we exhausted all the possible configurations in $\widehat{\mathcal{N}}_{m}^{k}\setminus\mathcal{N}_{m}^{k}$,
thus the following decomposition is valid:
\begin{equation}
\widehat{\mathcal{N}}_{m}^{k}\setminus\mathcal{N}_{m}^{k}=\mathcal{D}_{m,\mathfrak{0}}^{k}\cup\mathcal{D}_{m}^{k}\cup\mathcal{P}_{m,+}^{k}\cup\mathcal{R}_{m,+}^{k}\cup\mathcal{Q}_{m,-}^{k}\cup\mathcal{R}_{m,-}^{k}\cup\mathcal{P}_{m,-}^{k}\cup\mathcal{L}_{m,-}^{k}\cup\mathcal{Q}_{m,+}^{k}\cup\mathcal{L}_{m,+}^{k}.\label{eq:Nkm-hat dec}
\end{equation}
Refer to Figure \ref{Fig7.10} for the overall structure of $\widehat{\mathcal{N}}_{m}^{k}$.
In summary, we obtain the following lemma.
\begin{lem}
\label{lem:Nkm-hat stab plat}In $\widehat{\mathcal{N}}_{m}^{k}\setminus\mathcal{N}_{m}^{k}$,
the stable plateaux are the ones in $\mathcal{D}_{m}^{k}\cup\mathcal{R}_{m,+}^{k}\cup\mathcal{R}_{m,-}^{k}\cup\mathcal{L}_{m,-}^{k}\cup\mathcal{L}_{m,+}^{k}$.
Moreover, all the stable plateaux here have energy $\mathbb{H}_{0}+3$.
\end{lem}

\subsection{\label{sec7.5}Edge part}

Finally, we investigate the collections $\widehat{\mathcal{N}}^{k}$
that appear in the right-hand side of \eqref{eq:E landscape dec}.
The structure of $\widehat{\mathcal{N}}^{k}$ is far more complicated
than the other collections studied before, due to the fact that if
we start from the ground state $\bm{\sigma^{k}}$ (energy $\mathbb{H}_{0}$),
we have much more room for the energy to fluctuate compared to when
we start from $\mathcal{S}_{m}^{k}$ for $m\in[2,\,L-2]$ (energy
$\mathbb{H}_{0}+2$).

\subsubsection{\label{sec7.5.1}Small valleys}

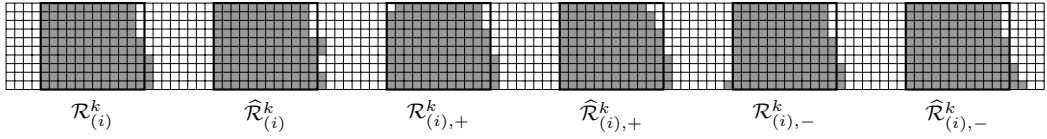
\begin{figure}
\begin{tikzpicture}[scale=0.115]
\fill[white] (-4,0) rectangle (116,10);

\fill[black!40!white] (0,0) rectangle (11,10);
\fill[black!40!white] (11,0) rectangle (12,6);
\fill[black!40!white] (12,0) rectangle (13,4);
\draw (6,0) node[below]{\scriptsize $\mathcal{R}_{(i)}^k$};

\fill[black!40!white] (20,0) rectangle (31,10);
\fill[black!40!white] (31,0) rectangle (32,6);
\fill[black!40!white] (32,0) rectangle (33,2);
\fill[black!40!white] (32,4) rectangle (33,6);
\draw (26,0) node[below]{\scriptsize $\widehat{\mathcal{R}}_{(i)}^k$};

\fill[black!40!white] (40,0) rectangle (51,10);
\fill[black!40!white] (51,0) rectangle (52,7);
\fill[black!40!white] (52,0) rectangle (53,4);
\fill[white] (40,9) rectangle (41,10);
\draw (46,0) node[below]{\scriptsize $\mathcal{R}_{(i),+}^k$};

\fill[black!40!white] (60,0) rectangle (70,10);
\fill[black!40!white] (70,0) rectangle (71,9);
\fill[black!40!white] (71,0) rectangle (72,7);
\fill[black!40!white] (72,0) rectangle (73,4);
\draw (66,0) node[below]{\scriptsize $\widehat{\mathcal{R}}_{(i),+}^k$};

\fill[black!40!white] (80,0) rectangle (91,10);
\fill[black!40!white] (79,0) rectangle (80,1);
\fill[black!40!white] (91,0) rectangle (92,6);
\fill[black!40!white] (92,0) rectangle (93,3);
\draw (86,0) node[below]{\scriptsize $\mathcal{R}_{(i),-}^k$};

\fill[black!40!white] (100,0) rectangle (111,10);
\fill[black!40!white] (111,0) rectangle (112,6);
\fill[black!40!white] (112,0) rectangle (113,3);
\fill[black!40!white] (113,0) rectangle (114,1);
\draw (106,0) node[below]{\scriptsize $\widehat{\mathcal{R}}_{(i),-}^k$};

\draw[very thin] (-4,0) grid (116,10);
\foreach \i in {0,...,5} {
\draw[thick] (0+20*\i,0) rectangle (12+20*\i,10);
}
\end{tikzpicture}\caption{\label{Fig7.11}Valley configurations in $\widehat{\mathcal{N}}^{k}$
in the case of $i=4$.}
\end{figure}

First, we investigate small valleys that are contained in $\widehat{\mathcal{N}}^{k}$.
For $i\ge2$, we define collections as in Table \ref{Tab3} (cf. Figure
\ref{Fig7.11}).\footnote{We say that two subsets $A,\,B$ of $V$ are \emph{separated} if the
graph distance between $A$ and $B$ are at least two.}

\begin{table}
\begin{tabular}{|c|l|}
\hline 
{\footnotesize{}$\mathcal{R}_{(i)}^{k}$} & {\footnotesize{}$\mathfrak{c}^{[k,\,k+\mathscr{N}_{0}-2]}\cup W\cup W'$
where $W\in\mathfrak{S}_{L-i}^{k+\mathscr{N}_{0}-1}$ and $W'\in\mathfrak{S}_{i}^{k+\mathscr{N}_{0}}$
s.t. $W\vdash W'$}\tabularnewline
\hline 
\multirow{3}{*}{{\footnotesize{}$\widehat{\mathcal{R}}_{(i)}^{k}$}} & {\footnotesize{}$\mathfrak{c}^{[k,\,k+\mathscr{N}_{0}-2]}\cup W\cup W'\cup W''$
with $|W|,\,|W'|,\,|W''|\ge2$ where $W,\,W'$ are separated sticks
in $\mathfrak{c}^{k+\mathscr{N}_{0}-1}$}\tabularnewline
 & {\footnotesize{}and $W''$ is a stick in $\mathfrak{c}^{k+\mathscr{N}_{0}}$
s.t. $|W|+|W'|=L-i$, $|W''|=i$ and $W\vdash W''$, or $W$ is a
stick in $\mathfrak{c}^{k+\mathscr{N}_{0}-1}$}\tabularnewline
 & {\footnotesize{}and $W',\,W''$ are separated sticks in $\mathfrak{c}^{k+\mathscr{N}_{0}}$
s.t. $|W|=L-i$, $|W'|+|W''|=i$ and $W\vdash W',\,W''$}\tabularnewline
\hline 
{\footnotesize{}$\mathcal{R}_{(i),+}^{k}$} & {\footnotesize{}$\mathfrak{c}^{[k,\,k+\mathscr{N}_{0}-2]}\cup W\cup W'\setminus\{w\}$
where $W\in\mathfrak{S}_{L-i+1}^{k+\mathscr{N}_{0}-1}$, $W'\in\mathfrak{S}_{i}^{k+\mathscr{N}_{0}}$,
$w\in\mathfrak{c}^{k}$ s.t. $W\vdash W'$}\tabularnewline
\hline 
{\footnotesize{}$\widehat{\mathcal{R}}_{(i),+}^{k}$} & {\footnotesize{}$\mathfrak{c}^{[k,\,k+\mathscr{N}_{0}-3]}\cup W\cup W'\cup W''$
where $W\in\mathfrak{S}_{L-1}^{k+\mathscr{N}_{0}-2}$, $W'\in\mathfrak{S}_{L-i+1}^{k+\mathscr{N}_{0}-1}$,
$W''\in\mathfrak{S}_{i}^{k+\mathscr{N}_{0}}$ s.t. $W\vdash W'\vdash W''$}\tabularnewline
\hline 
{\footnotesize{}$\mathcal{R}_{(i),-}^{k}$} & {\footnotesize{}$\mathfrak{c}^{[k,\,k+\mathscr{N}_{0}-2]}\cup W\cup W'\cup\{w\}$
where $W\in\mathfrak{S}_{L-i}^{k+\mathscr{N}_{0}-1}$, $W'\in\mathfrak{S}_{i-1}^{k+\mathscr{N}_{0}}$,
$w\in\mathfrak{c}^{k-1}$ s.t. $W\vdash W'$}\tabularnewline
\hline 
{\footnotesize{}$\widehat{\mathcal{R}}_{(i),-}^{k}$} & {\footnotesize{}$\mathfrak{c}^{[k,\,k+\mathscr{N}_{0}-2]}\cup W\cup W'\cup W''$
where $W\in\mathfrak{S}_{L-i}^{k+\mathscr{N}_{0}-1}$, $W'\in\mathfrak{S}_{i-1}^{k+\mathscr{N}_{0}}$,
$W''\in\mathfrak{S}_{1}^{k+\mathscr{N}_{0}+1}$ s.t. $W\vdash W'\vdash W''$}\tabularnewline
\hline 
\end{tabular}

\vspace{2mm}
\caption{\label{Tab3}Configurations in the small valleys in $\widehat{\mathcal{N}}^{k}$.
Denote by $\mathcal{L}_{(i)}^{k}$, $\widehat{\mathcal{L}}_{(i)}^{k}$,
$\mathcal{L}_{(i),-}^{k}$, $\widehat{\mathcal{L}}_{(i),-}^{k}$,
$\mathcal{L}_{(i),+}^{k}$ and $\widehat{\mathcal{L}}_{(i),+}^{k}$
the collections obtained from the six collections defined above, respectively,
by horizontally reflecting the configurations, where the notation
switches as $\mathcal{R}\leftrightarrow\mathcal{L}$ and $+\leftrightarrow-$.}
\end{table}

We remark that $\mathcal{R}_{(i)}^{k}$, $\mathcal{L}_{(i)}^{k}$
are defined for $i\le\frac{L}{2}$, $\widehat{\mathcal{R}}_{(i)}^{k},\,\widehat{\mathcal{L}}_{(i)}^{k}$
are defined for $i\le\frac{L-2}{2}$ and $\mathcal{R}_{(i),\pm}^{k},\,\widehat{\mathcal{R}}_{(i),\pm}^{k},\,\mathcal{L}_{(i),\mp}^{k},\,\widehat{\mathcal{L}}_{(i),\mp}^{k}$
are defined for $i\le\frac{L+1}{2}$. 

First, we focus on the six collections in the table. It is straightforward
that $\mathbb{H}(\eta)=\mathbb{H}_{0}+2$ for $\eta\in\mathcal{R}_{(i)}^{k}$
and $\mathbb{H}(\eta)=\mathbb{H}_{0}+3$ for $\eta\in\widehat{\mathcal{R}}_{(i)}^{k}\cup\mathcal{R}_{(i),+}^{k}\cup\widehat{\mathcal{R}}_{(i),+}^{k}\cup\mathcal{R}_{(i),-}^{k}\cup\widehat{\mathcal{R}}_{(i),-}^{k}$.
It is immediate from here that $\mathcal{N}(\widehat{\mathcal{R}}_{(i)}^{k})=\widehat{\mathcal{R}}_{(i)}^{k}$,
$\mathcal{N}(\mathcal{R}_{(i),+}^{k})=\mathcal{R}_{(i),+}^{k}$, $\mathcal{N}(\widehat{\mathcal{R}}_{(i),+}^{k})=\widehat{\mathcal{R}}_{(i),+}^{k}$,
$\mathcal{N}(\mathcal{R}_{(i),-}^{k})=\mathcal{R}_{(i),-}^{k}$ and
$\mathcal{N}(\widehat{\mathcal{R}}_{(i),-}^{k})=\widehat{\mathcal{R}}_{(i),-}^{k}$.

In turn, we investigate $\mathcal{N}(\mathcal{R}_{(i)}^{k})$. First,
assume that $i<\frac{L}{2}$. By the vertical mechanism explained
in Section \ref{sec7.2.1}, the stick $W'$ in $\mathfrak{c}^{k+\mathscr{N}_{0}}$
can slide vertically on the stick $W$ in $\mathfrak{c}^{k+\mathscr{N}_{0}-1}$
within energy $\mathbb{H}_{0}+3$. Moreover, switching the roles of
spins $0$ and $1$ in $\mathfrak{c}^{[k+\mathscr{N}_{0}-1,\,k+\mathscr{N}_{0}]}$,
the empty sites in $\mathfrak{c}^{k+\mathscr{N}_{0}-1}$ form a stick
and sits on the empty stick in $\mathfrak{c}^{k+\mathscr{N}_{0}}$,
such that it can also slide vertically. Thus, all the configurations
in $\mathcal{R}_{(i)}^{k}$ are connected within energy $\mathbb{H}_{0}+3$,
such that $\mathcal{N}(\mathcal{R}_{(i)}^{k})$ is a single connected
set.

On the other hand, if $i=\frac{L}{2}$ (in particular when $L$ is
even), then the two sticks in $\mathfrak{c}^{k+\mathscr{N}_{0}-1}$
and $\mathfrak{c}^{k+\mathscr{N}_{0}}$ have the same size; thus,
the sticks cannot move without visiting energy level $\mathbb{H}_{0}+4$.
This implies that $\mathcal{N}(\mathcal{R}_{(\frac{L}{2})}^{k})=\mathcal{R}_{(\frac{L}{2})}^{k}$,
where each element is isolated.

The same results hold for the reflected collections: $\mathcal{L}_{(i)}^{k}$,
$\widehat{\mathcal{L}}_{(i)}^{k}$, $\mathcal{L}_{(i),-}^{k}$, $\widehat{\mathcal{L}}_{(i),-}^{k}$,
$\mathcal{L}_{(i),+}^{k}$ and $\widehat{\mathcal{L}}_{(i),+}^{k}$.

\subsubsection{\label{sec7.5.2}Local geometry near $\mathcal{N}(\mathcal{R}_{(i)}^{k})$
and $\mathcal{N}(\mathcal{L}_{(i)}^{k})$ for $i\in[3,\,\frac{L}{2}]$}

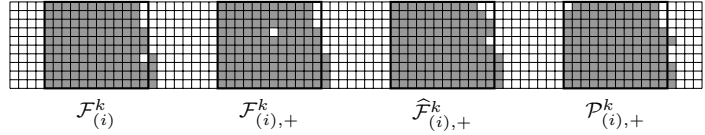
\begin{figure}
\begin{tikzpicture}[scale=0.115]
\fill[white] (-4,0) rectangle (76,10);

\fill[black!40!white] (0,0) rectangle (11,10);
\fill[black!40!white] (11,0) rectangle (12,3);
\fill[black!40!white] (11,4) rectangle (12,7);
\fill[black!40!white] (12,0) rectangle (13,4);
\draw (6,0) node[below]{\scriptsize $\mathcal{F}_{(i)}^k$};

\fill[black!40!white] (20,0) rectangle (31,10);
\fill[black!40!white] (31,0) rectangle (32,7);
\fill[black!40!white] (32,0) rectangle (33,4);
\fill[white] (26,6) rectangle (27,7);
\draw (26,0) node[below]{\scriptsize $\mathcal{F}_{(i),+}^k$};

\fill[black!40!white] (40,0) rectangle (50,10);
\fill[black!40!white] (50,0) rectangle (51,9);
\fill[black!40!white] (51,0) rectangle (52,5);
\fill[black!40!white] (51,6) rectangle (52,8);
\fill[black!40!white] (52,0) rectangle (53,4);
\draw (46,0) node[below]{\scriptsize $\widehat{\mathcal{F}}_{(i),+}^k$};

\fill[black!40!white] (60,0) rectangle (71,10);
\fill[black!40!white] (71,0) rectangle (72,7);
\fill[black!40!white] (72,0) rectangle (73,3);
\fill[black!40!white] (72,5) rectangle (73,6);
\fill[white] (60,9) rectangle (61,10);
\draw (66,0) node[below]{\scriptsize $\mathcal{P}_{(i),+}^k$};

\draw[very thin] (-4,0) grid (76,10);
\foreach \i in {0,...,3} {
\draw[thick] (0+20*\i,0) rectangle (12+20*\i,10);
}
\end{tikzpicture}\caption{\label{Fig7.12}Local configurations near $\mathcal{N}(\mathcal{R}_{(i)}^{k})$
for $i=4$.}
\end{figure}

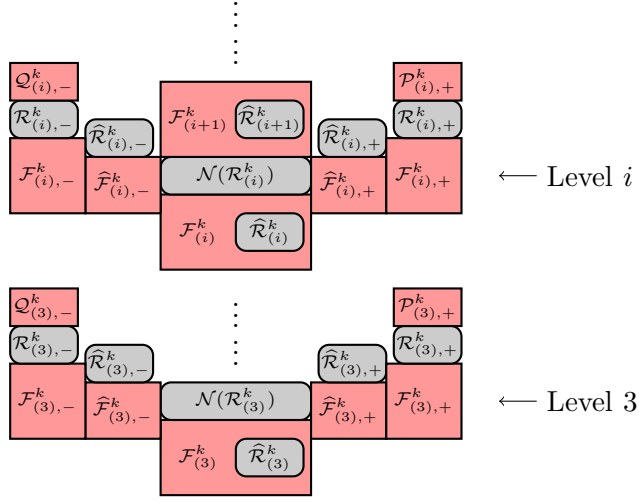
\begin{figure}
\begin{tikzpicture}
\fill[red!40!white] (1,3) rectangle (-1,4); \draw[thick] (1,3) rectangle (-1,4);
\fill[black!20!white,rounded corners] (1,4) rectangle (-1,4.5); \draw[thick,rounded corners] (1,4) rectangle (-1,4.5);
\fill[red!40!white] (1,3.75) rectangle (2,4.5); \draw[thick] (1,3.75) rectangle (2,4.5); \fill[black!20!white,rounded corners] (1.1,4.5) rectangle (2,5); \draw[thick,rounded corners] (1.1,4.5) rectangle (2,5);
\fill[red!40!white] (-1,3.75) rectangle (-2,4.5); \draw[thick] (-1,3.75) rectangle (-2,4.5); \fill[black!20!white,rounded corners] (-1.1,4.5) rectangle (-2,5); \draw[thick,rounded corners] (-1.1,4.5) rectangle (-2,5);
\fill[red!40!white] (2,3.75) rectangle (3,4.75); \draw[thick] (2,3.75) rectangle (3,4.75); \fill[red!40!white] (-2,3.75) rectangle (-3,4.75); \draw[thick] (-2,3.75) rectangle (-3,4.75);
\fill[black!20!white,rounded corners] (2.1,4.75) rectangle (3,5.25); \draw[thick,rounded corners] (2.1,4.75) rectangle (3,5.25); \fill[black!20!white,rounded corners] (-2.1,4.75) rectangle (-3,5.25); \draw[thick,rounded corners] (-2.1,4.75) rectangle (-3,5.25);
\fill[red!40!white] (2.1,5.25) rectangle (3,5.75); \draw[thick] (2.1,5.25) rectangle (3,5.75); \fill[red!40!white] (-2.1,5.25) rectangle (-3,5.75); \draw[thick] (-2.1,5.25) rectangle (-3,5.75);
\fill[black!20!white,rounded corners] (0,3.25) rectangle (0.9,3.75); \draw[thick,rounded corners] (0,3.25) rectangle (0.9,3.75);
\draw (2.5,4.25) node{\tiny $\mathcal{F}_{(3),+}^k$}; \draw (-2.5,4.25) node{\tiny $\mathcal{F}_{(3),-}^k$};
\draw (1.5,4.125) node{\tiny $\widehat{\mathcal{F}}_{(3),+}^k$}; \draw (-1.5,4.125) node{\tiny $\widehat{\mathcal{F}}_{(3),-}^k$};
\draw (-0.5,3.5) node{\tiny $\mathcal{F}_{(3)}^k$};
\draw (0.45,3.5) node{\tiny $\widehat{\mathcal{R}}_{(3)}^k$};
\draw (0,4.25) node{\tiny $\mathcal{N}(\mathcal{R}_{(3)}^k)$};
\draw (1.55,4.75) node{\tiny $\widehat{\mathcal{R}}_{(3),+}^k$};
\draw (-1.55,4.75) node{\tiny $\widehat{\mathcal{R}}_{(3),-}^k$};
\draw[->] (4,4.25)--(3.5,4.25); \draw (4,4.25) node[right]{Level $3$};
\draw (2.55,5) node{\tiny $\mathcal{R}_{(3),+}^k$};
\draw (-2.55,5) node{\tiny $\mathcal{R}_{(3),-}^k$};
\draw (2.55,5.5) node{\tiny $\mathcal{P}_{(3),+}^k$};
\draw (-2.55,5.5) node{\tiny $\mathcal{Q}_{(3),-}^k$};

\draw (0,5) node{$\vdots$};
\draw (0,5.5) node{$\vdots$};

\fill[red!40!white] (1,6) rectangle (-1,7); \draw[thick] (1,6) rectangle (-1,7);
\fill[black!20!white,rounded corners] (1,7) rectangle (-1,7.5); \draw[thick,rounded corners] (1,7) rectangle (-1,7.5);
\fill[red!40!white] (1,6.75) rectangle (2,7.5); \draw[thick] (1,6.75) rectangle (2,7.5); \fill[black!20!white,rounded corners] (1.1,7.5) rectangle (2,8); \draw[thick,rounded corners] (1.1,7.5) rectangle (2,8);
\fill[red!40!white] (-1,6.75) rectangle (-2,7.5); \draw[thick] (-1,6.75) rectangle (-2,7.5); \fill[black!20!white,rounded corners] (-1.1,7.5) rectangle (-2,8); \draw[thick,rounded corners] (-1.1,7.5) rectangle (-2,8);
\fill[red!40!white] (2,6.75) rectangle (3,7.75); \draw[thick] (2,6.75) rectangle (3,7.75); \fill[red!40!white] (-2,6.75) rectangle (-3,7.75); \draw[thick] (-2,6.75) rectangle (-3,7.75);
\fill[black!20!white,rounded corners] (2.1,7.75) rectangle (3,8.25); \draw[thick,rounded corners] (2.1,7.75) rectangle (3,8.25); \fill[black!20!white,rounded corners] (-2.1,7.75) rectangle (-3,8.25); \draw[thick,rounded corners] (-2.1,7.75) rectangle (-3,8.25);
\fill[red!40!white] (2.1,8.25) rectangle (3,8.75); \draw[thick] (2.1,8.25) rectangle (3,8.75); \fill[red!40!white] (-2.1,8.25) rectangle (-3,8.75); \draw[thick] (-2.1,8.25) rectangle (-3,8.75);
\fill[black!20!white,rounded corners] (0,6.25) rectangle (0.9,6.75); \draw[thick,rounded corners] (0,6.25) rectangle (0.9,6.75);
\draw (2.5,7.25) node{\tiny $\mathcal{F}_{(i),+}^k$}; \draw (-2.5,7.25) node{\tiny $\mathcal{F}_{(i),-}^k$};
\draw (1.5,7.125) node{\tiny $\widehat{\mathcal{F}}_{(i),+}^k$}; \draw (-1.5,7.125) node{\tiny $\widehat{\mathcal{F}}_{(i),-}^k$};
\draw (-0.5,6.5) node{\tiny $\mathcal{F}_{(i)}^k$};
\draw (0.45,6.5) node{\tiny $\widehat{\mathcal{R}}_{(i)}^k$};
\draw (0,7.25) node{\tiny $\mathcal{N}(\mathcal{R}_{(i)}^k)$};
\draw (1.55,7.75) node{\tiny $\widehat{\mathcal{R}}_{(i),+}^k$};
\draw (-1.55,7.75) node{\tiny $\widehat{\mathcal{R}}_{(i),-}^k$};
\draw[->] (4,7.25)--(3.5,7.25); \draw (4,7.25) node[right]{Level $i$};
\draw (2.55,8) node{\tiny $\mathcal{R}_{(i),+}^k$};
\draw (-2.55,8) node{\tiny $\mathcal{R}_{(i),-}^k$};
\draw (2.55,8.5) node{\tiny $\mathcal{P}_{(i),+}^k$};
\draw (-2.55,8.5) node{\tiny $\mathcal{Q}_{(i),-}^k$};

\fill[red!40!white] (1,7.5) rectangle (-1,8.5); \draw[thick] (1,7.5) rectangle (-1,8.5);
\fill[black!20!white,rounded corners] (0,7.75) rectangle (0.9,8.25); \draw[thick,rounded corners] (0,7.75) rectangle (0.9,8.25);
\draw (-0.5,8) node{\tiny $\mathcal{F}_{(i+1)}^k$};
\draw (0.45,8) node{\tiny $\widehat{\mathcal{R}}_{(i+1)}^k$};

\draw (0,9) node{$\vdots$};
\draw (0,9.5) node{$\vdots$};
\end{tikzpicture}\caption{\label{Fig7.13}Local structure near $\mathcal{N}(\mathcal{R}_{(i)}^{k})$
for $i\in[3,\,\frac{L}{2}]$.}
\end{figure}

Next, we look into the local geometry near $\mathcal{N}(\mathcal{R}_{(i)}^{k})$
and $\mathcal{N}(\mathcal{L}_{(i)}^{k})$ for $i\in[3,\,\frac{L}{2}]$.
Define the following collections in Table \ref{Tab4} for $i\ge2$.
See Figure \ref{Fig7.12} for a few examples.

\begin{table}
\begin{tabular}{|c|l|}
\hline 
{\footnotesize{}$\mathcal{F}_{(i)}^{k}$} & {\footnotesize{}$\mathfrak{c}^{[k,\,k+\mathscr{N}_{0}-2]}\cup W\cup W'$
where $W\subseteq\mathfrak{c}^{k+\mathscr{N}_{0}-1}$, $W'\subseteq\mathfrak{c}^{k+\mathscr{N}_{0}}$
s.t. $|W|=L-i$, $|W'|=i$, of energy $\mathbb{H}_{0}+4$}\tabularnewline
\hline 
{\footnotesize{}$\mathcal{F}_{(i),+}^{k}$} & {\footnotesize{}$\mathfrak{c}^{[k,\,k+\mathscr{N}_{0}-2]}\cup W\cup W'\setminus\{w\}$
where $W\in\mathfrak{S}_{L+1-i}^{k+\mathscr{N}_{0}-1}$, $W'\in\mathfrak{S}_{i}^{k+\mathscr{N}_{0}}$,
$w\in\mathfrak{c}^{[k+1,\,k+\mathscr{N}_{0}-3]}$ s.t. $W\vdash W'$}\tabularnewline
\hline 
{\footnotesize{}$\widehat{\mathcal{F}}_{(i),+}^{k}$} & {\footnotesize{}Reachable from $\widehat{\mathcal{R}}_{(i),+}^{k}$
by allowed jumps in $\mathfrak{c}^{[k+\mathscr{N}_{0}-2,\,k+\mathscr{N}_{0}]}$
s.t. $\mathfrak{c}^{k+\mathscr{N}_{0}-2}$ is not full}\tabularnewline
\hline 
{\footnotesize{}$\mathcal{P}_{(i),+}^{k}$} & {\footnotesize{}Reachable from $\mathcal{R}_{(i),+}^{k}$ by allowed
jumps in $\mathfrak{c}^{k}\cup\mathfrak{c}^{[k+\mathscr{N}_{0}-1,\,k+\mathscr{N}_{0}]}$,
excluding $\mathcal{R}_{(i),+}^{k}$}\tabularnewline
\hline 
{\footnotesize{}$\mathcal{F}_{(i),-}^{k}$} & {\footnotesize{}$\mathfrak{c}^{[k,\,k+\mathscr{N}_{0}-2]}\cup W\cup W'\cup\{w\}$
where $W\in\mathfrak{S}_{L-i}^{k+\mathscr{N}_{0}-1}$, $W'\in\mathfrak{S}_{i-1}^{k+\mathscr{N}_{0}}$,
$w\in\mathfrak{c}^{[k+\mathscr{N}_{0}+2,\,k-2]}$ s.t. $W\vdash W'$}\tabularnewline
\hline 
{\footnotesize{}$\widehat{\mathcal{F}}_{(i),-}^{k}$} & {\footnotesize{}Reachable from $\widehat{\mathcal{R}}_{(i),-}^{k}$
by allowed jumps in $\mathfrak{c}^{[k+\mathscr{N}_{0}-1,\,k+\mathscr{N}_{0}+1]}$
s.t. $\mathfrak{c}^{k+\mathscr{N}_{0}+1}$ is not empty}\tabularnewline
\hline 
{\footnotesize{}$\mathcal{Q}_{(i),-}^{k}$} & {\footnotesize{}Reachable from $\mathcal{R}_{(i),-}^{k}$ by allowed
jumps in $\mathfrak{c}^{k}\cup\mathfrak{c}^{[k+\mathscr{N}_{0}-1,\,k+\mathscr{N}_{0}]}$,
excluding $\mathcal{R}_{(i),-}^{k}$}\tabularnewline
\hline 
\end{tabular}

\vspace{2mm}
\caption{\label{Tab4}Other configurations in $\widehat{\mathcal{N}}^{k}$.
Denote by $\mathcal{G}_{(i)}^{k}$, $\mathcal{G}_{(i),-}^{k}$, $\widehat{\mathcal{G}}_{(i),-}^{k}$,
$\mathcal{P}_{(i),-}^{k}$, $\mathcal{G}_{(i),+}^{k}$, $\widehat{\mathcal{G}}_{(i),+}^{k}$
and $\mathcal{Q}_{(i),+}^{k}$ the collections obtained from the above
by horizontally reflecting the configurations, where the notation
switches as $\mathcal{F}\leftrightarrow\mathcal{G}$ and $+\leftrightarrow-$.}
\end{table}

We only consider $\mathcal{N}(\mathcal{R}_{(i)}^{k})$ as for $\mathcal{N}(\mathcal{L}_{(i)}^{k})$
the situation is totally symmetric. Escaping from $\mathcal{N}(\mathcal{R}_{(i)}^{k})$,
if the particles in $\mathfrak{c}^{[k+\mathscr{N}_{0}-1,\,k+\mathscr{N}_{0}]}$
are preserved, then we enter $\mathcal{F}_{(i)}^{k}\cup\mathcal{F}_{(i+1)}^{k}$.
From $\mathcal{F}_{(i)}^{k}$, provided that the particles in $\mathfrak{c}^{[k+\mathscr{N}_{0}-1,\,k+\mathscr{N}_{0}]}$
are still preserved, we may visit valley $\widehat{\mathcal{R}}_{(i)}^{k}$
(and then return to $\mathcal{F}_{(i)}^{k}$) or visit $\mathcal{N}(\mathcal{R}_{(i-1)}^{k})\cup\mathcal{N}(\mathcal{R}_{(i)}^{k})$.
If the particles in $\mathfrak{c}^{[k+\mathscr{N}_{0}-1,\,k+\mathscr{N}_{0}]}$
are not preserved from $\mathcal{N}(\mathcal{R}_{(i)}^{k})\cup\mathcal{F}_{(i)}^{k}$,
the only options are $\widehat{\mathcal{F}}_{(i),+}^{k}$ ($\mathfrak{c}^{k+\mathscr{N}_{0}-2}\to\mathfrak{c}^{k+\mathscr{N}_{0}-1}$)
or $\widehat{\mathcal{F}}_{(i),-}^{k}$ ($\mathfrak{c}^{k+\mathscr{N}_{0}}\to\mathfrak{c}^{k+\mathscr{N}_{0}+1}$).
From $\widehat{\mathcal{F}}_{(i),+}^{k}$, we either visit valley
$\widehat{\mathcal{R}}_{(i),+}^{k}$ or proceed as $\mathcal{F}_{(i),+}^{k}\to\mathcal{R}_{(i),+}^{k}\to\mathcal{P}_{(i),+}^{k}$,
which is a dead-end. Similarly, from $\widehat{\mathcal{F}}_{(i),-}^{k}$,
we either visit valley $\widehat{\mathcal{R}}_{(i),-}^{k}$ or proceed
as $\mathcal{F}_{(i),-}^{k}\to\mathcal{R}_{(i),-}^{k}\to\mathcal{Q}_{(i),-}^{k}$
and return.

A few minor differences occur at the edge part $i=\lfloor\frac{L}{2}\rfloor$
(the greatest integer less than or equal to $\frac{L}{2}$), depending
on the parity of $L$; details are left out to the readers. See Figure
\ref{Fig7.13} for a sketch of the local geometry.

\subsubsection{\label{sec7.5.3}Local geometry near $\mathcal{N}(\mathcal{R}_{(2)}^{k})$,
$\mathcal{N}(\mathcal{L}_{(2)}^{k})$ and $\mathcal{N}^{k}$}

\begin{figure}
\begin{tikzpicture}[scale=0.115]
\fill[white] (-4,0) rectangle (16,10);

\fill[black!40!white] (0,0) rectangle (12,10);
\fill[black!40!white] (14,7) rectangle (15,8);
\fill[white] (4,4) rectangle (5,5);
\draw (6,0) node[below]{\scriptsize $\mathcal{D}^k$};

\draw[very thin] (-4,0) grid (16,10);
\foreach \i in {0} {
\draw[thick] (0+20*\i,0) rectangle (12+20*\i,10);
}
\end{tikzpicture}\caption{\label{Fig7.14}Dead-end $\mathcal{D}^{k}$.}
\end{figure}
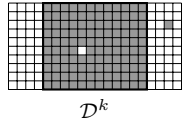

\begin{figure}
\begin{tikzpicture}
\fill[red!40!white] (3,1.5)--(-3,1.5)--(-3,3)--(-2,3)--(-2,2.75)--(-1,2.75)--(-1,2.5)--(1,2.5)--(1,2.75)--(2,2.75)--(2,3)--(3,3)--cycle;
\draw[thick] (3,1.5)--(-3,1.5)--(-3,3)--(-2,3)--(-2,2.75)--(-1,2.75)--(-1,2.5)--(1,2.5)--(1,2.75)--(2,2.75)--(2,3)--(3,3)--cycle;
\draw[thick] (2,1.5)--(2,2.75); \draw[thick] (-2,1.5)--(-2,2.75); \draw[thick] (1,1.5)--(1,2.5); \draw[thick] (-1,1.5)--(-1,2.5);
\fill[black!20!white,rounded corners] (2,3.25) rectangle (1.1,2.75); \draw[thick,rounded corners] (2,3.25) rectangle (1.1,2.75);
\fill[black!20!white,rounded corners] (-2,3.25) rectangle (-1.1,2.75); \draw[thick,rounded corners] (-2,3.25) rectangle (-1.1,2.75);
\fill[black!20!white,rounded corners] (1,2.5) rectangle (-1,3); \draw[thick,rounded corners] (1,2.5) rectangle (-1,3);
\fill[black!20!white,rounded corners] (0,1.75) rectangle (0.9,2.25); \draw[thick,rounded corners] (0,1.75) rectangle (0.9,2.25);
\draw (2.5,2.25) node{\tiny $\mathcal{F}_{(2),+}^k$}; \draw (-2.5,2.25) node{\tiny $\mathcal{F}_{(2),-}^k$};
\draw (1.5,2.125) node{\tiny $\widehat{\mathcal{F}}_{(2),+}^k$}; \draw (-1.5,2.125) node{\tiny $\widehat{\mathcal{F}}_{(2),-}^k$};
\draw (-0.5,2) node{\tiny $\mathcal{F}_{(2)}^k$};
\draw (0.45,2) node{\tiny $\widehat{\mathcal{R}}_{(2)}^k$};
\draw (0,2.75) node{\tiny $\mathcal{N}(\mathcal{R}_{(2)}^k)$};
\draw (1.55,3) node{\tiny $\widehat{\mathcal{R}}_{(2),+}^k$};
\draw (-1.55,3) node{\tiny $\widehat{\mathcal{R}}_{(2),-}^k$};

\fill[red!40!white] (1,3) rectangle (-1,4); \draw[thick] (1,3) rectangle (-1,4);
\fill[black!20!white,rounded corners] (0,3.25) rectangle (0.9,3.75); \draw[thick,rounded corners] (0,3.25) rectangle (0.9,3.75);
\draw (-0.5,3.5) node{\tiny $\mathcal{F}_{(3)}^k$};
\draw (0.45,3.5) node{\tiny $\widehat{\mathcal{R}}_{(3)}^k$};

\draw (0,4.5) node{$\vdots$};

\fill[red!40!white] (3,-1.5)--(-3,-1.5)--(-3,-3)--(-2,-3)--(-2,-2.75)--(-1,-2.75)--(-1,-2.5)--(1,-2.5)--(1,-2.75)--(2,-2.75)--(2,-3)--(3,-3)--cycle;
\draw[thick] (3,-1.5)--(-3,-1.5)--(-3,-3)--(-2,-3)--(-2,-2.75)--(-1,-2.75)--(-1,-2.5)--(1,-2.5)--(1,-2.75)--(2,-2.75)--(2,-3)--(3,-3)--cycle;
\draw[thick] (2,-1.5)--(2,-2.75); \draw[thick] (-2,-1.5)--(-2,-2.75); \draw[thick] (1,-1.5)--(1,-2.5); \draw[thick] (-1,-1.5)--(-1,-2.5);
\fill[black!20!white,rounded corners] (2,-3.25) rectangle (1.1,-2.75); \draw[thick,rounded corners] (2,-3.25) rectangle (1.1,-2.75);
\fill[black!20!white,rounded corners] (-2,-3.25) rectangle (-1.1,-2.75); \draw[thick,rounded corners] (-2,-3.25) rectangle (-1.1,-2.75);
\fill[black!20!white,rounded corners] (1,-2.5) rectangle (-1,-3); \draw[thick,rounded corners] (1,-2.5) rectangle (-1,-3);
\fill[black!20!white,rounded corners] (0,-1.75) rectangle (0.9,-2.25); \draw[thick,rounded corners] (0,-1.75) rectangle (0.9,-2.25);
\draw (2.5,-2.25) node{\tiny $\mathcal{G}_{(2),+}^k$}; \draw (-2.5,-2.25) node{\tiny $\mathcal{G}_{(2),-}^k$};
\draw (1.5,-2.125) node{\tiny $\widehat{\mathcal{G}}_{(2),+}^k$}; \draw (-1.5,-2.125) node{\tiny $\widehat{\mathcal{G}}_{(2),-}^k$};
\draw (-0.5,-2) node{\tiny $\mathcal{G}_{(2)}^k$};
\draw (0.45,-2) node{\tiny $\widehat{\mathcal{L}}_{(2)}^k$};
\draw (0,-2.75) node{\tiny $\mathcal{N}(\mathcal{L}_{(2)}^k)$};
\draw (1.55,-3) node{\tiny $\widehat{\mathcal{L}}_{(2),+}^k$};
\draw (-1.55,-3) node{\tiny $\widehat{\mathcal{L}}_{(2),-}^k$};

\fill[red!40!white] (1,-3) rectangle (-1,-4); \draw[thick] (1,-3) rectangle (-1,-4);
\fill[black!20!white,rounded corners] (0,-3.25) rectangle (0.9,-3.75); \draw[thick,rounded corners] (0,-3.25) rectangle (0.9,-3.75);
\draw (-0.5,-3.5) node{\tiny $\mathcal{G}_{(3)}^k$};
\draw (0.45,-3.5) node{\tiny $\widehat{\mathcal{L}}_{(3)}^k$};

\draw (0,-4.5) node{$\vdots$};

\foreach \i in {-2.5,2.5} { \foreach \j in {-1.25,1.25} {
\fill[black!20!white,rounded corners] (\i-0.5,\j-0.25) rectangle (\i+0.5,\j+0.25); \draw[thick,rounded corners] (\i-0.5,\j-0.25) rectangle (\i+0.5,\j+0.25); }}
\draw(2.5,1.25) node{\tiny $\mathcal{R}_{(2),+}^k$};
\draw(-2.5,-1.25) node{\tiny $\mathcal{L}_{(2),-}^k$};
\draw(2.5,-1.25) node{\tiny $\mathcal{L}_{(2),+}^k$};
\draw(-2.5,1.25) node{\tiny $\mathcal{R}_{(2),-}^k$};

\foreach \i in {-2,2} { \foreach \j in {-0.75,0.75} {
\fill[red!40!white] (\i-1,\j-0.25) rectangle (\i+1,\j+0.25); \draw[thick] (\i-1,\j-0.25) rectangle (\i+1,\j+0.25); }}
\draw (2,0.75) node{\tiny $\mathcal{P}_{(2),+}^k$};
\draw (-2,-0.75) node{\tiny $\mathcal{P}_{(2),-}^k$};
\draw (-2,0.75) node{\tiny $\mathcal{Q}_{(2),-}^k$};
\draw (2,-0.75) node{\tiny $\mathcal{Q}_{(2),+}^k$};

\foreach \i in {-4,4} { \foreach \j in {-1,1} {
\fill[white] (\i-1,\j+0.5) rectangle (\i+1,\j-0.5); \draw[thick] (\i-1,\j+0.5)--(\i+1,\j+0.5); \draw[thick] (\i-1,\j-0.5)--(\i+1,\j-0.5); }}
\draw[thick] (3,1.5)--(3,0.5); \draw[thick] (3,-1.5)--(3,-0.5); \draw[thick] (-3,1.5)--(-3,0.5); \draw[thick] (-3,-1.5)--(-3,-0.5);
\draw (4,1) node{$\small \mathcal{P}_2^k$};
\draw (4,-1) node{$\small \mathcal{Q}_2^k$};
\draw (-4,1) node{$\small \mathcal{Q}_{L-1}^{k-1}$};
\draw (-4,-1) node{$\small \mathcal{P}_{L-1}^{k-1}$};

\fill[black!20!white,rounded corners] (1,1.5) rectangle (-1,-1.5); \draw[thick,rounded corners] (1,1.5) rectangle (-1,-1.5);
\draw (0,0) node{$\small \mathcal{N}^k$};

\fill[red!40!white] (0.3,-0.6) rectangle (-0.3,-1.2); \draw[thick] (0.3,-0.6) rectangle (-0.3,-1.2);
\draw (0,-0.9) node{$\small \mathcal{D}^k$};
\end{tikzpicture}\caption{\label{Fig7.15}Local geometry near $\mathcal{N}(\mathcal{R}_{(2)}^{k})$,
$\mathcal{N}(\mathcal{L}_{(2)}^{k})$ and $\mathcal{N}^{k}$.}
\end{figure}
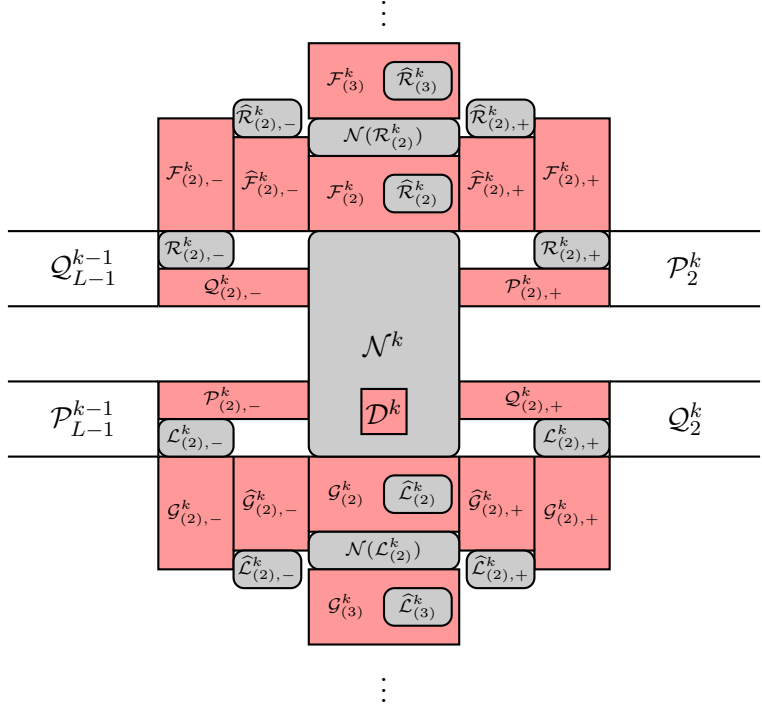

Finally, we analyze the local geometry near $\mathcal{N}(\mathcal{R}_{(2)}^{k})$,
$\mathcal{N}(\mathcal{L}_{(2)}^{k})$ and $\mathcal{N}^{k}$. One
more collection must be identified in Table \ref{Tab5} (cf. Figure
\ref{Fig7.14}).

\begin{table}
\begin{tabular}{|c|l|}
\hline 
{\footnotesize{}$\mathcal{D}^{k}$} & {\footnotesize{}$\mathfrak{c}^{[k,\,k+\mathscr{N}_{0}-1]}\cup\{w\}\setminus\{w'\}$
where $w\in\mathfrak{c}^{[k+\mathscr{N}_{0}+1,\,k-2]}$ and $w'\in\mathfrak{c}^{[k+1,\,k+\mathscr{N}_{0}-2]}$}\tabularnewline
\hline 
\end{tabular}

\vspace{2mm}
\caption{\label{Tab5}Dead-end configurations in $\widehat{\mathcal{N}}^{k}$.}
\end{table}

Starting from $\mathcal{N}(\mathcal{R}_{(2)}^{k})$, the situation
is quite similar to the bulk part identified in Section \ref{sec7.5.2};
the differences are that from $\mathcal{F}_{(2)}^{k}$ we enter $\mathcal{N}^{k}$
rather than $\mathcal{N}(\mathcal{R}_{(1)}^{k})$ (which does not
exist), that from $\mathcal{R}_{(2),+}^{k}\cup\mathcal{P}_{(2),+}^{k}$
(resp. $\mathcal{R}_{(2),-}^{k}\cup\mathcal{Q}_{(2),-}^{k}$) we may
also escape $\widehat{\mathcal{N}}^{k}$ and enter $\mathcal{P}_{2}^{k}$
(resp. $\mathcal{Q}_{L-1}^{k-1}$), and that from $\mathcal{P}_{(2),+}^{k}$
(resp. $\mathcal{Q}_{(2),-}^{k}$) we may also visit $\mathcal{N}^{k}$.
We leave the verification of these simple but tedious facts to the
readers. From $\mathcal{N}(\mathcal{L}_{(2)}^{k})$ the analysis remains
the same.

From $\mathcal{N}^{k}$, other than the above-identified collections
$\mathcal{F}_{(2)}^{k}\cup\mathcal{P}_{(2),+}^{k}\cup\mathcal{Q}_{(2),-}^{k}\cup\mathcal{G}_{(2)}^{k}\cup\mathcal{P}_{(2),-}^{k}\cup\mathcal{Q}_{(2),+}^{k}$,
we may also enter the collection $\mathcal{D}^{k}$ of dead-ends,
from which the only option is to return to $\mathcal{N}^{k}$.

Refer to Figure \ref{Fig7.15} for an illustration. In particular,
gathering the results in Sections \ref{sec7.5.1}, \ref{sec7.5.2}
and \ref{sec7.5.3} (cf. Figures \ref{Fig7.13} and \ref{Fig7.15}),
we have completed the analysis of $\widehat{\mathcal{N}}^{k}$. The
following lemma summarizes the analysis.
\begin{lem}
\label{lem:Nk-hat stab plat}In $\widehat{\mathcal{N}}^{k}\setminus\mathcal{N}^{k}$,
the stable plateaux are the ones in $\mathcal{R}_{(i)}^{k}\cup\widehat{\mathcal{R}}_{(i)}^{k}\cup\mathcal{R}_{(i),\pm}^{k}\cup\widehat{\mathcal{R}}_{(i),\pm}^{k}\cup\mathcal{L}_{(i)}^{k}\cup\widehat{\mathcal{L}}_{(i)}^{k}\cup\mathcal{L}_{(i),\mp}^{k}\cup\widehat{\mathcal{L}}_{(i),\mp}^{k}$.
Moreover, $\mathbb{H}(\eta)=\mathbb{H}_{0}+2$ for $\eta\in\mathcal{R}_{(i)}^{k}\cup\mathcal{L}_{(i)}^{k}$
and $\mathbb{H}(\eta)=\mathbb{H}_{0}+3$ otherwise.
\end{lem}

\section{\label{sec8}Proof of Theorems \ref{thm:P1 char}, \ref{thm:hier dec},
\ref{thm:Kawasaki} and \ref{thm:pos rates}}

In this final section, we conclude the proof of the main theorems
in Section \ref{sec3}.

First, gathering Lemmas \ref{lem:Nkm stab plat}, \ref{lem:Nk stab plat},
\ref{lem:Pkm Qkm stab plat}, \ref{lem:Nkm-hat stab plat} and \ref{lem:Nk-hat stab plat},
we readily check the full characterization given in Theorem \ref{thm:P1 char}.

To prove Theorem \ref{thm:hier dec}, we calculate the initial depth
(cf. \eqref{eq:Gammai1 def}) of each stable plateau. First, it is
clear for each stable plateau with energy $\mathbb{H}_{0}+3$ that
its initial depth is $1$ since $\overline{\Phi}=\mathbb{H}_{0}+4$.\footnote{One may also directly check this by inspecting the details in Sections
\ref{sec7.4} and \ref{sec7.5}.} This readily implies that
\[
\Gamma^{\star,1}=\min_{i}\Gamma_{i}^{1}=1.
\]
Now, we analyze each stable plateau.\smallskip{}

\begin{itemize}
\item According to Figure \ref{Fig7.6}, starting from each $\bm{\sigma^{k}}\in\mathcal{S}$,
the first exit is to raise energy by $3$ to enter $\mathcal{F}^{k}\cup\mathcal{G}^{k}$,
and then we may follow a downhill path to $\mathcal{R}^{k}$ or $\mathcal{L}^{k}$
which constitute of stable plateaux. This readily verifies that the
initial depth for each $\{\bm{\sigma^{k}}\}$ is exactly $3$. In
particular, by Theorem \ref{thm:P-star1 class}-(1), each $\{\bm{\sigma^{k}}\}$
is absorbing with respect to $\{\mathfrak{X}^{\star,1}(t)\}_{t\ge0}$.\smallskip{}
\item Again by Figure \ref{Fig7.6}, starting from $\mathcal{S}_{1}^{k}$,
we may raise the energy by $1$ to enter $\mathcal{P}_{1}^{k}$, move
to $\mathcal{F}^{k}$ which does not change the energy, and then finally
visit $\bm{\sigma^{k}}$ where $\mathbb{H}(\bm{\sigma^{k}})<\mathbb{H}(\mathcal{S}_{1}^{k})$.
Thus by Theorem \ref{thm:P-star1 class}-(2), $\mathcal{S}_{1}^{k}$
is a transient element of $\{\mathfrak{X}^{\star,1}(t)\}_{t\ge0}$.
Similar deductions readily imply that $\mathcal{S}_{L-1}^{k-1}$,
$\mathcal{R}^{k}$ and $\mathcal{L}^{k}$ are also transient elements
of $\{\mathfrak{X}^{\star,1}(t)\}_{t\ge0}$.\smallskip{}
\item Consider Figure \ref{Fig7.4}. Starting from each $\bm{\sigma_{m;\ell,\ell'}^{k}}\in\mathcal{S}_{m}^{k}$
for $m\in[2,\,L-2]$, we may raise energy by $1$ and enter the horizontal
or vertical blue line in Figure \ref{Fig7.4}, roam around therein
without changing the energy, and then finally end up at another $\bm{\sigma_{m;\ell'',\ell'''}^{k}}\in\mathcal{S}_{m}^{k}$.
Moreover, subject to this energy barrier, the elements of $\mathcal{S}_{m}^{k}$
are the only options to visit. This implies that each $\{\boldsymbol{\sigma_{m;\ell,\ell'}^{k}}\}$
is a recurrent element of $\{\mathfrak{X}^{\star,1}(t)\}_{t\ge0}$
and that
\[
\{\{\boldsymbol{\sigma_{m;\ell,\ell'}^{k}}\}:\,\ell,\,\ell'\in\mathbb{T}_{L}\}
\]
constitutes an irreducible component for each $k\in\mathbb{T}_{K}$
and $m\in[2,\,L-2]$.\smallskip{}
\item Referring to Figure \ref{Fig7.13}, the same classification applies
to collections $\mathcal{R}_{(i)}^{k}$ and $\mathcal{L}_{(i)}^{k}$
for $i\in[2,\,\frac{L}{2})$ as well, such that each $\{\eta\}$ for
$\eta\in\mathcal{R}_{(i)}^{k}\cup\mathcal{L}_{(i)}^{k}$ is recurrent
and
\[
\{\{\eta\}:\,\eta\in\mathcal{R}_{(i)}^{k}\}\quad\text{and}\quad\{\{\eta\}:\,\eta\in\mathcal{L}_{(i)}^{k}\}
\]
are irreducible components for $i\in[2,\,\frac{L}{2})$. On the other
hand, if $i=\frac{L}{2}$ then from the analysis in Section \ref{sec7.5.1},
we deduce that each $\{\eta\}$ for $\eta\in\mathcal{R}_{(\frac{L}{2})}^{k}\cup\mathcal{L}_{(\frac{L}{2})}^{k}$
has initial depth $2$, thus is an absorbing element with respect
to $\{\mathfrak{X}^{\star,1}(t)\}_{t\ge0}$.\smallskip{}
\item Finally, according to Figures \ref{Fig7.10}, \ref{Fig7.13} and \ref{Fig7.15},
we may easily check that all the energy-$(\mathbb{H}_{0}+3)$ stable
plateaux are transient elements with respect to $\{\mathfrak{X}^{\star,1}(t)\}_{t\ge0}$,
as starting therein we may always reach $\mathcal{N}^{k}$, in particular
$\bm{\sigma^{k}}$, via a path with depth $1$.\smallskip{}
\end{itemize}
The analysis above readily proves Theorem \ref{thm:hier dec}-(1).

Next, we analyze $\mathscr{P}^{\star,2}$. By part (1), we readily
deduce that
\begin{align*}
\mathscr{P}^{\star,2}= & \bigcup_{k\in\mathbb{T}_{K}}\{\{\bm{\sigma^{k}}\}\}\cup\bigcup_{k\in\mathbb{T}_{K}}\bigcup_{m=2}^{L-2}\{\mathcal{S}_{m}^{k}\}\cup\bigcup_{k\in\mathbb{T}_{K}}\bigcup_{i\in[2,\,\frac{L}{2})}\{\mathcal{R}_{(i)}^{k}\}\cup\bigcup_{k\in\mathbb{T}_{K}}\bigcup_{i\in[2,\,\frac{L}{2})}\{\mathcal{L}_{(i)}^{k}\}\\
 & \cup\bigcup_{k\in\mathbb{T}_{K}}\bigcup_{\eta\in\mathcal{R}_{(\frac{L}{2})}^{k}}\{\{\eta\}\}\cup\bigcup_{k\in\mathbb{T}_{K}}\bigcup_{\eta\in\mathcal{L}_{(\frac{L}{2})}^{k}}\{\{\eta\}\}\quad\text{(if}\quad L\quad\text{is even)}.
\end{align*}
Starting from $\mathcal{S}_{m}^{k}$, we may clearly follow the path
$\omega$ constructed in the proof of Lemma \ref{lem:E barrier UB}
(see also Figure \ref{Fig6.1}) to reach $\mathcal{S}$ with depth
$2$. Thus, it holds that $\Gamma^{\star,2}=2$. Moreover, from $\mathcal{S}$,
the characterization of $\mathcal{N}^{k}$ in Section \ref{sec7.2.2}
indicates that the initial depth $3$ is not enough to reach any other
collection in $\mathscr{P}^{\star,2}$. Thus, the second depth of
each $\{\bm{\sigma^{k}}\}$ is now $4$.

Finally, starting from a collection in $\mathscr{P}^{\star,2}$ other
than $\{\bm{\sigma^{k}}\}$ for some $k\in\mathbb{T}_{K}$, we may
follow a path of depth $2$ to reach $\mathcal{S}$ eventually, which
can be seen by either inspecting Figures \ref{Fig7.10}, \ref{Fig7.13}
and \ref{Fig7.15} or noting that $\overline{\Phi}=\max_{k,\,k'\in\mathbb{T}_{K}}\Phi(\bm{\sigma^{k}},\,\bm{\sigma^{k'}})=\mathbb{H}_{0}+4$.
According to Theorem \ref{thm:P-starh class}-(2), these observations
readily prove Theorem \ref{thm:hier dec}-(2).

Now, by part (2) it holds that
\[
\mathscr{P}^{\star,3}=\{\{\bm{\sigma^{k}}\}:\,k\in\mathbb{T}_{K}\}.
\]
Then, Theorem \ref{thm:E barrier} readily implies that the third
depth of each $\{\bm{\sigma^{k}}\}$ is exactly $4$, thus $\Gamma^{\star,3}=4$.
In addition, thanks to the symmetry and Theorem \ref{thm:P-starh class}-(3),
each $\{\bm{\sigma^{k}}\}$ is a recurrent element with respect to
$\{\mathfrak{X}^{\star,3}(t)\}_{t\ge0}$, thus
\[
\mathscr{P}_{{\rm rec}}^{\star,3}=\mathscr{P}^{\star,3}=\{\{\bm{\sigma^{k}}\}:\,k\in\mathbb{T}_{K}\}.
\]
Thus, this is the terminal level $\mathfrak{m}=3$ and the proof of
Theorem \ref{thm:hier dec} is completed.

Theorem \ref{thm:Kawasaki} is now straightforward from the general
theory developed in Section \ref{sec2} and proved in Sections \ref{sec4}
and \ref{sec5}.

Finally, we prove Theorem \ref{thm:pos rates}. Since $\Gamma^{\star,3}=4$,
by \eqref{eq:Vih def} it holds that for each $\{\bm{\sigma^{k}}\}=\mathcal{P}_{i}^{3}\in\mathscr{P}^{\star,3}$,
\[
\mathcal{V}_{i}^{3}=\{\eta\in\Omega:\,\Phi(\bm{\sigma^{k}},\,\eta)-\mathbb{H}_{0}<4\}=\mathcal{N}^{k}.
\]
Starting from $\mathcal{N}^{k}$, we may visit $\mathcal{P}_{(2),+}^{k}$,
$\mathcal{P}_{2}^{k}$, $\mathcal{P}_{2,-}^{k}$ and then $\mathcal{N}_{2}^{k}$
(cf. Figures \ref{Fig7.10} and \ref{Fig7.15}). Then, we may move
$\mathcal{N}_{m}^{k}\to\mathcal{P}_{m,+}^{k}\to\mathcal{P}_{m+1}^{k}\to\mathcal{P}_{m+1,-}^{k}\to\mathcal{N}_{m+1}^{k}$
for each $m\in[2,\,L-3]$ such that at the end we arrive at $\mathcal{N}_{L-2}^{k}$
(cf. Figure \ref{Fig7.10}). Next, we may move $\mathcal{N}_{L-2}^{k}\to\mathcal{P}_{L-2,+}^{k}\to\mathcal{P}_{L-1}^{k}$
(cf. Figure \ref{Fig7.10}). From $\mathcal{P}_{L-1}^{k}$, we may
directly move $\mathcal{P}_{L-1}^{k}\to\mathcal{P}_{(2),-}^{k+1}\to\mathcal{N}^{k+1}$
to arrive at $\mathcal{N}^{k+1}$. It is straightforward that $\mathfrak{R}^{3}(\cdot,\,\cdot)>0$
along this path from $\mathcal{N}^{k}$ to $\mathcal{N}^{k+1}$, thus
we obtain $\mathfrak{R}^{\star,3}(\bm{\sigma^{k}},\,\bm{\sigma^{k+1}})>0$.

On the other hand, at the last step from $\mathcal{P}_{L-1}^{k}$,
we may instead take a detour as (cf. Figure \ref{Fig7.15})
\[
\mathcal{P}_{L-1}^{k}\to\mathcal{L}_{(2),-}^{k+1}\to\mathcal{G}_{(2),-}^{k+1}\to\widehat{\mathcal{G}}_{(2),-}^{k+1}\to\mathcal{G}_{(2)}^{k+1}\to\widehat{\mathcal{G}}_{(2),+}^{k+1}\to\mathcal{G}_{(2),+}^{k+1}\to\mathcal{L}_{(2),+}^{k+1}\to\mathcal{Q}_{2}^{k+1}.
\]
Using this path, we can avoid visiting $\mathcal{N}^{k+1}$ and follow
a similar path forward to arrive at $\mathcal{P}_{L-1}^{k+1}$ near
$\mathcal{N}^{k+2}$. Thus, concatenating these detour paths, we obtain
that there exists a path from $\mathcal{N}^{k}$ to $\mathcal{N}^{k'}$
for any $k\ne k'$, along which $\mathfrak{R}^{3}(\cdot,\,\cdot)>0$
and avoids all other $\mathcal{N}^{k''}$ for $k''\ne k,\,k'$. Therefore,
we obtain $\mathfrak{R}^{\star,3}(\bm{\sigma^{k}},\,\bm{\sigma^{k'}})>0$
which concludes the proof of Theorem \ref{thm:pos rates}.

\section{\label{sec9}Exit Distribution from Cycles}

In Section \ref{sec9}, we prove Theorem \ref{thm:exit H1}. Since
we are interested only on the behavior of the dynamics starting from
$\mathcal{C}$ until it hits its boundary $\partial\mathcal{C}$,
without loss of generality, we may assume in this section that 
\begin{equation}
\Omega=\mathcal{C}\cup\partial\mathcal{C}\quad\text{and}\quad\xi\nsim\xi'\quad\text{for all}\quad\xi,\,\xi'\in\partial\mathcal{C}.\label{eq:sec9 assump}
\end{equation}
For each $f:\,\Omega\to\mathbb{R}$, denote by $\mathscr{D}_{\beta}(f)$
its \emph{Dirichlet form} defined as
\[
\mathscr{D}_{\beta}(f):=\langle f,\,-L_{\beta}f\rangle_{\mu_{\beta}}=\frac{1}{2}\sum_{\eta,\,\xi\in\Omega}\mu_{\beta}(\eta)r_{\beta}(\eta,\,\xi)(f(\xi)-f(\eta))^{2},
\]
where $\langle\cdot,\,\cdot\rangle_{\mu_{\beta}}$ denotes the inner
product with respect to $\mu_{\beta}$. For two disjoint nonempty
subsets $\mathcal{A},\,\mathcal{B}\subseteq\Omega$, denote by $\mathfrak{h}_{\mathcal{A},\mathcal{B}}:\,\Omega\to\mathbb{R}$
the \emph{equilibrium potential} between $\mathcal{A}$ and $\mathcal{B}$
given as $\mathfrak{h}_{\mathcal{A},\mathcal{B}}(\eta):=\mathbb{P}_{\eta}[\mathcal{T}_{\mathcal{A}}<\mathcal{T}_{\mathcal{B}}]$.
Here, $\mathfrak{h}_{\mathcal{A},\mathcal{B}}$ solves the following
Dirichlet problem:
\[
\mathfrak{h}_{\mathcal{A},\mathcal{B}}={\bf 1}_{\mathcal{A}}\quad\text{in}\quad\mathcal{A}\cup\mathcal{B}\quad\text{and}\quad L_{\beta}\mathfrak{h}_{\mathcal{A},\mathcal{B}}=0\quad\text{in}\quad(\mathcal{A}\cup\mathcal{B})^{c}.
\]
Then, the \emph{capacity} ${\rm cap}_{\beta}(\mathcal{A},\,\mathcal{B})$
between $\mathcal{A}$ and $\mathcal{B}$ is defined as ${\rm cap}_{\beta}(\mathcal{A},\,\mathcal{B}):=\mathscr{D}_{\beta}(\mathfrak{h}_{\mathcal{A},\mathcal{B}})$.
For simplicity, we write
\begin{equation}
\mathfrak{a}(\xi):=\sum_{\eta\in\mathcal{C}}{\bf 1}\{\eta\sim\xi\}\quad\text{for each}\quad\xi\in\partial^{\star}\mathcal{C}.\label{eq:axi def}
\end{equation}

\begin{lem}
\label{lem:9.1}For every $\eta_{0}\in\mathcal{C}$, it holds that
\[
{\rm cap}_{\beta}(\eta_{0},\,\partial\mathcal{C})=\frac{1+o(1)}{Z_{\beta}}\cdot\Big(\sum_{\xi\in\partial^{\star}\mathcal{C}}\mathfrak{a}(\xi)\Big)\cdot e^{-\beta\mathbb{H}(\partial^{\star}\mathcal{C})}.
\]
\end{lem}

\begin{proof}
Define a test function $F_{0}:\,\Omega\to\mathbb{R}$ as $F_{0}:={\bf 1}_{\mathcal{C}}$.
Then, we calculate
\[
\mathscr{D}_{\beta}(F_{0})=\frac{1}{2}\sum_{\eta,\,\xi\in\Omega}\mu_{\beta}(\eta)r_{\beta}(\eta,\,\xi)(F_{0}(\xi)-F_{0}(\eta))^{2}=\sum_{\eta\in\mathcal{C}}\sum_{\xi\in\partial\mathcal{C}}\mu_{\beta}(\eta)r_{\beta}(\eta,\,\xi).
\]
By \eqref{eq:Gibbs def} and \eqref{eq:det bal}, we may rewrite this
as
\[
\mathscr{D}_{\beta}(F_{0})=\sum_{\eta\in\mathcal{C}}\sum_{\xi\in\partial\mathcal{C}:\,\eta\sim\xi}\mu_{\beta}(\xi)=\sum_{\eta\in\mathcal{C}}\sum_{\xi\in\partial\mathcal{C}:\,\eta\sim\xi}\frac{e^{-\beta\mathbb{H}(\xi)}}{Z_{\beta}}.
\]
Since $\mathcal{F}(\partial\mathcal{C})=\partial^{\star}\mathcal{C}$,
we conclude that (cf. \eqref{eq:axi def})
\begin{equation}
\mathscr{D}_{\beta}(F_{0})=\frac{(1+o(1))e^{-\beta\mathbb{H}(\partial^{\star}\mathcal{C})}}{Z_{\beta}}\cdot\sum_{\xi\in\partial^{\star}\mathcal{C}}\mathfrak{a}(\xi).\label{eq:9.1 pf1}
\end{equation}
Moreover, a well-known renewal estimate (e.g., see \cite[Lemma 7.8]{KS IsingPotts-2D})
implies that
\begin{equation}
\sup_{\eta\in\mathcal{C}}|\mathfrak{h}_{\eta_{0},\partial\mathcal{C}}(\eta)-1|=o(1).\label{eq:eq pot flat}
\end{equation}
Thus, we may calculate
\begin{align*}
\mathscr{D}_{\beta}(F_{0}-\mathfrak{h}_{\eta_{0},\partial\mathcal{C}}) & =\langle F_{0}-\mathfrak{h}_{\eta_{0},\partial\mathcal{C}},\,-L_{\beta}(F_{0}-\mathfrak{h}_{\eta_{0},\partial\mathcal{C}})\rangle_{\mu_{\beta}}\\
 & =\mathscr{D}_{\beta}(F_{0})+\mathscr{D}_{\beta}(\mathfrak{h}_{\eta_{0},\partial\mathcal{C}})-\langle F_{0},\,-L_{\beta}\mathfrak{h}_{\eta_{0},\partial\mathcal{C}}\rangle_{\mu_{\beta}}-\langle\mathfrak{h}_{\eta_{0},\partial\mathcal{C}},\,-L_{\beta}F_{0}\rangle_{\mu_{\beta}}.
\end{align*}
Since $F_{0}=\mathfrak{h}_{\eta_{0},\partial\mathcal{C}}={\bf 1}_{\{\eta_{0}\}}$
in $\{\eta_{0}\}\cup\partial\mathcal{C}$, we have
\[
\langle F_{0},\,-L_{\beta}\mathfrak{h}_{\eta_{0},\partial\mathcal{C}}\rangle_{\mu_{\beta}}=\langle\mathfrak{h}_{\eta_{0},\partial\mathcal{C}},\,-L_{\beta}\mathfrak{h}_{\eta_{0},\partial\mathcal{C}}\rangle_{\mu_{\beta}}=\mathscr{D}_{\beta}(\mathfrak{h}_{\eta_{0},\partial\mathcal{C}}),
\]
thus
\begin{equation}
\mathscr{D}_{\beta}(F_{0}-\mathfrak{h}_{\eta_{0},\partial\mathcal{C}})=\mathscr{D}_{\beta}(F_{0})-\langle\mathfrak{h}_{\eta_{0},\partial\mathcal{C}},\,-L_{\beta}F_{0}\rangle_{\mu_{\beta}}.\label{eq:9.1 pf2}
\end{equation}
Moreover, we calculate the inner product as
\[
\langle\mathfrak{h}_{\eta_{0},\partial\mathcal{C}},\,-L_{\beta}F_{0}\rangle_{\mu_{\beta}}=\sum_{\eta\in\mathcal{C}}\mathfrak{h}_{\eta_{0},\partial\mathcal{C}}(\eta)\sum_{\xi\in\partial\mathcal{C}}\mu_{\beta}(\eta)r_{\beta}(\eta,\,\xi).
\]
By \eqref{eq:eq pot flat}, \eqref{eq:Gibbs def} and \eqref{eq:det bal},
this becomes
\begin{equation}
(1+o(1))\cdot\frac{e^{-\beta\mathbb{H}(\partial^{\star}\mathcal{C})}}{Z_{\beta}}\cdot\sum_{\xi\in\partial^{\star}\mathcal{C}}\mathfrak{a}(\xi).\label{eq:9.1 pf3}
\end{equation}
By \eqref{eq:9.1 pf1}, \eqref{eq:9.1 pf2} and \eqref{eq:9.1 pf3},
we obtain that
\[
\mathscr{D}_{\beta}(F_{0}-\mathfrak{h}_{\eta_{0},\partial\mathcal{C}})=o(1)\cdot\frac{e^{-\beta\mathbb{H}(\partial^{\star}\mathcal{C})}}{Z_{\beta}}\cdot\sum_{\xi\in\partial^{\star}\mathcal{C}}\mathfrak{a}(\xi)=o(\mathscr{D}_{\beta}(F_{0})).
\]
Therefore, by the $H^{1}$ computation in \cite[Proof of Theorem 4.2]{KS IsingPotts-2D},
we conclude that
\[
{\rm cap}_{\beta}(\eta_{0},\,\partial\mathcal{C})=(1+o(1))\cdot\mathscr{D}_{\beta}(F_{0})=(1+o(1))\cdot\frac{e^{-\beta\mathbb{H}(\partial^{\star}\mathcal{C})}}{Z_{\beta}}\cdot\sum_{\xi\in\partial^{\star}\mathcal{C}}\mathfrak{a}(\xi).
\]
This concludes the proof.
\end{proof}
\begin{lem}
\label{lem:9.2}For $\xi_{0}\in\partial^{\star}\mathcal{C}$, it holds
that
\[
{\rm cap}_{\beta}(\xi_{0},\,\partial\mathcal{C}\setminus\{\xi_{0}\})=\frac{1+o(1)}{Z_{\beta}}\cdot\Big(\frac{\mathfrak{a}(\xi_{0})\cdot\sum_{\xi\in\partial^{\star}\mathcal{C}\setminus\{\xi_{0}\}}\mathfrak{a}(\xi)}{\sum_{\xi\in\partial^{\star}\mathcal{C}}\mathfrak{a}(\xi)}\Big)\cdot e^{-\beta\mathbb{H}(\partial^{\star}\mathcal{C})}.
\]
\end{lem}

\begin{proof}
Similarly, define a test function $G_{0}:\,\Omega\to\mathbb{R}$ as
$G_{0}(\xi_{0}):=1$, $G_{0}:=0$ in $\partial\mathcal{C}\setminus\{\xi_{0}\}$
and
\[
G_{0}(\eta):=\frac{\mathfrak{a}(\xi_{0})}{\sum_{\xi\in\partial^{\star}\mathcal{C}}\mathfrak{a}(\xi)}\quad\text{for all}\quad\eta\in\mathcal{C}.
\]
Then, according to \eqref{eq:sec9 assump}, we calculate $\mathscr{D}_{\beta}(G_{0})$
as
\begin{equation}
\begin{aligned} & \sum_{\eta\in\mathcal{C}}\sum_{\xi\in\partial\mathcal{C}}\mu_{\beta}(\eta)r_{\beta}(\eta,\,\xi)(G_{0}(\xi)-G_{0}(\eta))^{2}\\
 & \simeq\frac{e^{-\beta\mathbb{H}(\partial^{\star}\mathcal{C})}}{Z_{\beta}}\cdot\Big(\sum_{\eta\in\mathcal{C}}\sum_{\xi\in\partial^{\star}\mathcal{C}\setminus\{\xi_{0}\}:\,\xi\sim\eta}\Big(\frac{\mathfrak{a}(\xi_{0})}{\sum_{\xi\in\partial^{\star}\mathcal{C}}\mathfrak{a}(\xi)}\Big)^{2}+\sum_{\eta\in\mathcal{C}:\,\eta\sim\xi_{0}}\Big(1-\frac{\mathfrak{a}(\xi_{0})}{\sum_{\xi\in\partial^{\star}\mathcal{C}}\mathfrak{a}(\xi)}\Big)^{2}\Big)\\
 & =\frac{e^{-\beta\mathbb{H}(\partial^{\star}\mathcal{C})}\cdot\mathfrak{a}(\xi_{0})\cdot(\sum_{\xi\in\partial^{\star}\mathcal{C}\setminus\{\xi_{0}\}}\mathfrak{a}(\xi))}{Z_{\beta}\sum_{\xi\in\partial^{\star}\mathcal{C}}\mathfrak{a}(\xi)}.
\end{aligned}
\label{eq:9.2 pf1}
\end{equation}
In addition, another renewal estimate verifies that
\begin{equation}
\sup_{\eta,\,\eta'\in\mathcal{C}}|\mathfrak{h}_{\xi_{0},\partial\mathcal{C}\setminus\{\xi_{0}\}}(\eta)-\mathfrak{h}_{\xi_{0},\partial\mathcal{C}\setminus\{\xi_{0}\}}(\eta')|=o(1).\label{eq:eq pot flat 2}
\end{equation}
Thus, as in the proof of Lemma \ref{lem:9.1}, we calculate
\begin{equation}
\mathscr{D}_{\beta}(G_{0}-\mathfrak{h}_{\xi_{0},\partial\mathcal{C}\setminus\{\xi_{0}\}})=\mathscr{D}_{\beta}(G_{0})-\langle\mathfrak{h}_{\xi_{0},\partial\mathcal{C}\setminus\{\xi_{0}\}},\,-L_{\beta}G_{0}\rangle_{\mu_{\beta}},\label{eq:9.2 pf2}
\end{equation}
where $\langle\mathfrak{h}_{\xi_{0},\partial\mathcal{C}\setminus\{\xi_{0}\}},\,-L_{\beta}G_{0}\rangle_{\mu_{\beta}}$
now becomes
\begin{equation}
\begin{aligned} & \mu_{\beta}(\xi_{0})\sum_{\eta\in\mathcal{C}}r_{\beta}(\xi_{0},\,\eta)(G_{0}(\xi_{0})-G_{0}(\eta))\\
 & +\sum_{\eta\in\mathcal{C}}\mu_{\beta}(\eta)\mathfrak{h}_{\xi_{0},\partial\mathcal{C}\setminus\{\xi_{0}\}}(\eta)\sum_{\xi\in\partial\mathcal{C}}r_{\beta}(\eta,\,\xi)(G_{0}(\eta)-G_{0}(\xi)).
\end{aligned}
\label{eq:9.2 pf3}
\end{equation}
By \eqref{eq:det bal}, the first part of \eqref{eq:9.2 pf3} becomes
\begin{equation}
\sum_{\eta\in\mathcal{C}}\mu_{\beta}(\xi_{0}){\bf 1}\{\eta\sim\xi_{0}\}\cdot\Big(1-\frac{\mathfrak{a}(\xi_{0})}{\sum_{\xi\in\partial^{\star}\mathcal{C}}\mathfrak{a}(\xi)}\Big)=\frac{e^{-\beta\mathbb{H}(\partial^{\star}\mathcal{C})}}{Z_{\beta}}\cdot\frac{\mathfrak{a}(\xi_{0})\cdot(\sum_{\xi\in\partial^{\star}\mathcal{C}\setminus\{\xi_{0}\}}\mathfrak{a}(\xi))}{\sum_{\xi\in\partial^{\star}\mathcal{C}}\mathfrak{a}(\xi)}.\label{eq:9.2 pf4}
\end{equation}
By \eqref{eq:eq pot flat 2} and \eqref{eq:det bal}, the second part
of \eqref{eq:9.2 pf3} becomes
\[
(\mathfrak{h}_{\xi_{0},\partial\mathcal{C}\setminus\{\xi_{0}\}}(\eta_{0})+o(1))\cdot\sum_{\eta\in\mathcal{C}}\sum_{\xi\in\partial\mathcal{C}}\mu_{\beta}(\xi){\bf 1}\{\eta\sim\xi\}(G_{0}(\eta)-G_{0}(\xi)),
\]
where $\eta_{0}$ is arbitrarily chosen in $\mathcal{C}$. By the
definition of $G_{0}$, the double summation can be rewritten as
\begin{equation}
o(\mu_{\beta}(\partial^{\star}\mathcal{C}))+\mu_{\beta}(\xi_{0})\cdot\Big(\frac{\mathfrak{a}(\xi_{0})\sum_{\xi\in\partial^{\star}\mathcal{C}\setminus\{\xi_{0}\}}\mathfrak{a}(\xi)}{\sum_{\xi\in\partial^{\star}\mathcal{C}}\mathfrak{a}(\xi)}-\frac{\mathfrak{a}(\xi_{0})\sum_{\xi\in\partial^{\star}\mathcal{C}\setminus\{\xi_{0}\}}\mathfrak{a}(\xi)}{\sum_{\xi\in\partial^{\star}\mathcal{C}}\mathfrak{a}(\xi)}\Big)=o(\mu_{\beta}(\partial^{\star}\mathcal{C})).\label{eq:9.2 pf5}
\end{equation}
Therefore, collecting \eqref{eq:9.2 pf1}, \eqref{eq:9.2 pf2}, \eqref{eq:9.2 pf3},
\eqref{eq:9.2 pf4} and \eqref{eq:9.2 pf5}, we conclude that
\begin{equation}
\mathscr{D}_{\beta}(G_{0}-\mathfrak{h}_{\xi_{0},\partial\mathcal{C}\setminus\{\xi_{0}\}})=o(\mathscr{D}_{\beta}(G_{0})),\label{eq:9.2 pf6}
\end{equation}
thus
\[
{\rm cap}_{\beta}(\xi_{0},\,\partial\mathcal{C}\setminus\{\xi_{0}\})=\frac{(1+o(1))e^{-\beta\mathbb{H}(\partial^{\star}\mathcal{C})}\cdot\mathfrak{a}(\xi_{0})\cdot(\sum_{\xi\in\partial^{\star}\mathcal{C}\setminus\{\xi_{0}\}}\mathfrak{a}(\xi))}{Z_{\beta}\sum_{\xi\in\partial^{\star}\mathcal{C}}\mathfrak{a}(\xi)},
\]
as wanted.
\end{proof}
\begin{proof}[Proof of Theorem \ref{thm:exit H1}]
 Recall the test function $G_{0}$ from the proof of Lemma \ref{lem:9.2}
and consider a new function $\delta_{0}:=G_{0}-\mathfrak{h}_{\xi_{0},\partial\mathcal{C}\setminus\{\xi_{0}\}}$.
Then, $\delta_{0}\equiv0$ in $\partial\mathcal{C}$, thus by the
well-known \emph{Dirichlet principle} (cf. \cite[(16.2.2)]{BdH}),
it holds that
\[
\delta_{0}(\eta_{0})^{2}\cdot{\rm cap}_{\beta}(\eta_{0},\,\partial\mathcal{C})\le\mathscr{D}_{\beta}(\delta_{0}).
\]
By \eqref{eq:9.2 pf6}, the right-hand side equals
\[
\mathscr{D}_{\beta}(G_{0}-\mathfrak{h}_{\xi_{0},\partial\mathcal{C}\setminus\{\xi_{0}\}})=o(\mu_{\beta}(\partial^{\star}\mathcal{C})).
\]
On the other hand, by Lemma \ref{lem:9.1}, the left-hand side equals
\[
\Big(\frac{\mathfrak{a}(\xi_{0})}{\sum_{\xi\in\partial^{\star}\mathcal{C}}\mathfrak{a}(\xi)}-\mathfrak{h}_{\xi_{0},\partial\mathcal{C}\setminus\{\xi_{0}\}}(\eta_{0})\Big)^{2}\cdot\frac{1+o(1)}{|\partial^{\star}\mathcal{C}|}\cdot\Big(\sum_{\xi\in\partial^{\star}\mathcal{C}}\mathfrak{a}(\xi)\Big)\cdot\mu_{\beta}(\partial^{\star}\mathcal{C}).
\]
Therefore, we conclude that
\[
\frac{\mathfrak{a}(\xi_{0})}{\sum_{\xi\in\partial^{\star}\mathcal{C}}\mathfrak{a}(\xi)}-\mathfrak{h}_{\xi_{0},\partial\mathcal{C}\setminus\{\xi_{0}\}}(\eta_{0})=o(1),
\]
which is exactly Theorem \ref{thm:exit H1}.
\end{proof}

\appendix

\section{\label{appenA}Proof of Lemma \ref{lem:nbd dec}}

In this appendix, we prove Lemma \ref{lem:nbd dec}. We refer the
readers to \cite[Appendix A]{KS IsingPotts-2D} for a proof of part
(1) and focus on proving part (2). For disjoint $\mathcal{A}$ and
$\mathcal{B}$, it is clear that $\widehat{\mathcal{N}}(\mathcal{A};\mathcal{B})$
and $\mathcal{B}$ are disjoint. Applying part (1) for these two sets,
we obtain that
\[
\widehat{\mathcal{N}}\big(\widehat{\mathcal{N}}(\mathcal{A};\mathcal{B})\cup\mathcal{B}\big)=\widehat{\mathcal{N}}(\widehat{\mathcal{N}}(\mathcal{A};\mathcal{B});\mathcal{B})\cup\widehat{\mathcal{N}}(\mathcal{B};\widehat{\mathcal{N}}(\mathcal{A};\mathcal{B})).
\]
Thus, the proof of part (2) is completed if the following two displayed
identities are verified:
\[
\widehat{\mathcal{N}}(\widehat{\mathcal{N}}(\mathcal{A};\mathcal{B})\cup\mathcal{B})=\widehat{\mathcal{N}}(\mathcal{A}\cup\mathcal{B})\quad\text{and}\quad\widehat{\mathcal{N}}(\widehat{\mathcal{N}}(\mathcal{A};\mathcal{B});\mathcal{B})=\widehat{\mathcal{N}}(\mathcal{A};\mathcal{B}).
\]
We divide this verification into the following four parts.\smallskip{}

\begin{itemize}
\item $\widehat{\mathcal{N}}(\widehat{\mathcal{N}}(\mathcal{A};\mathcal{B})\cup\mathcal{B})\supseteq\widehat{\mathcal{N}}(\mathcal{A}\cup\mathcal{B})$:
Take $\eta\in\widehat{\mathcal{N}}(\mathcal{A}\cup\mathcal{B})$.
Then, there exists an allowed path $\omega:\,\xi\to\eta$ with $\xi\in\mathcal{A}\cup\mathcal{B}$.
If $\xi\in\mathcal{B}$ then clearly $\eta\in\widehat{\mathcal{N}}(\mathcal{B})$.
If $\xi\in\mathcal{A}$ then since $\omega$ is an allowed path, $\mathbb{H}(\xi)\le\mathbb{H}_{0}+4$,
thus obviously $\xi\in\widehat{\mathcal{N}}(\mathcal{A};\mathcal{B})$.
This implies that $\eta\in\widehat{\mathcal{N}}(\widehat{\mathcal{N}}(\mathcal{A};\mathcal{B}))$.\smallskip{}
\item $\widehat{\mathcal{N}}(\widehat{\mathcal{N}}(\mathcal{A};\mathcal{B})\cup\mathcal{B})\subseteq\widehat{\mathcal{N}}(\mathcal{A}\cup\mathcal{B})$:
Take $\eta\in\widehat{\mathcal{N}}(\widehat{\mathcal{N}}(\mathcal{A};\mathcal{B})\cup\mathcal{B})$
so that there exists an allowed path $\omega:\,\xi\to\eta$ with $\xi\in\widehat{\mathcal{N}}(\mathcal{A};\mathcal{B})\cup\mathcal{B}$.
If $\xi\in\mathcal{B}$ then clearly $\eta\in\widehat{\mathcal{N}}(\mathcal{B})$.
If $\xi\in\widehat{\mathcal{N}}(\mathcal{A};\mathcal{B})$ then there
exists another allowed path $\omega':\,\zeta\to\xi$ with $\zeta\in\mathcal{A}$
and $\omega'\cap\mathcal{B}=\emptyset$. Concatenating $\omega'$
and $\omega$, we obtain an allowed path from $\zeta$ to $\eta$,
which implies that $\eta\in\widehat{\mathcal{N}}(\mathcal{A})$.\smallskip{}
\item $\widehat{\mathcal{N}}(\widehat{\mathcal{N}}(\mathcal{A};\mathcal{B});\mathcal{B})\supseteq\widehat{\mathcal{N}}(\mathcal{A};\mathcal{B})$:
Take $\eta\in\widehat{\mathcal{N}}(\mathcal{A};\mathcal{B})$ so that
there exists an allowed path $\omega:\,\xi\to\eta$ with $\xi\in\mathcal{A}$
and $\omega\cap\mathcal{B}=\emptyset$. Since $\omega$ is an allowed
path, $\mathbb{H}(\xi)\le\mathbb{H}_{0}+4$, thus obviously $\xi\in\widehat{\mathcal{N}}(\mathcal{A};\mathcal{B})$.
This implies that $\eta\in\widehat{\mathcal{N}}(\widehat{\mathcal{N}}(\mathcal{A};\mathcal{B});\mathcal{B})$.\smallskip{}
\item $\widehat{\mathcal{N}}(\widehat{\mathcal{N}}(\mathcal{A};\mathcal{B});\mathcal{B})\subseteq\widehat{\mathcal{N}}(\mathcal{A};\mathcal{B})$:
Take $\eta\in\widehat{\mathcal{N}}(\widehat{\mathcal{N}}(\mathcal{A};\mathcal{B});\mathcal{B})$
so that there exists an allowed path $\omega:\,\xi\to\eta$ with $\xi\in\widehat{\mathcal{N}}(\mathcal{A};\mathcal{B})$
and $\omega\cap\mathcal{B}=\emptyset$. Then, there exists another
allowed path $\omega':\,\zeta\to\xi$ with $\zeta\in\mathcal{A}$
and $\omega'\cap\mathcal{B}=\emptyset$. Concatenating $\omega'$
and $\omega$, we obtain an allowed path from $\zeta$ to $\eta$
avoiding $\mathcal{B}$. This deduces that $\eta\in\widehat{\mathcal{N}}(\mathcal{A};\mathcal{B})$.\smallskip{}
\end{itemize}
Combining these four observations, we conclude the proof of part (2)
of Lemma \ref{lem:nbd dec}.
\begin{acknowledgement*}
SK was supported by KIAS Individual Grant (HP095101) at the Korea
Institute for Advanced Study. SK thanks Jungkyoung Lee for his valuable
comments and advices that helped to clarify the presentation of this
article.
\end{acknowledgement*}

\end{document}